\pdfoutput=1
\RequirePackage{fix-cm}
\documentclass[smallcondensed]{svjour3}     % onecolumn (ditto)
\smartqed  % flush right qed marks, e.g. at end of proof
\usepackage[utf8]{luainputenc}
\usepackage[T1]{fontenc}
\usepackage{uniinput}
\usepackage{graphicx}
\usepackage{hyperref}
\usepackage{mathptmx}      % use Times fonts if available on your TeX system
\newcommand\svnid[1]{}
\usepackage{algebra}
\usepackage{tsabk}
\usepackage[natbib=true,sorting=none,style=numeric]{biblatex}
\usepackage{mathtools}
\usepackage{amsfonts}
\usepackage{tikz-cd}
\usepackage{musixnfss}
\newcommand*\defindex[2][]{\emph{#2}}
\journalname{to be submtited}

\smartqed

\begin{document}

\title{Factorisations of some partially ordered sets and small categories%
  % \thanks{Grants or other notes
}


\author{Tobias Schlemmer}


\institute{Tobias Schlemmer \at
  TU Dresden \\
  \email{Tobias.Schlemmer@tu-dresden.de}           %  \\
  % \emph{Present address:} of F. Author  %  if needed
}

\date{Received: date / Accepted: date}

\maketitle

\begin{abstract}
  Orbits of automorphism groups of partially ordered sets are not
  necessarily congruence classes, i.e. images of an order
  homomorphism. Based on so-called orbit categories a framework of
  factorisations and unfoldings is developed that preserves the
  antisymmetry of the order Relation. Finally some suggestions are
  given, how the orbit categories can be represented by simple
  directed and annotated graphs and annotated binary relations. These
  relations are reflexive, and, in many cases, they can be chosen to
  be antisymmetric. From these constructions arise different
  suggestions for fundamental systems of partially ordered sets and
  reconstruction data which are illustrated by examples from
  mathematical music theory.

\end{abstract}

\section{Introduction}

In general, the orbits of automorphism groups of partially ordered
sets cannot be considered as equivalence classes of a convenient
congruence relation of the corresponding partial orders. If the orbits
are not convex with respect to the order relation, the factor relation
of the partial order is not necessarily a partial order. However, when
we consider a partial order as a directed graph, the direction of the
arrows is preserved during the factorisation in many cases, while the
factor graph of a simple graph is not necessarily simple. Even if the
factor relation can be used to anchor unfolding information
\cite{Borchmann:2009}, this structure is usually not visible as a
relation.

According to the common mathematical usage a fundamental system is a
structure that generates another (larger structure) according to some
given rules. In this article we will discuss several suggestions for
such fundamental systems. For example, orbit categories and
transversal categories can be considered as fundamental systems.

However, ordered sets are considered as relational structures,
which is not always possible for the above mentioned categories. When
we map them into simple categories we often lose the direction
information. This can sometimes be avoided when we throw away
redundant information that can be regained during the reconstruction
process.

For ordered sets local orders come into mind as possible fundamental
systems. The usual definition defines a local order as a binary
relation that is an order relation in every environment of some set
covering. However, there is no natural covering for an ordered
set. This means that any reflexive and antisymmetric relation can be
considered as a local order. In the same way as order relations
generalise to categories, local orders generalise to partial
categories.

The current work is inspired by the objects of Mathematical Music
Theory. Some of them (e.g.\ pitch and time intervals) have a natural
notion of order \cite{NeumaierWille1990,NeumaierWille1989}. While
pitch interval groups are usually commutative the time information is
often equivalent to the affine group of the straight line
\cite{davi2007generalized}. In both cases some kind of periodicity
plays an important role. The most prominent one for interval systems
is the octave identification and the most obvious one in the time
dimension is the metric periodicity. With the Shepard tones and the
forward/backward directions we have aspects of the order relations
that can still be discussed using the factor structures even though
the usual factorisation of relations contradicts this fact
\cite{shepard:2346}.

It is a well known fact that the factor groups of po-groups with
respect to convex normal subgroups are again po-groups. However, the
above mentioned examples do not belong to this class.

Orbit factorisations of linearly ordered sets can be described using
cyclic orders~\cite{0146.41805}. However, this fails in case of higher
order dimension. This can be seen when we factor
$\tupel{\mathbb Z,\leq}\times\tupel{\mathbb Z,\leq}$ by its subgroup
$\tupel{4\mathbb Z,\leq}\times\tupel{3\mathbb Z,\leq}$. In that case
we have the two paths
\begin{align*}
  \paar{0}{0} &\to \paar{0}{1} \to \paar11 \to \paar{2}{1}\to\paar
                31\text{ and}\\
  \paar{0}{0} &\to \paar 10 \to \paar{2}{0} \to \paar{3}{0}\to\paar{3}{1}\to\paar{0}{1}.
\end{align*}
Conseqeuntly, the factored cyclic order relation contains the two triplets
\begin{align*}
  \bigl(\paar{0}{0}, \paar{0}{1}, \paar 31\bigr)\text{ and }
  \bigl(\paar{0}{0}, \paar{3}{1}, \paar{0}{1}\bigr).
\end{align*}
This violates the antisymmetry condition of cyclic orders and
cyclically ordered groups.

In the previous case we still get a result when we consider the
neighbourhood relation. Its factorisation leads again to a similar
relation. So it is a good candidate when it comes to positive facts
about a proper definition of a factor relation.

A drawback of this approach is that it does not work with orders
containing infinite closed intervals like the product order
$\paar{\mathbb Q}{\leq}\times\paar{\mathbb Q}{\leq}$.

In the approach of the current article during factorisation as much
information is to be preserved as possible. Thus, the order relation
is considered as a category. This also allows chaining of
factorisations.

For non-commutative po-groups there is a difference between the group
acting on itself via left and via right multiplication. For example
David Lewin~\cite{davi2007generalized} distinguishes between both operations
as “transpositions” and “interval preserving functions”, such that
they get different interpretations. This suggests that either of these
actions can be used independently of the other to describe the
factorisations. Thus, it suffices to consider the factorisations of
right po-groups. This simplifies the construction a lot as we can
restrict ourselves to semi-regular group actions on ordered sets.

After a preliminary section (Section~\ref{priliminaries}), in
Section~\ref{factorisation} we will introduce orbiit categories that
model factored po-groups. This will be enriched in
Section~\ref{representation} with additional information that is used
to define and study basic properties of an unfolding operation based
on group extensions. The actual reconstruction and its isomorphy to
the original po-group is considered in
Section~\ref{reconstruction}. The following three sections describe
modifications of the previous results which use the isomorphic
category structure of the orbits. In
Section~\ref{sec:flat-orbit-categ} conditions are studied that allow
to factor out the vertex monoids from the orbit categories. These so
called “flat orbit categories” are used to define corresponding flat
category representations and their unfolding in
Section~\ref{sec:flat-categ-repr}. In this section it is also proved
that unfoldings of category representations and flat category
representations are isomorphic. Finally, in
Section~\ref{sec:flat-representations} settings are discussed which
allow to use partial categories as part of the representations. A list
of additional properties of flat representations is introduced,
including conditions under which the partial categories can be
antisymmetric graphs. The article closes with some remarks about
motivations and applications in mathematical music theory in
Section~\ref{sec:furth-results-appl} and some outlook in
Section~\ref{sec:furth-rese-topics}

\section{Preliminaries}\label{priliminaries}
Throughout this paper we consider an order relation to be a symmetric,
reflexive and transitive relation. When it is not explicitly stated an
order relation needs not to be linearly ordered. A partially ordered
group (po-group) $\OrdGruppe{\gruppe}{\cdot}{\multinverses}{1}{\leq}$
is a group with an order relation that for any $x,y,a,b∈\gruppe$
fulfils the law $a\leq b$ iff $xay\leq xby$. In the same way the
wording “right po-group” is a group that fulfils the one sided
relation $a\leq b$ iff $ax \leq bx$. An element $a∈\gruppe$ is called
positive iff $1\leq a$ holds. The set of all positive elements
$\poskegel {\gruppe}\definiert \Menge{a∈\gruppe}{1\leq a}$ is called
positive cone. Likewise an element $a∈\gruppe$ is called negative iff
$a\leq 1$ holds.  The symbol $\normalteilervon$ denotes a normal
subgroup, while an ordinary subgroup is denoted by the order relation
of groups $\leq$. A linearly ordered set is called a chain.

A left associative group action of a group $\gruppe$ is a mapping
$\wabbildung{\gruppe\times M\to M}{\paar ga}{g^a}$ such that the
equations $a^1 =a$ and $(a^g)^h=a^{(gh)}$ hold. The orbits are denoted
by $a^{\gruppe}\definiert\Menge{a^g}{g∈\gruppe}$.  In the same way, a
mapping $\wabbildung{\gruppe\times M\to M}{\paar ga}{{}^ag}$ with
${}^1a =a$ and ${}^h({}^ga)={}^{(gh)}a$ is called right associative
group action. The corresponding orbits are denoted by ${}^{\gruppe}a$.
In the same way a right associative group action is denoted by
${}^ga$ and the right associative orbit with ${}^{\gruppe}a$,
respectively. The group acts semi-regular if for any set elements $a$
and any group elements $g$ the equation $a=a^g$ implies $g=1$. It acts
transitive $a^{\gruppe} = M$. The action is called regular if it is
semi-regular and transitive.

We will consider a category $\kategorie C$ as a (not necessarily
simple) directed graph with vertex set $\Ob\kategorie C$ and edge set
$\Morall{\kategorie C}$.\footnote{Thus, in this paper all categories
  are small categories.} We refer to the edges as arrows. The set of
arrows between two vertices $x,y∈\Ob\kategorie C$ is denoted by
$\Mor{\kategorie C}xy$. The start and end vertex of an arrow
$\klmorphismus a$ are denoted by $\moranfang a$ and $\morende a$.  For
the concatenation $*$ of a category $\kategorie C$ a covariant
notation has been chosen, leading to
$\moranfang( \klmorphismus a*\klmorphismus b) = \moranfang
\klmorphismus a$ and
$\morende (\klmorphismus a*\klmorphismus b) =\morende \klmorphismus
b$. The identity loop of a vertex $x$ is denoted by $\id x$.

Given a relation $ϱ\subseteq M\times N$, a mapping
$\abbildung{φ}{M\times N}{M'\times N'}$ is called homomorphism into a
relation $ϱ'\subseteq M'\times N'$, if for all elements $x∈M$ and
$y∈N$ the equivalence $x\mathrel{ϱ} y⇒φx\mathrel{ϱ'} φy$ holds. If
furthermore $φ$ is bijective and $x\mathrel{ϱ} y⇔φx\mathrel{ϱ'} φy$
holds, $φ$ is called an isomorphism. Considering two products sets
$M_1\times M_2 \times \ldots \times M_n$ and
$(\cdots(M_1\times M_2)\times \cdots \times M_{n-1} )\times M_n)$ as
equal we can extend these definitions to arbitrary relations. In this
setting each $\abbildung fMN$ is represented as a pair of relations
$ϱ_f$ and $ϱ_{\bar f}$ with $f(x) = y ⇔ x\mathrel{ϱ}_f y$ and
$f(x) \neq y ⇔ x\mathrel{ϱ}_{\bar f} y$. It is easy to see, that
homomorphisms and isomorphisms of such pairs of relations fulfil the
properties of the usual homomorphisms and isomorphisms of
operations. Given two sets $M,N$, a set $R$ of relations on $M$ and a
set $S$ of a relations on $N$ together with an injective mapping
$\abbildung{Ψ}{R}{S}$ a mapping $\abbildung{φ}{M}{N}$ is called a
homomorphism (isomorphism) from $\paar{M}{R}$ to
$\paar{N}{S}$, iff for each relation $ϱ∈R$ the mapping
$φ$ is a homomorphism (isomorphism) from $ϱ$ into
$Ψ(ϱ)$. An automorphism is an isomorphism from $\paar MR$ into itself. The
set of automorphisms of a certain structure $S$ is denoted by
$\Aut S$.

A congruence relation is the kernel of a homomorphism. The kernel of a
mapping $f$ is denoted by $\ker f$.

For a subset $A$ of a structure $B$ let $\erzgruppe{A}_B$ denote the
substructure generated by A.

Let $M$ be a set and $\equiv$ an equivalence relation, then
$\faktorisiert M\equiv \definiert \{[x]_\equiv\mid x∈M\}$ is the
partition into equivalence classes of $M$. A set $T\subseteq M$ is
called \defindex{transversal} of the partition of $\equiv$, if it
contains exactly one element in each equivalence class, i.e. for all
$x∈M$ the equation $|[x]_\equiv\cap T| = 1$ holds.  The
\defindex{factor relation} of a relation $ϱ∈M^n$ is defined by the set
\begin{align*}
    [ϱ]_\equiv:=\{([x_1]_\equiv,\ldots,[x_n]_\equiv)\mid (x_1,\ldots,x_n)∈ϱ\}
\end{align*}

For a tuple $t$ we denote the projection to the $n$th place with the symbol $π_nt$.

\section{Factorisation}\label{factorisation}
In this section we will define a structure that can be considered as a
factorisation of po-groups. This work is inspired by binary
relation orbifolds as discussed in
\cite{Zickwolff1991,Zickwolff1991a,Borchmann:2009}. However, these
ideas are based on factor relations while we use a description here,
that is more focused on the internal structure of the factor
structure. For simplicity, we focus on semi-regular group actions, that
generate orbits which are pairwise isomorphic with respect to the order
relation. A well-known example for such group actions are po-groups.

In a right po-group $\Halbord{\gruppe}{\leq}$ every group element
$x∈\gruppe$ acts as an automorphism on the order relation by
simultaneous right multiplication to all group elements. Cancelability
of the group tells us that $\gruppe$ acts regular on
$\Halbord{\gruppe}{\leq}$. Thus every subgroup $\gruppe[U]$ of
$\gruppe$ acts semi-regular on $\gruppe$. For po-groups the same is
also true for the left multiplication. The elements of the group are
either strictly negative, strictly positive, the neutral element or
incomparable to the latter one, which means that for all elements
$a,b∈\gruppe$ and $x∈\gruppe[U]$ the inequality $a\leq xa$ holds iff
for all elements $b∈\gruppe$ the same inequality $b\leq xb$
holds. Furthermore, when we fix $a$ and $b$, the mapping
$\wabbildung{φ}{x}{xa^{-1}b}$ defines an isomorphism between the
cosets $\gruppe[U]a$ and $\gruppe[U]b$, both with respect to the group
action of $\gruppe[U]$ and with respect to the order relation. We will
call this property translative, if the isomorphism commutes with the
group operation in $\gruppe[U]$. We define it in the more general
setting of small categories which allows us to express nested
factorisations.

The factor relation of an order relation is not necessarily an order
any more. Thus, we will focus on the orbits of the tuples of the
relation. This may lead to a structure that has more than one
arrow that may connect two objects. As a relation can always be
considered as a simple directed graph we could generalise to
non-simple graphs, in our case small categories, at least for
intermediate results. For this view we need a notion that resembles
the properties of right po-groups.

We call a category foldable by a permutation group iff its orbits can
be considered again as a category in an obvious way.
\begin{definition}\label{def:foldable}
  Let $\kategorie K$ be a category and $\gruppe$ be a group that acts on
  $\kategorie K$. Then $\kategorie K$ is called \defindex{foldable} by
  $\gruppe$ iff for any four arrows
  $\klmorphismus a,\klmorphismus b,\klmorphismus c,\klmorphismus d
  ∈\Morall{\kategorie K}$ with
  $\klmorphismus a^{\gruppe}=\klmorphismus c^{\gruppe}$ and
  $\klmorphismus b^{\gruppe}=\klmorphismus d^{\gruppe}$ the following
  equation holds whenever $\klmorphismus a * \klmorphismus b$ and
  $\klmorphismus c * \klmorphismus d$ exist:
  \begin{equation}
    \label{eq:1}
    (\klmorphismus a * \klmorphismus b)^{\gruppe} = (\klmorphismus c *\klmorphismus d)^{\gruppe}.
  \end{equation}
\end{definition}

\noindent Actually, not every category is foldable with respect to every
automorphism group as it can be seen in the following example.

\begin{example}
  Suppose the following category $\kategorie K$:

  
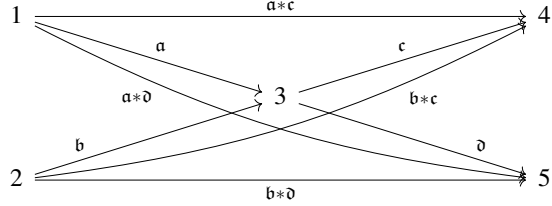
{\centering%
    \begin{tikzcd}[column sep=20ex]
      1\arrow{rd}{\klmorphismus a}\arrow{rr}{\klmorphismus a*\klmorphismus c}%
      \arrow[bend angle=10,bend right,swap,near start]{rrdd}{\klmorphismus a*\klmorphismus d}&&4\\
      &3\arrow{ru}{\klmorphismus c}\arrow[near end]{rd}{\klmorphismus d}&\\
      2\arrow[near start]{ru}{\klmorphismus b}\arrow[swap]{rr}{\klmorphismus b*\klmorphismus d}%
      \arrow[bend angle=10,bend right,swap,near end]{rruu}{\klmorphismus b*\klmorphismus c}&&5
    \end{tikzcd}\\
  }

  Then,
  \begin{align*}
    φ&\definiert (12)(\klmorphismus a \klmorphismus b) \bigl((\klmorphismus a*\klmorphismus
       c)(\klmorphismus b*\klmorphismus c)\bigr) \bigl((\klmorphismus
       a*\klmorphismus d)(\klmorphismus b*\klmorphismus d)\bigr)\text{ and}\\
    ψ&\definiert (45)(\klmorphismus
       c\klmorphismus d) \bigl((\klmorphismus a*\klmorphismus
       c)(\klmorphismus a*\klmorphismus d)\bigr) \bigl((\klmorphismus
       b*\klmorphismus c)(\klmorphismus b*\klmorphismus d)\bigr)
  \end{align*}
  are two automorphisms with $φ^2 = 1$, $ψ^2=1$ and $φψ=ψφ$. The
  cyclic automorphism group $\gruppe = \erzgruppe {φψ}$ has four orbits of
  arrows. These are $\menge{\klmorphismus a,\klmorphismus b}$,
  $\menge{\klmorphismus c,\klmorphismus d}$,
  $\menge{\klmorphismus a*\klmorphismus c,\klmorphismus
    b*\klmorphismus d}$ and
  $\menge{\klmorphismus a*\klmorphismus d,\klmorphismus
    b*\klmorphismus c}$.

  Obviously, we get $\klmorphismus c^{\gruppe} = d^{\gruppe}$, but
  $(\klmorphismus a * \klmorphismus c)^{\gruppe} \neq (\klmorphismus a
  * \klmorphismus d)^{\gruppe}$.  Consequently $\kategorie K$ is not
  foldable with respect to $\gruppe$.

  A different result gives the group
  $\gruppe[H] = \erzgruppe {\menge{φ,ψ}}$. It has three orbits of
  arrows: $\menge{\klmorphismus a,\klmorphismus b}$,
  $\menge{\klmorphismus c,\klmorphismus d}$,
  $\menge{\klmorphismus a*\klmorphismus c,\klmorphismus
    b*\klmorphismus d,\klmorphismus a*\klmorphismus d,\klmorphismus
    b*\klmorphismus c}$. The category is obviously foldable by
  $\gruppe[H]$.
\end{example}

\noindent So we can consider the set of orbits as a category:

\begin{definition}\label{def:Bahnkategorie}
  Let $\kategorie{K}$ be a category with the concatenation $∘$
  that is foldable by an automorphism group
  $\gruppe ≤\Aut\kategorie K$. Then the category $\orbitfaltigkeit{\kategorie K}{\gruppe}$ with
  \begin{subequations}
    \begin{align}
      \Ob({\orbitfaltigkeit{\kategorie K}{\gruppe}})
      &\definiert\orbitfaltigkeit {(\Ob{\kategorie{K}})}{\gruppe}\label{eq:D64:1}\\
      \Mor{\orbitfaltigkeit{\kategorie K}{\gruppe}}{x^{\gruppe}}{y^{\gruppe}}
      &\coloneqq\Mor{\kategorie K}{x}{y}^{\gruppe}\label{eq:D64:3}\\
      \id_{x^{\gruppe}}
      &\definiert (\id_x)^{\gruppe}\label{eq:D64:2}
        \intertext{and the concatenation}
        \klmorphismus a^{\gruppe}*\klmorphismus b^{\gruppe}
      &\definiert (\klmorphismus a \circ \klmorphismus b)^{\gruppe}\text{ whenever $\klmorphismus a * \klmorphismus b$ exists,}
    \end{align}
  \end{subequations}
  is called \defindex{orbit category} of $\kategorie K$ by $\gruppe$.
\end{definition}

\begin{lemma}
  For every category $\kategorie K$ and every semi-regular
  automorphism group $\gruppe\leq\Aut{\kategorie K}$ there exists an
  orbit category $\orbitfaltigkeit{\kategorie K}{\gruppe}$.
\end{lemma}
\begin{proof}
  Let
  \begin{tikzcd}w^{\gruppe}\arrow{r}{\klmorphismus a^{\gruppe}}
    &x^{\gruppe}\arrow{r}{\klmorphismus b^{\gruppe}}
    &y^{\gruppe}\arrow{r}{\klmorphismus c^{\gruppe}}
    &z^{\gruppe}
  \end{tikzcd} be a diagram of $\orbitfaltigkeit{\kategorie
    K}{\gruppe}$. Then, because of the semi-regularity of the group
  action there are unique vertices $\hat x, \hat y, \hat z$ and arrows
  $\hat {\klmorphismus b}∈\klmorphismus b^{\gruppe}$ and
  $\hat {\klmorphismus c}∈\klmorphismus c^{\gruppe}$ such that
  \begin{tikzcd}
    w\arrow{r}{\hat{\klmorphismus a}}
    &\hat x\arrow{r}{\hat{\klmorphismus b}}
    &\hat y\arrow{r}{\hat{\klmorphismus c}}
    &\hat z
  \end{tikzcd} is a diagram of $\kategorie K$. So
  $\klmorphismus a^{\gruppe}*\klmorphismus
  b^{\gruppe}*\klmorphismus c^{\gruppe}=(\klmorphismus a \circ
  \klmorphismus b \circ \klmorphismus c)^{\gruppe}$. In particular it shows
  \begin{align*}
    (\klmorphismus a^{\gruppe}*\klmorphismus
    b^{\gruppe})*\klmorphismus c^{\gruppe}
    =\bigl((\klmorphismus a \circ
      \klmorphismus b) \circ \klmorphismus c\bigr)^{\gruppe}
    =\bigl(\klmorphismus a \circ
      (\klmorphismus b \circ \klmorphismus c\bigr)^{\gruppe})
    =\klmorphismus a^{\gruppe}*(\klmorphismus
      b^{\gruppe}*\klmorphismus c^{\gruppe}).
  \end{align*}
  So the orbit category is indeed a category.
  \qed
\end{proof}

\noindent Subgroups of po-groups act semi-regular via left
multiplication on their supergroups. So we can work with the orbit
category.

The vertex monoids of an orbit category may not be isomorphic to each
other. If they were isomorphic we could separate them from the
category, which would allow an additional compression of the
mathematical structure. As we will see later, all vertex monoids of
po-groups are pairwise isomorphic. This information can be seen as a
property of an automorphism group.

\begin{definition}%
  \label{def:translative}%
  Let $\kategorie K$ a category and
  $\gruppe\leq\Aut{\kategorie K}$ an automorphism group. Then
  $\gruppe$ \defindex{acts translatively} on $\kategorie K$, iff
  $\gruppe$ acts semi-regular on $\kategorie K$ and for every two
  elements $a,b∈\Ob\kategorie K$ there exists a category isomorphism
  between the orbits $\erzkategorie{\kategorie K}{a^{\gruppe}}$ and
  $\erzkategorie{\kategorie K}{b^{\gruppe}}$ that commutes with the
  group action of $\gruppe$ so that for all automorphisms $g∈\gruppe$
  the following diagram commutes:

  \noindent{%
    \hfill
    \begin{tikzcd}
      \erzkategorie{\kategorie K}{a^{\gruppe}}\arrow{r}{g}\arrow{d}{φ}
      &\erzkategorie{\kategorie K}{a^{\gruppe}}\arrow{d}{φ} \\
      \erzkategorie{\kategorie K}{b^{\gruppe}}\arrow{r}{g}
      &\erzkategorie{\kategorie K}{b^{\gruppe}}
    \end{tikzcd}
    \hfill
    \strut
  }
\end{definition}

In the following we will use this definition in order to separate the
group action of a po-group on itself from the order relation:

\begin{lemma}\label{lemma:orderd-group-translative}
  Every subgroup $\gruppe[S]\normalteilervon\gruppe$ of a po-group $\gruppe$ acts translatively on $\gruppe$ via right
  multiplication.
\end{lemma}
\begin{proof}
  Let $\gruppe[U]$ act left-associative on $\gruppe$. Furthermore
  chose a transversal $T$ of the orbits of $\gruppe[U]$ on $\gruppe$.
  Then, each orbit can be represented in the form $t\gruppe[U]$ where
  $t∈T$. For another element $s∈T$ multiplication from the left by
  $st^{-1}$ forms an order isomorphism from $t\gruppe[U]$ to
  $s\gruppe[U]$. As $\gruppe[U]$ acts from right on itself, and the
  elements from $T$ from the left, both commute with each other.
  \qed
\end{proof}

\section{Representation and Unfolding}\label{representation}

In the preceding section far we have made a category from a larger one. In a similar
way to group extensions~\cite{Schreier:1926,Schreier:1925} we may
define a representation of the larger category using the orbit
category and some additional mathematical magic similar to
\cite{Zickwolff1991,BorchmannDipl,Borchmann:2009}.

Throughout this paper we will use the name \defindex{annotation} for a
contravariant homomorphism from a category into a group. Consequently
whenever there is a (covariant) homomorphism $F$ from one category
$\kategorie K$ into another one $\kategorie L$ and an annotation
$\abbildung A{\kategorie L}{\gruppe}$, then also the concatenation
$F\circ A$ is an annotation.

\begin{definition}
  Let $\kategorie K$ be a category, $\gruppe$ a be group and
  $\abbildung A{\kategorie K}{\gruppe}$ be an annotation. The triplet
  $\tupel{\kategorie K, A, \gruppe}$ is called a
  \defindex{representation} (of $\kategorie K$).

  If $A$ is faithful, $\tupel{\kategorie K, A, \gruppe}$ is said to be
  \defindex{faithful}.
\end{definition}

\noindent In our case the representation represents a category. First we
introduce the concatenation, and the we define the unfolding as the
object the representation refers to.

\begin{lemma}
  Let $\tupel{\kategorie K, A, \gruppe}$ be a representation and
  $R=\Morall{\kategorie K}\times\gruppe$. Then the partial operation
  defined by
  \begin{align}
    \wabbildung{\abbildung {⊙&}{R\times R}{R}}
                               {\paar[Big]{\paar {\klmorphismus a}{g}}{\paar[big] {\klmorphismus b}{A(\klmorphismus a)g}}}
                               {\paar{\klmorphismus a*\klmorphismus b}{g}}\text{, whenever $\klmorphismus a *\klmorphismus b$ exists}
  \end{align}
  is associative.
\end{lemma}

\begin{proof}
  Let
  $\klmorphismus a, \klmorphismus b,\klmorphismus c∈\Morall{\kategorie
    K}$ such that $\klmorphismus a *\klmorphismus b$ and
  $\klmorphismus b * \klmorphismus c$ exist. Then
  \begin{align*}
    \bigl(\paar {\klmorphismus a}g⊙\paar{\klmorphismus b}{A(\klmorphismus a)g}\bigr)⊙\paar{\klmorphismus c}{A(\klmorphismus b)A(\klmorphismus a)g}
    &= \paar {\klmorphismus a*\klmorphismus b}g⊙\paar{\klmorphismus c}{A(\klmorphismus a *\klmorphismus b)g}\\
    &=\paar {\klmorphismus a*\klmorphismus b*\klmorphismus c}g\\
    &=\paar {\klmorphismus a}g⊙\paar {\klmorphismus b*\klmorphismus c}{A(\klmorphismus a)g}\\
    &=\paar {\klmorphismus a}g⊙\bigl(\paar {\klmorphismus b}{A(\klmorphismus a)g}⊙\paar {\klmorphismus c}{A(\klmorphismus b)A(\klmorphismus a)g}\bigr)
  \end{align*}
  \qed
\end{proof}

\begin{definition}\label{def:Entfaltung}
  Let $\tupel{\kategorie K, A, \gruppe}$ be a representation.

  The category $\kategorie U$ with vertices
  $\Ob{\kategorie U}\definiert(\Ob{\kategorie K})\times\gruppe$ and
  arrows $\Morall {\kategorie U}\definiert\Morall{\kategorie K}\times\gruppe$ and
  start and end of the arrows according to
  \begin{subequations}
    \begin{align}
      \moranfang{\paar{\klmorphismus a}g}
      &\definiert\paar{\moranfang{\klmorphismus a}}{g},\\
      \morende{\paar{\klmorphismus a}{g}}&\definiert\paar[big]{\morende\klmorphismus a}{A(\klmorphismus a)g}\\
      \intertext {and the concatenation for $\morende{\klmorphismus a} = \moranfang{\klmorphismus b}$}
      \paar {\klmorphismus a}{g}⊙\paar[big] {\klmorphismus b}{A(\klmorphismus a)g}&\definiert
      \paar{\klmorphismus a*\klmorphismus b}{g}
    \end{align}
  \end{subequations}
  is called the \defindex{unfolding} $\entfaltung{\kategorie K}A\gruppe$ of
  $\tupel{\kategorie K, A, \gruppe}$. In this case, we call
  $\tupel{\kategorie K, A, \gruppe}$
  a \defindex{representation} of $\kategorie U$.
\end{definition}

\noindent The group of a representation has an induced automorphism action on its unfolding.

\begin{lemma}\label{lemma:L159}
  Let $\tupel{\kategorie K,A,\gruppe}$ be a representation.
  Then, the pair of mappings $\paar{Φ}{Ψ}$ with
  \begin{subequations}
    \begin{align}
      \wabbildung{\abbildung{Φ&}{\gruppe×\bigl(\Ob{(\entfaltung{\kategorie K}{A}{\gruppe})}\bigr)}{\Ob({\entfaltung{\kategorie K}{A}{\gruppe}})}}
      {\paar[big]{g}{\paar{x}{h}}}{\paar{x}{hg}}.\\
      \wabbildung{\abbildung{Ψ&}{\gruppe×\bigl(\Morall{(\entfaltung{\kategorie K}{A}{\gruppe})}\bigr)}{\Morall{\entfaltung{\kategorie K}{A}{\gruppe}}}}
      {\paar[big]{g}{\paar{\klmorphismus x}{h}}}{\paar{\klmorphismus x}{hg}}.
    \end{align}
  \end{subequations}
  defines a right associative automorphism action on
  $\entfaltung{\kategorie K}{A}{\gruppe}$.
\end{lemma}
\begin{proof}
  Let
  \begin{align*}
    \paar{\klmorphismus a}{h}&∈\Mor{\entfaltung{\kategorie K}{A}{\gruppe}}{\paar[big]
    {x}{h}}{\paar{y}{{h'}}}\text{ and}\\
    \paar{\klmorphismus b}{{h'}}&∈\Mor{\entfaltung{\kategorie K}{A}{\gruppe}}{\paar[big]
    {y}{{h'}}}{\paar{z}{{h''}}}
  \end{align*}
  be two arrows from $\entfaltung{\kategorie K}{A}{\gruppe}$. Then:
  \begin{align*}
    A\paar{\klmorphismus a}{h} &= {h'}\multnegpot{h}& A\paar{\klmorphismus b}{h'} &= {h''}\multnegpot{{h'}}&
    A\bigl(\paar{\klmorphismus a}{h}*\paar{\klmorphismus b}{h'}\bigr) &= {h''}\multnegpot{{h'}}{h'}\multnegpot{h} = {h''}\multnegpot{h}\\
    \moranfang\paar{\klmorphismus a}h &= \paar x{h} & \moranfang\paar{\klmorphismus b}{h'}&=\paar y{{h'}}&
    \moranfang\bigl(\paar{\klmorphismus a}{h}*\paar{\klmorphismus b}{h'}\bigr) &= \paar x{h} \\
    \morende\paar{\klmorphismus a}{h} &= \paar y{{h'}} & \morende\paar{\klmorphismus b}{h'}&=\paar z{{h''}}&
    \morende\bigl(\paar{\klmorphismus a}{h} *\paar{\klmorphismus b}{h'})&=\paar z{{h''}}
  \end{align*}
  Applying $Φ$ and $Ψ$ together with a group element $g∈\gruppe$, we get
  \begin{align*}
    Φ\paar[big]{g}{\moranfang\paar{\klmorphismus a}{h}} &= \paar x{hg}
    &Φ\paar[big]{g}{\moranfang\paar{\klmorphismus b}{{h'}}}&=\paar y{{h'}g}\\
    Φ\paar[Big]{g}{\moranfang\bigl(\paar{\klmorphismus a}{h}*\paar{\klmorphismus b}{{h'}}\bigr)} &= \paar x{hg} \\
    Φ\paar[big]{g}{\morende\paar{\klmorphismus a}{h}} &= \paar y{{h'}g}
    & Φ\paar[big]{g}{\morende\paar{\klmorphismus b}{{h'}}}&=\paar z{{h''}g}\\
    Φ\paar[Big]{g}{\morende\bigl(\paar{\klmorphismus a}{h} *\paar{\klmorphismus b}{{h'}}\bigr)}&=\paar z{{h''}g}\\
  \end{align*}
  \sloppypar For start and end of an arrow $\paar{\klmorphismus a}{h}$ we get the equation
  ${h'}g\cdot \multnegpot{(hg)} = {h'}g\multnegpot
  g\multnegpot{h}={h'}\multnegpot{h} = A(\klmorphismus a)$. Thus,
  $\paar{\klmorphismus a}{hg}$ is an arrow in the set
  $\Mor{\entfaltung{\kategorie K}{A}{\gruppe}}{\paar
    {x}{hg}}{\paar{y}{{h'}g}}$, and the equations
  $\moranfang\bigl(Ψ\paar{\klmorphismus
    a}{h}\bigr)=Φ\bigl(\moranfang\paar{\klmorphismus a}{h}\bigr)$ and
  $\morende\bigl(Ψ\paar{\klmorphismus
    a}{h})=Φ\bigl(\morende\paar{\klmorphismus a}{h}\bigr)$ hold.

  Applying another group element $g'∈\gruppe$ to these formulas
  substitutes $g$ with $gg'$ since $g$ is used only as factor in the
  right multiplication.  For $h=\multnegpot g$ we get the identity
  automorphism. Consequently, the two mappings are right associative
  group actions.

  In combination with Definition~\ref{def:Entfaltung}, this leads to the equations
  \begin{align*}
    Φ\paar g{\moranfang{\klmorphismus a}} = \moranfangΨ\paar g{\klmorphismus a},\qquad
    Φ\paar g{\morende{\klmorphismus a}} = \morendeΨ\paar g{\klmorphismus a},\qquad
    ψ\paar g{\klmorphismus a * \klmorphismus b}
    = ψ\paar g{\klmorphismus a} * ψ\paar g{\klmorphismus b}
  \end{align*}
  
  Thus, $Φ$ and $Ψ$ together form a right associative automorphism.
  \qed
\end{proof}

\noindent In the following we call the automorphism action of the previous lemma
\defindex{induced automorphism action} of $\gruppe$ on
$\entfaltung{\kategorie K}A{\gruppe}$.

The next two lemmas show, that
orbit category and unfolding correspond to each other.

\begin{lemma}\label{lemma:L183}
  Let $\tupel{\kategorie K,A,\gruppe}$ denote a representation. Then
  \begin{align*}
    \wabbildung{\wabbildung{\abbildung{π&}{\Ob{\entfaltung{\kategorie K}A{\gruppe}}}{\Ob{\kategorie K}}}
                                          {\paar{x}{g}}{x}\\\nonumber
    \abbildung{&}{\Morall{\entfaltung{\kategorie K}A{\gruppe}}}{\Morall{\kategorie K}}}
                                          {\paar{\klmorphismus a}{g}}{\klmorphismus a}
  \end{align*}
  is a full faithful category homomorphism.
\end{lemma}

\begin{proof}
  Let
  $\paar{\klmorphismus a}{g},\paar[big]{\klmorphismus
    b}{A(a)g}∈\Morall{\entfaltung{\kategorie K}A{\gruppe}}$. Then the following equations hold:
  \begin{align*}
    \moranfang\bigl(π\paar{\klmorphismus a}{g})
    &=\moranfang \klmorphismus a
      = π\paar[big]{\moranfang (\klmorphismus a)}{g}
      = π\bigl(\moranfang\paar{\klmorphismus a}{g}\bigr)\\
    \morende\bigl(π\paar{\klmorphismus a}{g})
    &=\morende \klmorphismus a
      = π\paar[big]{\morende(\klmorphismus a)}{A(\klmorphismus a)g}
      = π\bigl(\morende\paar{\klmorphismus a}{g}\bigr).
      \intertext{for the concatenation we get}
      π\paar{\klmorphismus a}{g}*π\paar[big]{\klmorphismus b}{A(g\klmorphismus a)}
    &=\klmorphismus a*\klmorphismus b
      = π\paar{\klmorphismus a* \klmorphismus b}{g}
      % π\paar{\id_x}{g}&=\id_x
  \end{align*}
  Thus $π$ is a homomorphism. Two arrows are mapped to the same image
  iff they coincide in their first component, which means that they
  are either identical or have different second components. In the
  latter case they start at different vertices.

  For every arrow $\klmorphismus a$ of the category $\kategorie K$ the
  set $\menge{\klmorphismus a}\times\gruppe$ is the subset of the
  arrows of the unfolding, which maps to the set
  $\menge{\klmorphismus a}$. So the homomorphism is full, too.
  \qed
\end{proof}

\begin{lemma}
  Under the conditions of Lemma~\ref{lemma:L183} we get the following equation
  \begin{equation}
    \label{eq:L54:1}
    \faktorisiert{(\entfaltung{\kategorie K}{A}{\gruppe})}{\ker{π}} =
    \orbitfaltigkeit{(\entfaltung{\kategorie K}{A}{\gruppe})}{\gruppe}
  \end{equation}
\end{lemma}

\begin{proof}
  Let
  $\paar {\klmorphismus a}{g},\paar {\klmorphismus b}{h}∈\Morall{\entfaltung{\kategorie K}A{\gruppe}}$
  two arrows.
  If the concatenation $\paar {\klmorphismus a}{g}*\paar {\klmorphismus b}{h}$ exists, the following
  equations hold:
  \begin{align*}
    [\paar{\klmorphismus a}{g}]_{\ker{π}}*[\paar{\klmorphismus b}{h}]_{\ker{π}}
    &= [\paar{\klmorphismus a}{g}*\paar{\klmorphismus b}{h}]_{\ker{π}} =
      [\paar{\klmorphismus a*\klmorphismus b}{g}]_{\ker{π}} \\
    &=
    \Menge{\paar {\klmorphismus c}{ξ}}{\klmorphismus c
      = \klmorphismus a*\klmorphismus b, ξ∈\gruppe}
    = \menge{\klmorphismus a*\klmorphismus b}\times {\gruppe}\\
    &= \paar{\klmorphismus a*\klmorphismus b}{g}^{\gruppe},
      \intertext{in particular,}
    [\paar{\klmorphismus a}{g}]_{\ker{π}}*[\paar{\id_{\morende\klmorphismus a}}{h}]_{\ker{π}}
    &= \paar{\klmorphismus a}{g}^{\gruppe}\text{ and}
    [\paar{\id_x}{g}]_{\ker{π}}*[\paar{\id_x}{h}]_{\ker{π}}
    &= \paar{\id_x}{g}^{\gruppe},      
  \end{align*}
  so \eqref{eq:L54:1} holds.
  \qed
\end{proof}

\begin{corollary}
  The orbit category
  $\orbitfaltigkeit{(\entfaltung{\kategorie K}A{\gruppe})}{\gruppe}$
  of an unfolding $\entfaltung{\kategorie K}A{\gruppe}$ of a category
  $\kategorie K$ with respect to the same group \gruppe{} and the group action from 
  Lemma~\ref{lemma:L159} is isomorphic to the original category $\kategorie K$.
\end{corollary}
\begin{proof}
  The projection $π$ can be divided into an epimorphism
  $\abbildung {φ}{\entfaltung{\kategorie
      K}A{\gruppe}}{\ker{π}=\orbitfaltigkeit{(\entfaltung{\kategorie
        K}A{\gruppe})}{\gruppe}}$ and an isomorphism between $\ker π$
  and the image of $π$. Since $π$ is full, the image of $π$ is the
  whole category $\kategorie K$.
  \qed
\end{proof}

\noindent Now, as we can re-fold an unfolded representation into the
corresponding category we can hope that we can find a similar
equivalence for unfolding folded categories.

\section{Reconstruction}\label{reconstruction}

As we have seen, there is a relationship between orbit categories, the
unfolding and their representations. So far we did not answer the
following question: Which categories can be folded into a
representation such that the unfolding is isomorphic to the original
category? This shall be addressed in the current section.

One part of the unfolding is given by the orbit category. It is
sufficient to find an annotation which unfolds into a category
annotation that is natural to the original category.

One such candidate is hidden in the group action. If
$\Halbord{\gruppe}{\leq}$ is a po-group. Then the mapping
\[
  \wabbildung{\abbildung{A}{\gruppe\times\gruppe}{\gruppe}}
  {\paar gh}{hg^{-1}}
\]
is an annotation of $\Halbord{\gruppe}{\leq}$. If
$\gruppe[S]\leq\gruppe$ is a subgroup of $\gruppe$. The orbits of the
right associative group action of $\gruppe[S]$ results in the coset
partition $\Menge{g\gruppe[S]}{g∈\gruppe}$. By fixing one Element
$g_0$ we can transfer the group structure of $\gruppe[S]$ to
$g_0\gruppe[S]$, where $g_0$ is the neutral element and
$g_0s\cdot_{g_0}g_st\definiert (g_0s)g_0^{-1}(g_0t)=g_0st$. In that
way we get a similar construction for every orbit. We formalise this
idea with the help of a transversal set.

For a category $\kategorie K$ and a translative automorphism
group $\gruppe\leq\Aut\kategorie K$ a \defindex{transversal (set)}
$T\subseteq\Ob\kategorie K$ is a set such that
$T^{\gruppe}=\Ob\kategorie K$ and $∀g∈\gruppe:|g^{\gruppe}\cap T| =1$.
Obviously, a transversal always exists. As the group acts
semi-regular, for every vertex $x∈\Ob\kategorie K$ there exists a
unique automorphism $g_T(x)$ such that $x^{\big(g_T(x)\bigr)^{-1}}∈T$
holds. This automorphism is called \defindex{canonical}, in the same
way as the pair $\paar[big]{x^{\bigl(g_T(x)\bigr)^{-1}}}{g_T(x)}$ is
denoted by the name \defindex{canonical vertex
  annotation}. Consequently, for every vertex, every translative
automorphism group and every transversal there exists a unique vertex
annotation. In the same way we call the pair
$\paar{\klmorphismus a^{\gruppe}}{g_T(\moranfang{\klmorphismus a})}$
\defindex{canonical arrow annotation}. The canonical arrow annotation
is unique for every arrow.

In the following the  mapping
$\wabbildung{\abbildung{A_T}{\Morall{\kategorie K}}{\gruppe}}{\klmorphismus a}{g_T(\morende\klmorphismus
a)\bigl(g_T(\moranfang\klmorphismus a)\bigr)^{-1}}$ will be called
\defindex{natural annotation}. Obviously the natural annotation is
constant on every orbit of $\Morall {\kategorie K}$ under the action
of $\gruppe$.

With these ingredients we construct a representation, now. First we
rebuild the orbit category from the transversal and all arrows which
start in the transversal. Let $\kategorie C$ be a category such that
\begin{align*}
  \Ob\kategorie C
  &\definiert T\\
  \Mor{\kategorie C}{x}{y}
  &\definiert \bigcup_{z∈y^{\gruppe}}\Mor{\kategorie K}{x}{z}
\end{align*}
\begin{align*}
  \klmorphismus a *_{\kategorie C}\klmorphismus b&\definiert
  \klmorphismus a*\klmorphismus b^{A_T(\klmorphismus a)}.
\end{align*}
Then we get
\begin{align*}
  \moranfang_{\kategorie C} \klmorphismus a
  &= \moranfang_{\kategorie K} \klmorphismus a ∈ T
    \intertext{and – by abuse of notation –}
  \morende_{\kategorie C} \klmorphismus a
  &= \morende_{\kategorie K} \klmorphismus a^{\gruppe}\cap T
\end{align*}

As the natural annotation is constant on the orbits of the arrows of
$\kategorie K$, we can define
\[
  A(\klmorphismus a)
  \definiert A_T(\klmorphismus a).
\]

Finally, the tuplet $\tupel{\kategorie C,A,\gruppe}$ describes
a representation. For the unfolding of this representation we can prove the following:

\begin{lemma}
  The unfolding $\entfaltung{\kategorie C}A{\gruppe}$ is isomorphic to $\kategorie K$.
\end{lemma}
\begin{proof}
  Let $\kategorie K'\definiert \entfaltung{\kategorie C}A{\gruppe}$ denote
  the unfolding of $\tupel{\kategorie C,A,\gruppe}$. Then:
  \begin{align*}
    \Ob\kategorie K'
    &= \Menge{\paar xg}{g∈\gruppe, x∈T}\\
    \Mor{\kategorie K'}[big]{\paar xg}{\paar yh}
    &= \Menge{\paar{\klmorphismus a}{g}}
      {\klmorphismus a∈\Mor{\kategorie C}xy,A_T(\klmorphismus a) = hg^{-1}}
      \intertext{with the concatenation}
      \paar {\klmorphismus a}{g}⊙\paar {\klmorphismus b}{h}&\definiert
      \paar{\klmorphismus a*_{\kategorie C}\klmorphismus b}{g}
  \end{align*}
  We observe that for any arrow $\paar{\klmorphismus a}{g}∈\Morall{\kategorie K'}$ the following equations hold:
  \begin{align*}
    \moranfang_{\kategorie K'}{\paar{\klmorphismus a}{g}}
    &= \paar{\moranfang \klmorphismus a}{g}\\
    \morende_{\kategorie K'}{\paar{\klmorphismus a}{g}}
    &= \paar[big]{\moranfang \klmorphismus a}{A_T(\klmorphismus a)g}\\
  \end{align*}

  It is easy to see that the two mappings
  \begin{align*}
    \wabbildung{\abbildung{Φ&}{\Ob\kategorie K'}{\Ob\kategorie K}}
                              {\paar xg}{x^g},\\
    \wabbildung{\abbildung{Ψ&}{\Morall{\kategorie K'}}
                              {\Morall{\kategorie K}}}
                              {\paar {\klmorphismus a}{g}}
                              {\klmorphismus a^g}
  \end{align*}
  form an isomorphism between $\kategorie K$ and $\kategorie K'$ with
  the inverse mappings 
  \begin{align*}
    \wabbildung{\abbildung{Φ^{-1}&}{\Ob\kategorie K}{\Ob\kategorie K'}}
                              {x}{\paar[big] {x^{(g_T(x))^{-1}}}{g_T(x)},\\
    \wabbildung{\abbildung{Ψ^{-1}&}{\Morall{\kategorie K}}
                              {\Morall{\kategorie K'}}}
                              {\klmorphismus a}
                              {\paar[big]{\klmorphismus a^{(g_T(\moranfang\klmorphismus a))^{-1}}}{g_T(\moranfang\klmorphismus a)}}},
  \end{align*}
  Each vertex can be uniquely described as a pair of a transversal
  element and an automorphism, since the automorphism group acts
  semi-regular on $\kategorie K$. This means that for each fixed
  transversal element the group action is a bijection between the
  automorphism group and the orbit of that transversal element.

  The same is true for the orbits of the arrows while we use the
  starting vertex of every arrow as index into the arrow orbits. So
  also $Ψ$ is bijective. Let
  $\paar {\klmorphismus a}g∈\Morall {\kategorie K'}$ and
  $\paar {\klmorphismus b}h∈\Morall {\kategorie K'}$ two arrows. Then
  $\paar {\klmorphismus a}g ⊙ \paar {\klmorphismus b}h$ is defined iff
  $\klmorphismus a*\klmorphismus b$ exists and 
  $h = A_T(\klmorphismus a)g$ holds. Then, we get
  \begin{align*}
    Ψ\bigl(\paar {\klmorphismus a}g ⊙ \paar {\klmorphismus b}h\bigr)
    &= Ψ\paar {\klmorphismus a *_{\kategorie C}\klmorphismus b}g 
      = Ψ\paar {\klmorphismus a * \klmorphismus b^{A_T(\klmorphismus a)}}g
      = (\klmorphismus a * \klmorphismus b^{A_T(\klmorphismus a)}) ^g\\
    &= \klmorphismus a^g * (\klmorphismus b^{A_T(\klmorphismus a)g})
    = \klmorphismus a^g * \klmorphismus b^{h}
    = Ψ\paar {\klmorphismus a}g*Ψ \paar {\klmorphismus b}h
  \end{align*}
  Thus, the pair $\paar{Φ}{Ψ}$ is an isomorphism between the
  categories $\kategorie K$ and $\kategorie K'$.
  \qed
\end{proof}

\noindent So we have some way to reconstruct the category from a transversal
representation. This means that the category $\kategorie C$ above can be
considered as a fundamental system of $\kategorie K$. Now, we replace
the transversal category by the orbit category.

\begin{lemma}
  The categories $\orbitfaltigkeit{\kategorie K}{\gruppe}$ and $\kategorie C$ are isomorphic.
\end{lemma}
\begin{proof}
  Consider the mappings
  \begin{align*}
    \wabbildung{\abbildung{Φ&}{\Ob\kategorie C}{\Ob{\orbitfaltigkeit{\kategorie K}{\gruppe}}}}{x}{x^{\gruppe}}\\
    \wabbildung{\abbildung{Ψ&}{\Morall{\kategorie C}}{\Morall{{\orbitfaltigkeit{\kategorie K}{\gruppe}}}}}{\klmorphismus a}{\klmorphismus a^{\gruppe}}.
  \end{align*}
  Both are obviously bijective with the inverse mappings (when we identify singleton sets with their elements)
  \begin{align*}
    \wabbildung{\abbildung{Φ^{-1}&}{\Ob{\orbitfaltigkeit{\kategorie K}{\gruppe}}}{\Ob\kategorie C}}{x}{x\cap T}\\
    \wabbildung{\abbildung{Ψ^{-1}&}{\Morall{{\orbitfaltigkeit{\kategorie K}{\gruppe}}}}{\Morall{\kategorie C}}}{\klmorphismus a}{\hat{\klmorphismus  a}∈\klmorphismus a: \moranfang{\hat{\klmorphismus a}}∈T}.
  \end{align*}
  As all arrows in one orbit differ pairwise in their starting vertex,
  the inverse mappings are well defined.

  Obviously the mappings act as a homomorphism on the starting
  points of the arrows. For the end points we get
  \begin{align*}
    Φ(\morende \klmorphismus a) = (\morende \klmorphismus a )^{\gruppe} =
    \morende (\klmorphismus a^{\gruppe}) = \morende Ψ(\klmorphismus a).
  \end{align*}
  And for the concatenation $\klmorphismus a *_{\kategorie C}\klmorphismus b$ of two arrows $\klmorphismus a,\klmorphismus b∈\kategorie C$ we get:
  \begin{align*}
    Ψ(\klmorphismus a*_{\kategorie C} \klmorphismus b)
    &=
      (\klmorphismus a*_{\kategorie C}\klmorphismus b)^{\gruppe}
      = (\klmorphismus a*\klmorphismus b^{A_T(\klmorphismus a)})^{\gruppe}
      \intertext{since $\gruppe$ acts semi-regular both on $\klmorphismus b^{\gruppe}$ and on $(\moranfang \klmorphismus b^{A_T(\klmorphismus a)})^{\gruppe}$, we get}
      Ψ(\klmorphismus a*_{\kategorie C} \klmorphismus b)
    &=\klmorphismus a^{\gruppe}*_{\orbitfaltigkeit{\kategorie K}{\gruppe}}\klmorphismus b^{\gruppe}
      =Ψ\klmorphismus a*_{\orbitfaltigkeit{\kategorie K}{\gruppe}}Ψ\klmorphismus b
  \end{align*}
  The concatenation $\klmorphismus a*_{\kategorie C} \klmorphismus b$
  is defined iff there exists some group element $g$ such that
  $\klmorphismus a*\klmorphismus b^g$ is defined, which is the case
  iff
  $\klmorphismus a^{\gruppe}*_{\orbitfaltigkeit{\kategorie
      K}{\gruppe}}\klmorphismus b^{\gruppe}$ is defined. So, the pair
  $\paar {Φ}{Ψ}$ is an isomorphism between
  $\orbitfaltigkeit{\kategorie K}{\gruppe}$ and $\kategorie C$.
  \qed
\end{proof}

\noindent As $A_T$ is constant on the orbits, we get for every
arrow
$\klmorphismus a∈\Morall {\orbitfaltigkeit{\kategorie K}{\gruppe}}$
the equation $A_T[\klmorphismus a] = A(Φ(\klmorphismus a))$.

Consequently we can prove the unfolding.
\begin{theorem}
  \[
    \entfaltung {(\orbitfaltigkeit{\kategorie K}{\gruppe})}{A_T}{\gruppe}\isomorph \kategorie K.
  \]
\end{theorem}
\begin{proof}
  The unfolding $\entfaltung {\kategorie C}{A}{\gruppe}$ is isomorphic
  to the category $\kategorie K$.  As
  $\orbitfaltigkeit{\kategorie K}{\gruppe}$ is isomorphic to
  $\kategorie C$ and this isomorphism preserves the annotation, we can
  replace the elements (vertices and arrows) of $\kategorie C$ in
  $\entfaltung {\kategorie C}{A}{\gruppe}$ by their isomorphic image
  and get
  $\entfaltung {(\orbitfaltigkeit{\kategorie
      K}{\gruppe})}{A_T}{\gruppe}$.
  \qed
\end{proof}

\section{Vertex categories}
The idea behind a fundamental system is to reduce the information that
is managed by a certain structure. If an orbit category has cycles
that don't occur in the original category, a possibly infinite number
of arrows is preserved between any two vertices of the cycle. In some
case the information about their structure can be stored more
efficiently in the preimage of the vertex monoids. These are the
generated subcategories of vertex orbits. We call them vertex
categories.

According to the definition, for a translative automorphism group
$\gruppe∈\Aut \kategorie K$ of a category $\kategorie K$, any
two orbits $x^{\gruppe},y^{\gruppe}∈\kategorie K^{\gruppe}$ are
isomorphic to each other. As the group action is semi-regular, for
every arbitrary, but fixed vertex $x∈\Ob\kategorie K$ there is a
bijection
\begin{equation}
  \label{eq:6}
  \wabbildung{\abbildung{φ_x}{\gruppe}{x^{\gruppe}}}{g}{x^g}
\end{equation}
between the group $\gruppe$ and the orbit $x^{\gruppe}$.
So we can define a category $\kategorie G$ in the following way:
\begin{definition}
  Let $\kategorie K$ be a category and $\gruppe∈\Aut \kategorie K$
  a translative automorphism group. Then, for every vertex
  $x∈\Ob\kategorie K$ the right groupal category
  $\kategorie G_x\definiert \erzkategorie{\kategorie K}{x^{\gruppe}}$
  with the additional group action
  $\abbildung{\cdot_{\kategorie G_x}}{\kategorie G_x\times \Ob\kategorie G}{\kategorie G_x}$
  defined via
  \begin{enumerate}
  \item $∀g,h∈\gruppe: x^g\cdot_{\kategorie G_x}x^h \definiert x^{g\cdot_{\gruppe}h}$
  \item $∀\klmorphismus a∈\Morall{\kategorie G_x},h∈\gruppe: \klmorphismus a\cdot_{\kategorie G_x}x^g \definiert \klmorphismus a^{g}$
  \end{enumerate}
  is called \defindex{vertex category}.
\end{definition}

If the category $\kategorie K$ is an ordered set, the vertex category
is a right po-group. If it is a po-group, which
is factorised by one of its normal subgroups, the vertex categories
are po-groups that are isomorphic to the normal subgroup with
respect to the group operation and the order relation.

Every po-group can be considered as a category whose objects act on
itself via a left- and a right-associative automorphism actions. Both
actions commute due to the associativity of the group operation. For
simple categories this actions also define a binary operator on the
arrows.  In general such a category with such a binary operator is
called “groupal category”. Here, we consider only strict associativity.

\begin{definition}
  Let $\kategorie K$ be a category. Furthermore let
  $\abbildung \cdot {\kategorie K\times\Ob\kategorie K}{\kategorie K}$
  a binary operation that acts as a group operation on the objects
  $\Ob\kategorie K$ and as a left-associative automorphism action on
  $\kategorie K$.  Then $\paar {\kategorie K}{\cdot}$ is called a
  \defindex{right-groupal category}. Dually, the term
  \defindex{left-groupal category} denotes such a category if the
  group action is right-associative.

  The categroy is called a \defindex{groupal category}, if it is both
  left groupal and right groupal, such that both group actions commute
  with each other and the following equation holds for all arrows
  $\klmorphismus a,\klmorphismus b ∈\kategorie K$:

  \begin{equation}
    \label{eq:23}
    \klmorphismus a ^{\moranfang\klmorphismus b} * {}^{\morende\klmorphismus a}\klmorphismus b
    = {}^{\moranfang\klmorphismus a}\klmorphismus b * \klmorphismus a^{\morende\klmorphismus b}
  \end{equation}
\end{definition}

\begin{corollary}
  Every vertex category is a right-groupal category.
\end{corollary}

\noindent In the following we will denote the neutral element with respect to the group operation of a right groupal category by the symbol “1”.

It is a well-known fact that every right po-group can be considered as
a right-groupal category and every po-group can be considered as a
groupal category.

In the same way as with right po-groups, right-groupal categories
are completely defined by the set of arrows that start or end in the
neutral element $1$. For that to prove we use the two operators for a category $\kategorie K$:
\begin{equation}
  \ordnungsideal[\kategorie K]{x} \definiert \bigdisjverein_{y∈\Ob\kategorie K}\Mor{\kategorie K}yx\text{ and }
  \ordnungsfilter[\kategorie K]{x} \definiert \bigdisjverein_{y∈\Ob\kategorie K}\Mor{\kategorie K}xy,
\end{equation}
which can be interpreted as the objects of the slice category and the
coslice category of a given object. Note that both operators
correspond to each other by the duality principle. Both
$\ordnungsideal[\kategorie K]{x}$ and
$\ordnungsfilter[\kategorie K]{x}$ contain the vertex monoid
$\Mor{\kategorie K}xx$ as a subset. Then, $\Morall{\kategorie K}$ is
completely defined by $\Ob\kategorie K$ and
$\ordnungsideal[\kategorie K]1$. The concatenation is defined by the
concatenation on
$\ordnungsideal[\kategorie K]{1}\cup\ordnungsfilter[\kategorie K]{1}$.

\begin{lemma}
  Let $\paar{\kategorie K}{\cdot}$ be a right-groupal category. Then, 
  \begin{equation}\label{eq:27}
    \Mor{\kategorie K}xy = \Mor{\kategorie K}1{y\cdot_{\kategorie K}x^{-1}}\cdot_{\kategorie K} x
    = \Mor{\kategorie K}{x\cdot_{\kategorie K}y^{-1}}1\cdot_{\kategorie K} y,
  \end{equation}
  and for all arrows
  $\klmorphismus a, \klmorphismus b∈\Morall{\kategorie K}$ with
  $\morende \klmorphismus a = \moranfang\klmorphismus b$ there exist
  arrows $\hat{\klmorphismus a}∈\ordnungsideal[{\kategorie K}]1$ and
  $\hat{\klmorphismus b}∈\ordnungsfilter[{\kategorie K}] 1$ such that
  $\klmorphismus a * \klmorphismus b =(\hat{\klmorphismus
    a}*\hat{\klmorphismus b})\cdot \morende\klmorphismus a$.
\end{lemma}
\begin{proof}
  Let $x,y∈\Ob\kategorie K$ then \eqref{eq:27} follows directly from
  the definition.  Now, we consider the
  diagram \begin{tikzcd}x\arrow{r}{\klmorphismus
      a}&y\arrow{r}{\klmorphismus b}&z\end{tikzcd}. This is the same
  as
  $(\begin{tikzcd}x\cdot y^{-1}\arrow{r}{\klmorphismus a\cdot
      {y^{-1}}}&1\arrow{r}{\klmorphismus b\cdot y^{-1}}&z\cdot
    y^{-1}\end{tikzcd})\cdot y$. Consequently
  $\klmorphismus a*\klmorphismus b = (\klmorphismus a\cdot
  y^{-1}*\klmorphismus b \cdot y^{-1})\cdot y$. As
  $\klmorphismus a\cdot y^{-1}∈\ordnungsideal[\kategorie K]{1}$ and
  $\klmorphismus b \cdot y^{-1}∈\ordnungsfilter[\kategorie K]1$ hold,
  the lemma is proved.
  \qed
\end{proof}

\noindent As with groups, also the group action of a right groupal category can
be used to define a group operation that mainly shifts the neutral
element.

\begin{lemma}\label{lemma:vertex-categories-move-neutral-element}
  Let $\kategorie K$ be a right groupal category and $a∈\Ob\kategorie K$
  a vertex. Then the binary operator
  \begin{equation}
    \wabbildung{\abbildung{\cdot_a}{\Ob\kategorie K\times \Ob\kategorie
        K}{\Ob\kategorie K}}{\paar xy}{xa^{-1}y}
  \end{equation}
  is a group operation on $\Ob\kategorie K$ such that
  $\wabbildung{\abbildung{φ}{\kategorie K}{\kategorie K}}{x}{xa}$ is a
  right groupal category isomorphism which maps the group operations
  to each other.
\end{lemma}
\begin{proof}
  First we show that $φ$ is a group isomorphism: It is bijective as
  the group operation is bijective in each argument. Let
  $a,b∈\Ob\kategorie K$. Then
  \begin{align*}
    φ(xy) = xya = xaa^{-1}ya = φ(x)\cdot_aφ(y).
  \end{align*}
  Thus $φ$ is a group isomorphism.

  Let further $\klmorphismus x,\klmorphismus y∈\Morall{\kategorie K}$
  be two arrows such that $\klmorphismus x*\klmorphismus y∈\kategorie K$
  exists. Then both
  \begin{align*}
    (\klmorphismus x*\klmorphismus y)^a=\klmorphismus x^a*\klmorphismus y^a&∈\Morall{\kategorie K},\\
    (\klmorphismus x*\klmorphismus y)^{a^{-1}}=\klmorphismus x^{a^{-1}}*\klmorphismus y^{a^{-1}}&∈\Morall{\kategorie K}\text{ and}\\
    \bigl((\klmorphismus x*\klmorphismus y)^a\bigr)^{a^{-1}}=\klmorphismus x*\klmorphismus y&∈\Morall{\kategorie K}
  \end{align*}
  hold. Thus, the bijection $φ$ is also with respect to the category structure an isomorphism.
  \qed
\end{proof}

\noindent The group operations of two vertex categories of the same
category with the same translative automorphism group can be chosen in
a way that is compatible with the translative group action. So, we can
consider the vertex categories of all vertices as isomorphic right
groupal cateogies.

\begin{lemma}\label{lemma:vertex-categories-isomorphic-right-groupal-categories}
  Let $\kategorie K$ and $\kategorie L$ be right groupal categories and
  $\abbildung{φ}{\kategorie K}{\kategorie L}$ a category isomorphism
  between them with the following properties:
  \begin{align}
    φ(xy) = φ(x)\bigl(φ(1)\bigr)^{-1}φ(y).\label{eq:13}
  \end{align}

  Then, the mapping
  $\wabbildung{\abbildung{ψ}{\kategorie K}{\kategorie L}}{x}{φ(1)}$ is
  an isomorphism of right groupal categories.
\end{lemma}

\begin{proof}
  The equation \eqref{eq:13} can be rewritten in the form:
  \begin{align*}
    φ(xy) = φ(x) \cdot_{φ(1)} φ(y)
  \end{align*}
  That means that $φ$ is also a group isomorphism with respect to the
  group operation $\cdot_{φ(1)}$ on $\Ob\kategorie L$. Thus $φ$ is a
  right groupal category isomorphism between $\kategorie K$ with the
  standard group operation and $\kategorie L$ with the group operation
  $\cdot_{φ(1)}$.

  Let
  $\wabbildung{\abbildung{ψ}{\kategorie L}{\kategorie L}}{x}{xφ(1)}$
  the right groupal isomorphism on $\kategorie L$ that makes $φ(1)$ the neutral
  element. Then $φ\circ ψ^{-1}$ is a right groupal isomorphism between
  $\kategorie K$ and $\kategorie L$ both with their standard group
  operation.
  \qed
\end{proof}

\noindent For each orbit we can find a representative, so that the group
operation created above coincides with the normal group operation of a
vertex category.

\begin{lemma}\label{lemma:vertex-category-1}
  Let $\kategorie K$ be a category and $\gruppe \leq\Aut\kategorie K$ a
  translative automorphism group. Then for any two vertices
  $a,b∈\Ob\kategorie K$ a third vertex $c∈b^{\gruppe}$ exists such
  that there exists a right groupal category isomorphism between the
  vertex categories $\kategorie G_a$ and $\kategorie G_c$ that
  commutes with the automorphisms from $\gruppe$.
\end{lemma}
\begin{proof}
  Since $\gruppe$ is translative, there exists a category isomorphism
  $\abbildung{φ}{\kategorie G_b}{\kategorie G_a}$. It induces a group
  structure on $\kategorie G_b$, whose neutral element will be denoted
  by $c$ and the group operation by $\cdot$. It remains to show that
  the induced group structure is the same as in $\kategorie G_c$. Let
  $x,y ∈\Ob\kategorie G_c$ two vertices in $\kategorie G_c$. Then,
  there exist two elements $g,h∈\gruppe$ with $c^g = x$ and $c^h=y$,
  such that
  \begin{align*}
    x\cdot_{\kategorie G_c}y
    &= c^g\cdot_{\kategorie G_c}c^h = c^{gh} = φ^{-1}\bigl(φ(c^{gh})\bigr)
      = φ^{-1}(a^{gh}) = φ^{-1}(a^g\cdot_{\kategorie G_a}a^h)\\
    &= φ^{-1}(a^g)\cdot φ^{-1}(a^h) = c^g \cdot c^h\\
    &= x \cdot y.
  \end{align*}
  \qed
\end{proof}
\begin{proof}
  Let $a,b∈\Ob\kategorie K$ be two vertices. As $\gruppe$ acts
  translative on $\kategorie K$, there exists a category isomorphism
  $φ$ between $\kategorie G_a$ and $\kategorie G_b$ which commutes with the
  group action. That means for any two group elements
  $g,h∈\gruppe$ the following equations hold:
  \begin{align*}
    φ(a^ga^h) &= φ(a)^{gh}
    \intertext{Let $b^f = c \definiert φ(a)$:}
                φ(a^ga^h) &= b^{fgh}= c^gc^{-1}c^h\\
              &= φ(a^g)φ(a)^{-1}φ(a^h).
  \end{align*}
  As $a$ is the neutral element in $\kategorie G_a$, the previous
  lemmas~\ref{lemma:vertex-categories-move-neutral-element}
  and~\ref{lemma:vertex-categories-isomorphic-right-groupal-categories}
  ensure the existence of a right groupal category isomorphism between
  the vertex categories $\kategorie G_a$ and $\kategorie G_c$.  Now, we
  look at the group action of $h∈\gruppe$. Let $x=a^g∈a^{\gruppe}$,
  then
  \begin{align*}
    φ\bigl((a^g)^h\bigr) &= φ(a^{gh}) = φ(a)^{gh}= (c^g)^h= c^gc^{-1}c^h.
  \end{align*}
  \qed
\end{proof}

\noindent
As we can see in the last proof the role of the orbit vertices with
respect to the automorphism action of the group depends on the choice
of the neutral element. However, the action of the automorphism group
does not depend on that choice.

\section{Flat orbit categories}\label{sec:flat-orbit-categ}
In some cases the orbit categories contain lots of redundant
information. For example for po-groups we are often more interested in
the local structure than in far relationships. As the latter are mainly a
result of transitivity, we may omit them and reconstruct them from
transitivity.

As all vertex monoids of the orbit category are isomorphic, every
orbit category is the homomorphic image of the product of a vertex
monoid with some other category. Since a translative group acts
semi-regular on vertices, we can encode the information of the vertex
monoids in the automorphism group.

In analogy to normal sub(semi)groups we can also consider vertex
monoids as normal, if we find a way to exchange the operands of the
concatenation.

\begin{definition}\label{def:normal-group-action-category}
  Let $\kategorie K$ denote a category and
  $\gruppe\leq\Aut\kategorie K$ be a translative automorphism group.
  Then, $\gruppe$ is called \defindex{right-normal} on $\kategorie K$
  (in symbols: $\gruppe \normalteilervon \kategorie K$), iff for every
  two arrows $\klmorphismus a ∈\Morall{\kategorie K}$ and
  $\klmorphismus x∈\Morall{\kategorie G_{\moranfang \klmorphismus a}}$
  there exists an arrow
  $\klmorphismus C\paar{\klmorphismus a^{\gruppe}}{\klmorphismus
    x^{\gruppe}}∈\Morall{\orbitfaltigkeit{\kategorie G_{\morende
        \klmorphismus a}}{\gruppe}}$ such that in the orbit category
  $\orbitfaltigkeit{\kategorie K}{\gruppe}$ the following equations
  hold:
  \begin{subequations}
    \begin{align}
      \label{eq:12a}
      \klmorphismus x^{\gruppe}*\klmorphismus a^{\gruppe}
      &=\klmorphismus a^{\gruppe}*\klmorphismus C\paar{\klmorphismus a^{\gruppe}}{\klmorphismus x^{\gruppe}},\\
      \id_{\morende\klmorphismus a}^{\gruppe}
      &= \klmorphismus C\paar{\klmorphismus a^{\gruppe}}{\id_{\moranfang\klmorphismus a}^{\gruppe}}\label{eq:36}.
    \end{align}
  \end{subequations}
\end{definition}

We can express right normal group actions in terms of the underlying
category.

\begin{lemma}\label{lemma:normal-auch-representant}
  Let $\kategorie K$ denote a category and $\gruppe\leq\Aut\kategorie K$ be a
  translative automorphism group.  Then, $\gruppe$ is right-normal,
  iff for every two arrows $\klmorphismus a ∈\Morall{\kategorie K}$
  and
  $\klmorphismus x∈\Morall{\kategorie G_{\moranfang \klmorphismus a}}$
  with
  $\morende\klmorphismus x = \moranfang \klmorphismus a$ there exists
  an automorphism $g∈\gruppe$ and an arrow
  $\klmorphismus C'\paar{\klmorphismus a}{\klmorphismus
    x}∈\Morall{\kategorie G_{\morende \klmorphismus a}}$
  such that
  \begin{subequations}
    \begin{align}
      \label{eq:12b}
      \klmorphismus x*\klmorphismus a
      &=\klmorphismus a^{g}*\klmorphismus C'\paar{\klmorphismus a}{\klmorphismus x},\\
      \id_{\morende\klmorphismus a}
      &= \klmorphismus C'\paar{\klmorphismus a}{\id_{\moranfang\klmorphismus a}}
        \label{eq:36b}.
    \end{align}
  \end{subequations}
\end{lemma}

\begin{proof}
  Let first \eqref{eq:12a} be true. As a semi-regular group action is
  regular on its orbits there exist unique homomorphisms $g,h∈\gruppe$ such
  that the concatenation $\klmorphismus x*\klmorphismus a^h$ exists
  and the equation
  $\moranfang\klmorphismus x=\moranfang\klmorphismus a^g$ holds. Let
  $\klmorphismus y∈\klmorphismus C\paar{\klmorphismus
    a^{\gruppe}}{\klmorphismus x^{\gruppe}}$. Then, an automorphism
  $f∈\gruppe$ exists such that the concatenation
  $\klmorphismus a^g*\klmorphismus y^f$ exists. As $\gruppe$ acts
  translative, the orbit
  $\klmorphismus a^{\gruppe}*\klmorphismus y^{\gruppe}$ has exactly
  one arrow starting in $\moranfang \klmorphismus x$, which is by
  construction the arrow $\klmorphismus a^g*\klmorphismus y^f$. So we
  can define
  $\klmorphismus C'\paar{\klmorphismus a^h}{\klmorphismus x}\definiert
  \klmorphismus y^f$, which fulfils the condition \eqref{eq:12b}.

  Conversely suppose Equation~\eqref{eq:12b} holds under the mentioned conditions.
  As $\gruppe$ is translative there exists a choice function
  $\abbildung{c}{\Morall{\orbitfaltigkeit{\kategorie
        K}{\gruppe}}^2}{\Morall{\kategorie K}^2}$ such that for
  $\paar{\hat{\klmorphismus a}}{\hat{\klmorphismus x}} = c\paar{\klmorphismus
    a^{\gruppe}}{\klmorphismus x^{\gruppe}}$ with
  $\klmorphismus
  x^{\gruppe}∈\Morall{\erzkategorie{\kategorie K}{\moranfang\klmorphismus
      a^{\gruppe}}}$ the equation
  $\morende\hat{\klmorphismus x}=\moranfang\hat{\klmorphismus a}$ holds.

  For two arrows
  $\klmorphismus a^{\gruppe}\Morall{\orbitfaltigkeit{\kategorie
      K}{\gruppe}}$ and
  $\klmorphismus x^{\gruppe}∈\Mor{\orbitfaltigkeit{\kategorie
      K}{\gruppe}}{\moranfang\klmorphismus
    a^{\gruppe}}{\moranfang\klmorphismus a^{\gruppe}}$ we chose a pair
  of arrows
  $\paar{\hat{\klmorphismus a}}{\hat {\klmorphismus x}} = c
  \paar{\klmorphismus a^{\gruppe}}{\klmorphismus x^{\gruppe}}$. Then,
  \[\klmorphismus x^{\gruppe}*\klmorphismus a^{\gruppe} = (\hat{\klmorphismus x}*\hat{\klmorphismus a})^{\gruppe}
    =(\hat{\klmorphismus a}^{g}*\klmorphismus C'\paar{\klmorphismus
      a}{\klmorphismus x})^{\gruppe} = \hat{\klmorphismus
    a}^{\gruppe}*\klmorphismus C'\paar{\hat{\klmorphismus
      a}}{\hat{\klmorphismus x}}^{\gruppe} = \klmorphismus a^{\gruppe}*\klmorphismus C'\bigl(c\paar{\klmorphismus a^{\gruppe}}{\klmorphismus x^{\gruppe}}\bigr)^{\gruppe}.\]
  So we can define
  $\klmorphismus C\paar{\kategorie a^{\gruppe}}{\klmorphismus
    x^{\gruppe}}\definiert\klmorphismus C'\bigl(c\paar{\klmorphismus
    a^{\gruppe}}{\klmorphismus x}\bigr)^{\gruppe}$, which fulfils Equation~\eqref{eq:12a}.

  The equations \eqref{eq:36b} and \eqref{eq:36} can be proved in an analogous way.
  \qed
\end{proof}

\begin{corollary}
  If a partial mapping $\klmorphismus C''$ fulfils Equation~\eqref{eq:12b}, then the partial mapping
  \begin{align}
    \klmorphismus C'\paar{\klmorphismus a}{\klmorphismus x}\definiert
    \begin{cases}
      \id_{\morende\klmorphismus a},&\klmorphismus x = \id_{\moranfang\klmorphismus a}\\
      \klmorphismus C''\paar{\klmorphismus a}{\klmorphismus x},&\text{else}
    \end{cases}
  \end{align}
  fulfils Equations~\eqref{eq:12b} and~\eqref{eq:36b}
\end{corollary}

\begin{lemma}\label{lemma:subgroup-of-groupal-is-right-normal}
  Every subgroup of a groupal category is right-normal on its
  category structure.
\end{lemma}

\begin{proof}
  Let $\kategorie K$ a groupal category, $\gruppe\leq\Ob\kategorie K$
  and $\klmorphismus a ∈\Morall{\kategorie K}$
  and
  $\klmorphismus x∈\Morall{\kategorie G_{1}}$. Then,
  \begin{align*}
    {}^{\moranfang\klmorphismus a} \klmorphismus x*\klmorphismus a^{\morende\klmorphismus x}
    &= \klmorphismus a^{\moranfang\klmorphismus x}*{}^{\morende\klmorphismus a}\klmorphismus x\\
    \klmorphismus x*{}^{\moranfang\klmorphismus a^{-1}} \klmorphismus a^{\morende\klmorphismus x}
    &= {}^{\moranfang\klmorphismus a^{-1}} \klmorphismus a^{\moranfang\klmorphismus x}*{}^{\moranfang\klmorphismus a^{-1}\morende\klmorphismus a}\klmorphismus x
      \intertext{Substituting $\klmorphismus b\definiert {}^{\moranfang\klmorphismus a^{-1}} \klmorphismus a^{\morende\klmorphismus x}$ leads to ${}^{\moranfang\klmorphismus a^{-1}}\klmorphismus a = \klmorphismus b^{\morende\klmorphismus x^{-1}}$}
    \klmorphismus x* \klmorphismus b
    &=  \klmorphismus b^{\morende\klmorphismus x^{-1}\moranfang\klmorphismus x}*{}^{\morende\klmorphismus b\morende\klmorphismus x^{-1}}\klmorphismus x=  \klmorphismus b^{\morende\klmorphismus x^{-1}\moranfang\klmorphismus x}*{}^{\morende\klmorphismus b\moranfang\klmorphismus b^{-1}}\klmorphismus x
  \end{align*}
  When we define
  $g\definiert \morende\klmorphismus x^{-1}\moranfang\klmorphismus x$
  and
  $\klmorphismus C'\paar{\klmorphismus b}{\klmorphismus x}\definiert
  {}^{\morende\klmorphismus b\moranfang\klmorphismus
    b^{-1}}\klmorphismus x$, then also
  $\klmorphismus C'\paar{\klmorphismus b}{\id_{\moranfang\klmorphismus
      b}} = {}^{\morende\klmorphismus b\moranfang\klmorphismus
    b^{-1}}\id_{\moranfang\klmorphismus b} =
  \id_{\morende\klmorphismus b}$ holds. Since every arrow
  $\klmorphismus b$ can be represented in the form
  $\klmorphismus b=\hat {\klmorphismus b}^{\morende{\klmorphismus
      x^{-1}\moranfang\klmorphismus x}}$, we can apply
  Lemma~\ref{lemma:normal-auch-representant}.
  \qed
\end{proof}

\begin{corollary}
  Every subgroup of a po-group acts right-normal on its
  category structure.
\end{corollary}

\noindent In many cases normal group actions define isomorphisms
between their vertex categories. As an example we prove it for simple
categories.

\begin{lemma}\label{lemma:13a}
  Let $\kategorie K$ be a category and
  $\gruppe \leq \Aut\kategorie K$ a translative automorphism group,
  and $\klmorphismus C'$ defined as in
  Lemma~\ref{lemma:normal-auch-representant}.  Then for all arrows
  $\klmorphismus a ∈\Morall{\kategorie K}$ and
  $\klmorphismus x,\klmorphismus y∈\kategorie
  G_{\moranfang\klmorphismus a}$, there exist automorphisms
  $g,h∈\gruppe$ such that the following equations hold:
  \begin{subequations}
    \begin{align}
      \klmorphismus a*\klmorphismus C'(\klmorphismus a,\klmorphismus x*\klmorphismus y)
      &=
        \klmorphismus a*\klmorphismus C'(\klmorphismus a,\klmorphismus x) *\klmorphismus C'(\klmorphismus a^h,\klmorphismus y)
      \label{eq:53}\\
      \klmorphismus a
      &=\klmorphismus a * \klmorphismus C'\paar{\klmorphismus a}{\id_{\moranfang\klmorphismus a}}\label{eq:54}\\
      \klmorphismus x
      &=\klmorphismus C'\paar{\id_{\morende\klmorphismus x}}{\klmorphismus x}\label{eq:55}
    \end{align}
  \end{subequations}
\end{lemma}

\begin{proof}
  \begin{align*}
    \klmorphismus a*\klmorphismus C'\paar{\klmorphismus a}{\klmorphismus x*\klmorphismus y}
    = \klmorphismus x * \klmorphismus y *\klmorphismus a ^{g}
    = \klmorphismus x * \klmorphismus a ^h * \klmorphismus C'\paar{\klmorphismus a^h}{\klmorphismus y}
    = \klmorphismus a * \klmorphismus C'\paar {\klmorphismus a}{\klmorphismus x} *
    \klmorphismus C'\paar{\klmorphismus a^h}{\klmorphismus y}
  \end{align*}
  \begin{align*}
    \klmorphismus a = \id_{\moranfang \klmorphismus a}*\klmorphismus a
    = \klmorphismus a * \klmorphismus C'\paar{\klmorphismus a}{\id_{\moranfang\klmorphismus a}}
  \end{align*}
  \begin{align*}
    \klmorphismus x = \klmorphismus x * \id_{\morende\klmorphismus x}
    = \id_{\morende\klmorphismus x} * \klmorphismus C'\paar{\id_{\morende\klmorphismus x}}{\klmorphismus x}
    = \klmorphismus C'\paar{\id_{\morende\klmorphismus x}}{\klmorphismus x}
  \end{align*}
  \qed
\end{proof}

\begin{lemma}\label{lemma:13}
  Let $\kategorie K$ be a category,
  $\gruppe \leq \Aut\kategorie K$ a translative automorphism group,
  $\klmorphismus a∈\Morall{\kategorie K}$ and $\klmorphismus C'$ defined as in
  Lemma~\ref{lemma:normal-auch-representant}.  For every arrow $\klmorphismus x∈\Morall{\kategorie G_{\morende\klmorphismus a}}$ let
  $h_{\klmorphismus a}(\klmorphismus x)\definiert
  \morende{\klmorphismus x}\moranfang\klmorphismus a^{-1}$ denote the
  automorphism that shifts the arrow $\klmorphismus a$ such that
  $\klmorphismus a^{h_{\klmorphismus a}(\klmorphismus
    x)}*\klmorphismus x$ exists.  If the mapping
  \begin{equation}\label{eq:19}
    \wabbildung{s_{\klmorphismus a}}{\klmorphismus y}{\klmorphismus a^{h_{\klmorphismus a}(\klmorphismus y)} * \klmorphismus y}
  \end{equation}
  is injective, Eq.~\eqref{eq:12b} defines a
  category homomorphism
  \begin{equation}
    \wabbildung{\abbildung{φ_{\klmorphismus
          a}}{\Morall{\kategorie G_{\moranfang \klmorphismus
            a}}}{\Morall{\kategorie G_{\morende \klmorphismus
            a}}}}{\klmorphismus x}{\klmorphismus C'(\klmorphismus
      a^{h_{\klmorphismus a}(x)},\klmorphismus x)}.\label{eq:20}  
  \end{equation}
\end{lemma}
\begin{proof}
  Let
  $\klmorphismus x_1,\klmorphismus x_2 ∈ \Morall{\kategorie G_{\moranfang \klmorphismus
      a}}$ be two arrows such that $\klmorphismus x_1*\klmorphismus x_2$ exists. Consequently an automorphism $g∈\gruppe$ exists such that with \ref{eq:53} the following equation holds:
  \begin{align*}
      \klmorphismus a^g *\klmorphismus C'\paar{\klmorphismus a^{h_{\klmorphismus a}(\klmorphismus x_1*\klmorphismus x_2)}}{\klmorphismus x_1 * \klmorphismus x_2}
    &=
  (\klmorphismus x_1*\klmorphismus x_2) * \klmorphismus a =
      \klmorphismus a^g *\bigl(\klmorphismus C'\paar{\klmorphismus a^{h_{\klmorphismus a}(\klmorphismus x_1)}}{\klmorphismus x_1} * \klmorphismus C'\paar{\klmorphismus a^{h_{\klmorphismus a}(\klmorphismus x_2)}}{\klmorphismus x_2}\bigr).
  \end{align*}
  As $s_{\klmorphismus a}$ is injective, this implies
  $φ_{\klmorphismus a}({\klmorphismus x_1 * \klmorphismus x_2})
  =
  φ_{\klmorphismus a}({\klmorphismus x_1}) * φ_{\klmorphismus a}({\klmorphismus x_2})$.
 So 
  the mapping $φ_{\klmorphismus a^{\gruppe}}$ is a category
  homomorphism.
  \qed
\end{proof}
\begin{corollary}
  If $\klmorphismus a$ is an epimorphism, $φ_{\klmorphismus a}$ is a homomorphism.
\end{corollary}
In general, the inverse implication is not true, as we are considering
only a subcategory, here.

\begin{lemma}
  In Lemma~\ref{lemma:13}
  $φ_\klmorphismus a$ commutes with every automorphism
  in $\gruppe$.
\end{lemma}

\begin{proof}
  Let $g∈\gruppe$ be an automorphism. We use that the group action is
  semi-regular, so that distinct arrows of an orbit have distinct
  starting points. Following the notation of Lemma~\ref{lemma:13}, two
  automorphisms $\hat g,\hat{\hat g}∈\gruppe$ exist such that the
  application of equation~\eqref{eq:12b} leads to
  \begin{align*}
    \klmorphismus a^{h_{\klmorphismus a}(\klmorphismus x^g)\hat g}*\klmorphismus C'\paar{\klmorphismus a}{\klmorphismus x^g}
    &=\klmorphismus x^g*\klmorphismus a^{h_{\klmorphismus a}(\klmorphismus x^g)}
      =(\klmorphismus x*\klmorphismus a^{h_{\klmorphismus a}(\klmorphismus x)})^g
    =\klmorphismus a^{h_{\klmorphismus a}(\klmorphismus x)\hat{\hat g}g}*\klmorphismus C'\paar{\klmorphismus a}{\klmorphismus x}^g\\
    &=\klmorphismus a^{h_{\klmorphismus a}(\klmorphismus x^g)\hat g}*\klmorphismus C'\paar{\klmorphismus a}{\klmorphismus x}^g,
  \end{align*}
  \noindent since $\moranfang\klmorphismus a^{h_{\klmorphismus a}(\klmorphismus x)\hat{\hat g}g} =
  \moranfang\klmorphismus a^{h_{\klmorphismus a}(\klmorphismus x^g)\hat g}$. This implies
  \(
    φ_{\klmorphismus a}({\klmorphismus x^g})
    =φ_{\klmorphismus a}({\klmorphismus x})^g.
    \)
  \qed
\end{proof}

\noindent For groupal categories we may use the left multiplicaton for the construction of $\klmorphismus C$ in Eq.~\eqref{eq:12b} and \eqref{eq:36b}, which is shown by the following two lemmas.

\begin{lemma}\label{lemma:linksmult-ist-phi}
  Let $\kategorie K$ denote a groupal category and
  $\gruppe \leq \Ob\kategorie K$ be a translative right-associative
  automorphism group, then for any two vertices $x,y∈\Ob\kategorie K$
  the mapping
  \begin{equation}
    \label{eq:21}
    \wabbildung{\abbildung{φ_{x,y}(\klmorphismus x)}{\kategorie G_x}{\kategorie G_y}}{\klmorphismus x}{{}^{yx^-1}\klmorphismus x}
  \end{equation}
  is an isomorphism that commutes with the right group action of  $\gruppe$. 
\end{lemma}
\begin{proof}
  As $\kategorie K$ is groupal, the mapping $φ_{x,y}$ can be extended
  to an automorphism on $\kategorie K$ so it is  injective and invertible if it is well-defined.
  It remains to show that it is well-defined and surjective.
  
  The mapping is well-defined. Let $z∈x^{\gruppe}$
  then ${}^{yx^{-1}}z ∈ {}^{yx^{-1}}x^{\gruppe} = y^{\gruppe}$, which
  follows from the group properties of $\Ob\kategorie K$.

  In the same way the mapping $φ_{y,x}$ is well-defined, injective and
  invertible. Furthermore for
  $\klmorphismus y∈\Morall{\kategorie G_y}$
  \[
    φ_{x,y}\bigl(φ_{y,x}(\klmorphismus x)\bigr)
    = φ_{x,y}({}^{xy^{-1}}\klmorphismus y)
    = {}^{yx^{-1}xy^{-1}}\klmorphismus y = \klmorphismus y.
  \]
  So $φ_{x,y}$ is surjective and $φ_{y,x}$ is its inverse.
  
  As $φ_{x,y}$ is defined by applying a left group action of the
  objects, it commutes by definition with the right group action.
  \qed
\end{proof}

\begin{lemma}
  In Lemma~\ref{lemma:linksmult-ist-phi} for any arrow $\klmorphismus x∈\ordnungsfilter[\kategorie G_x]x$, any automorphism $g∈\gruppe$ and any arrow $\klmorphismus a∈\Mor{\kategorie K}{x}{y}$ the following equation holds:
  \begin{equation}
    \label{eq:22}
    \klmorphismus x^g * \klmorphismus a^{h_a(\klmorphismus x^g)}
    = \bigl(\klmorphismus a* φ_{x,y}(\klmorphismus x)\bigr)^g
  \end{equation}
  
\end{lemma}

\begin{proof}
  As we consider the categories $\kategorie G_x$ and $\kategorie G_y$
  the vertices $x$ and $y$ play the role of the neutral element with
  respect to the group operation. Furthermore $\moranfang \klmorphismus x = x = 1_{\kategorie G_x}$. So we get 
  \begin{align*}
    (\klmorphismus x*\klmorphismus a ^{\moranfang\klmorphismus a^{-1}\morende\klmorphismus x})^g
    &= 
      \klmorphismus a^{\moranfang\klmorphismus a^{-1}g}
      * {}^{\morende\klmorphismus a(\moranfang\klmorphismus a)^{-1}}(\klmorphismus x)^{g}\\
    &= \klmorphismus a^{\moranfang\klmorphismus a^{-1}g}
      * φ_{\moranfang\klmorphismus a,\morende\klmorphismus a}(\klmorphismus x)^{g}\\
    &= \klmorphismus a^{\moranfang\klmorphismus a^{-1}g}
      * φ_{x,y}(\klmorphismus x)^g\\
  \end{align*}
  \qed
\end{proof}
\begin{lemma}
  Let $\kategorie K$ be a category,
  $\gruppe[S]\leq \gruppe\leq\Aut\kategorie K$ two translative
  automorphism groups with $\gruppe\normalteilervon \kategorie
  K$. Then, $\gruppe[S]\normalteilervon \kategorie K$.
\end{lemma}
\begin{proof}
  Since $\gruppe\normalteilervon \kategorie K$, by
  Lemma~\ref{lemma:normal-auch-representant} for every arrow
  $\klmorphismus a∈\Morall{\kategorie K}$ and
  $\klmorphismus x∈\gruppe[S]_{\moranfang \klmorphismus a}$
  there exist automorphisms $g∈\gruppe[S]$, $h∈{\gruppe}$
  and an arrow
  $\klmorphismus C'\paar{\klmorphismus a^g}{\klmorphismus
    x}∈\gruppe_{\moranfang \klmorphismus a}$ such that
  \[\klmorphismus x*\klmorphismus a^g = \klmorphismus a^h*\klmorphismus
    C'\paar{\klmorphismus a^g}{\klmorphismus x}.\tag{*}\] Since
  $\moranfang \klmorphismus x = \moranfang \klmorphismus a^h$ and
  $\morende\klmorphismus x = \moranfang \klmorphismus a$ holds, we
  derive the equation
  $h = g\multnegpot{(\morende{\klmorphismus
      x})}\moranfang\klmorphismus x∈\gruppe[S]$. For start and
  end of $\klmorphismus C'\paar{\klmorphismus a^g}{\klmorphismus x}$
  we get the equations
  $\moranfang\klmorphismus C'\paar{\klmorphismus a^g}{\klmorphismus x}
  = \morende\klmorphismus a^h = \morende\klmorphismus
  a^{g\multnegpot{(\morende{\klmorphismus x})}\moranfang\klmorphismus
    x}$ and
  $\morende\klmorphismus C'\paar{\klmorphismus a^g}{\klmorphismus x} =
  \morende\klmorphismus a^h$. So
  $\moranfang\klmorphismus C'\paar{\klmorphismus a^g}{\klmorphismus x}
  ∈\morende \klmorphismus C'\paar{\klmorphismus a^g}{\klmorphismus
    x}^{\gruppe[S]}$, which means
  $\klmorphismus C'\paar{\klmorphismus a^g}{\klmorphismus
    x}∈\gruppe[S]_{\morende\klmorphismus a^g}$. So we can apply
  Lemma~\ref{lemma:normal-auch-representant} with respect to the group
  $\gruppe[S]$ to show that there exists a partial mapping
  $\klmorphismus C$ such that
  \begin{align*}
    \klmorphismus x^{\gruppe[S]}*\klmorphismus a^{\gruppe[S]}
    &= \klmorphismus a^{\gruppe[S]} * \klmorphismus C\paar{\klmorphismus a^{\gruppe[S]}}{\klmorphismus x}.
      \intertext{With the same lemma, we may shift Equation~\eqref{eq:36} to \eqref{eq:36b} and back to prove}
    \id_{\morende\klmorphismus a}^{\gruppe[S]}
    &= \klmorphismus C\paar{\klmorphismus a}{\id_{\moranfang\klmorphismus a}}.
  \end{align*}
  So we have shown $\gruppe[S]\normalteilervon\kategorie K$.
  \qed
\end{proof}

\noindent As we have seen, we can replace loops on the left hand side of a
concatenation in the orbit category by loops concatenated from the
right. Now, we introduce a possibility to use this knowledge to
transfer the vertex monoids to the unfolding group. We must identify
those arrows that have to be preserved in our simplified orbit
category.

\begin{definition}
  % Gruppale Kategorien haben zu $a\to b$ auch $a\to a^{-1}ba$ und
  % $a\to ab a^{-1}$.
  
  Let $\kategorie K$ be a category and $x,y∈\Ob\kategorie K$ vertices of
  $\kategorie K$. We say an arrow
  $\klmorphismus a∈\Mor{\kategorie K}xy$ is \defindex{reducible} iff
  arrows $\klmorphismus b∈\Mor{\kategorie K}xy$ and
  $\klmorphismus y∈\Mor{\kategorie K}yy\setminus\menge{\id_x}$ exist
  such that
  \begin{equation}
    \klmorphismus a = \klmorphismus b *\klmorphismus
    y.\label{eq:17}
  \end{equation}
  Otherwise $\klmorphismus a$ is called \defindex{irreducible}.
  
  If every non-loop arrow of $\kategorie K$ is the concatenation of an
  irreducible arrow and a loop, the category is called
  \defindex{representable by irreducible arrows}. It is called
  \defindex{uniquely representable by irreducible arrows}, iff for
  each arrow $\klmorphismus a∈\Morall{\kategorie K}$ there exists
  exactly one representation in the form \eqref{eq:17}.
\end{definition}
\begin{corollary}
  All non-identity loops are reducible.
\end{corollary}
\begin{corollary}
  Identity loops are irreducible.
\end{corollary}

\begin{lemma}\label{lemma:15}
  Let $\Halbord{\gruppe}{\leq}$ be an $\ell$-group with a lattice
  ordered normal subgroup $N\normalteilervon\gruppe$. Then for each
  arrow orbit $\klmorphismus a$ in the preceding definition there exists at
  most one irreducible arrow $\klmorphismus b$ such that \eqref{eq:17}
  holds. In that case  arrow $\klmorphismus y$ in that equation is unique.
  iff $\klmorphismus b$ exists.
\end{lemma}
\begin{proof}
  Suppose we have two descriptions
  $\klmorphismus a^{\gruppe} = \klmorphismus b^{\gruppe} *
  \klmorphismus y^{\gruppe} = \klmorphismus c^{\gruppe} *
  \klmorphismus z^{\gruppe}$, with arrows
  $\klmorphismus y^{\gruppe}, \klmorphismus
  z^{\gruppe}∈\Mor{\orbitfaltigkeit{\Halbord{\gruppe}{\leq}}N}{\morende\klmorphismus
    a^{\gruppe}}{\morende \klmorphismus a^{\gruppe}}$. W.l.o.g.{} the
  arrows can be chosen such that
  $\moranfang \klmorphismus a = \moranfang \klmorphismus b =
  \moranfang\klmorphismus c$. Then, the relations
  $\moranfang \klmorphismus a \leq \morende \klmorphismus b$ and
  $\moranfang \klmorphismus a\leq \morende \klmorphismus c$
  hold. Consequently an arrow $\klmorphismus d$ exists with
  $\moranfang \klmorphismus a = \moranfang\klmorphismus d$ and
  $\morende \klmorphismus d \definiert \morende \klmorphismus b
  \schnitt \morende \klmorphismus c\geq \moranfang \klmorphismus a$.

  On the other hand there exists an element $n∈N$ such that
  $\morende\klmorphismus b^n = \morende\klmorphismus c$ so that we get
  \begin{align*}
    \morende \klmorphismus d = \morende \klmorphismus b \schnitt \morende\klmorphismus c
    = \morende \klmorphismus b \schnitt \morende\klmorphismus b^n
    = \morende \klmorphismus b (1\schnitt n) \geq \moranfang \klmorphismus a.
  \end{align*}
  As $\Halbord{N}{\leq\eingeschrmenge N}$ is a sublattice of
  $\Halbord{\gruppe}{\leq}$, $1\schnitt n∈N$. With $1\schnitt n\leq 1$,
  it follows that $1\vereinigung n^{-1}=(1\schnitt n)^{-1}\geq 1$ and
  $\klmorphismus b = \klmorphismus d * \paar[big]{\morende
    \klmorphismus b (1 \schnitt n)}{\morende \klmorphismus b}$, where
  $\paar[big]{\morende\klmorphismus b(1\schnitt
    n)}{\morende\klmorphismus b}∈\kategorie G_{\morende\klmorphismus
    b}$. So $\klmorphismus d=\klmorphismus b$ or $\klmorphismus b$ is
  not irreducible. The same holds for $\klmorphismus c$. So, if
  $\klmorphismus b$ and $\klmorphismus c$ are irreducible,
  $\klmorphismus b = \klmorphismus c$ holds. As there is only one
  arrow between two vertices, the arrow
  $\paar[big]{\morende\klmorphismus b(1\schnitt n)}{\morende
    \klmorphismus a}$ is uniquely defined.
  \qed
\end{proof}

\noindent This lemma covers an important class of po-groups. However, it
depends on the uniqueness of certain greatest lower bounds. So we can
easily rephrase this lemma in terms of category theory and finite products.

The statement of Lemma~\ref{lemma:15} is not true in general for
normal subcategories. Take for example the po-group
$\ganzzahlen\times \ganzzahlen$ with the product order and the normal
subgroup $N$ that is generated by the set $\menge{\paar {1}{1},\paar 1{-1}}$. Then the
arrow $\paar[big]{\paar 00}{\paar 21}^N$ can be reduced to both
$\paar[big]{\paar 00}{\paar 01}^N$ and
$\paar[big]{\paar 00}{\paar 10}^N$ but none can be reduced to the
other since $\paar 01$ and $\paar 10$ are incomparable.

The uniqueness of irreducible arrows suggests that we can use for a
further compression of the orbit categories in some cases. For lattice
po-groups such a representation is unique. Before we can do that
we need some partial operations:
\begin{definition}\label{def:r-n}
  Let $\kategorie K$ be a category,
  $\gruppe \leq\Aut\kategorie K$ a translative automorphism group such that
  $\orbitfaltigkeit{\kategorie K}{\gruppe}$ is uniquely representable
  by irreducible arrows. Then we denote by
  \begin{align}
    \abbildung{\klmorphismus r}{\Morall{\orbitfaltigkeit{\kategorie K}{\gruppe}}}
    {\Morall{\orbitfaltigkeit{\kategorie K}{\gruppe}}}
  \end{align}
  the partial mapping that maps each arrow into the set of irreducible arrows, and 
  \begin{align}
    \abbildung{\klmorphismus n}{\Morall{\orbitfaltigkeit{\kategorie K}{\gruppe}}}
    {\Morall{\orbitfaltigkeit{\kategorie K}{\gruppe}}}
  \end{align}
  the complementary partial mapping such that for every reducible $\klmorphismus x∈\Morall{\orbitfaltigkeit{\kategorie K}{\gruppe}}$
  arrow the following equation holds:
  \begin{align}
    \klmorphismus x = \klmorphismus r(\klmorphismus x) * \klmorphismus n (\klmorphismus x).
  \end{align}
\end{definition}
\begin{corollary}\label{corr:r-n-1}
  Let $\Halbord{\gruppe}{\leq}$ denote an $\ell$-group with a normal lattice ordered subgroup
  $N\normalteilervon\gruppe$. Then, for each arrow $\klmorphismus a$
  and every arrow
  $\klmorphismus n∈\Mor{\orbitfaltigkeit
    {\Halbord{\gruppe}{\leq}}{N}}{\morende\klmorphismus
    a}{\morende\klmorphismus a}$ the following equation holds:
  \begin{align}
    \klmorphismus r(\klmorphismus a) = \klmorphismus r(\klmorphismus a * \klmorphismus n)
  \end{align}
  if $\klmorphismus a *\klmorphismus n$ is definied and $\klmorphismus r \klmorphismus a$ and $\klmorphismus r(\klmorphismus a * \klmorphismus n)$ exist.
\end{corollary}

\begin{corollary}\label{corr:Rechenregeln-r-n}
  For the mappings $\klmorphismus r$ and $\klmorphismus n$ from Definition~\ref{def:r-n} the following holds:
  \begin{subequations}
    \begin{align}
      \klmorphismus r(\id_x)
      &= \id_x
      &  \klmorphismus n(\id_x)
      &= \id_x\label{eq:37}\\
      \klmorphismus r(\klmorphismus a * \id_{\morende\klmorphismus a})
      &= \klmorphismus r(\klmorphismus a)
      &\klmorphismus n(\klmorphismus a * \id_{\morende\klmorphismus a})
      &= \klmorphismus n(\klmorphismus a)\label{eq:38}\\
      \klmorphismus r(\id_{\moranfang\klmorphismus a} * \klmorphismus a )
      &= \klmorphismus r(\klmorphismus a)
      &\klmorphismus n(\id_{\morende\klmorphismus a} * \klmorphismus a )
      &= \klmorphismus n(\klmorphismus a)\label{eq:39}\\
      &&
         \klmorphismus n(\klmorphismus r \klmorphismus a) & = \klmorphismus \id_{\morende\klmorphismus a}\label{eq:40}
    \end{align} 
  \end{subequations}
\end{corollary}

\noindent As a next step we will construct a concatenation based on a
category $\kategorie K$ that is uniquely representable by irreducible
arrows, so that $\klmorphismus r[\kategorie K]$ equipped with this
concatenation is a category again.

\begin{lemma}\label{lemma:eind.reduzibel-phi-injektiv}
  Under the conditions of Lemma~\ref{lemma:13}, if
  $\orbitfaltigkeit {\kategorie K}{\gruppe}$ is uniquely representable by
  irreducible arrows then for each arrow
  $\klmorphismus a∈\Morall{\kategorie K}$ the homomorphism
  $φ_{\klmorphismus r\klmorphismus a}$ from \eqref{eq:20} is
  injective.
\end{lemma}

\begin{proof}
  Let
  $\klmorphismus x,\klmorphismus y∈\Mor{\orbitfaltigkeit{\kategorie
      K}{\gruppe}} {\morende\klmorphismus a}{\morende\klmorphismus
    a}$. If
  $\klmorphismus r\klmorphismus a *\klmorphismus x = \klmorphismus r\klmorphismus a*\klmorphismus
  y$, then $\klmorphismus r\klmorphismus a*\klmorphismus x$ is not uniquely representable or
  $\klmorphismus x = \klmorphismus y$.
  \qed
\end{proof}

\begin{lemma}\label{lemma:16}
  Let $\kategorie K$ be a category,
  $\gruppe\leq\Aut \kategorie K$ a translative automorphism group that acts normal on
  $\kategorie K$ such that
  $\orbitfaltigkeit{\kategorie K}{\gruppe}$ is uniquely representable
  by irreducible arrows.
  Then the partial mapping
  \begin{align}
    \wabbildung{\abbildung \bullet {\Morall{\orbitfaltigkeit{\kategorie K}{\gruppe}}\times\Morall{\orbitfaltigkeit{\kategorie K}{\gruppe}}}{\Morall{\orbitfaltigkeit{\kategorie K}{\gruppe}}}}{\paar{\klmorphismus x}{\klmorphismus y}}{\klmorphismus r(\klmorphismus x * \klmorphismus y)}
  \end{align}
  is a category concatenation.
\end{lemma}
\begin{proof}
  Using  Corollary~\ref{corr:r-n-1} and Definition~\ref{def:normal-group-action-category} we know:
  $\klmorphismus y∈\Mor{\orbitfaltigkeit{\kategorie
      K}{\gruppe}}{\morende \klmorphismus x}{\morende
    \klmorphismus x}$ that
  $\klmorphismus x\bullet\klmorphismus y=\klmorphismus
  r(\klmorphismus x)$.
  
  In general, 
  \begin{align*}
    \klmorphismus x\bullet\klmorphismus y
    &= \klmorphismus r(\klmorphismus r\klmorphismus x * \klmorphismus n\klmorphismus x*\klmorphismus r\klmorphismus y * \klmorphismus n\klmorphismus y)\\
    &= \klmorphismus r(\klmorphismus r\klmorphismus x *\klmorphismus r\klmorphismus y * \klmorphismus C'\paar{\klmorphismus r\klmorphismus y}{\klmorphismus n\klmorphismus x}* \klmorphismus n\klmorphismus y)\\
    &= \klmorphismus r(\klmorphismus r\klmorphismus x * \klmorphismus r\klmorphismus y)\\
    &= \klmorphismus r\klmorphismus x \bullet \klmorphismus r\klmorphismus y
  \end{align*}
  Consequently we can prove the associativity:
  \begin{align*}
    (\klmorphismus x\bullet\klmorphismus y) \bullet\klmorphismus z
    &= \klmorphismus r\bigl(\klmorphismus r(\klmorphismus r\klmorphismus x * \klmorphismus r\klmorphismus y) * \klmorphismus r\klmorphismus z\bigr)\\
    &= \klmorphismus r\bigl(\klmorphismus r(\klmorphismus r\klmorphismus x * \klmorphismus r\klmorphismus y) * \klmorphismus r(\klmorphismus r\klmorphismus z)\bigr)\\
    &= \klmorphismus r\bigl((\klmorphismus r\klmorphismus x * \klmorphismus r\klmorphismus y) * \klmorphismus z\bigr)\\
    &= \klmorphismus r\bigl(\klmorphismus r\klmorphismus x * (\klmorphismus r\klmorphismus y * \klmorphismus z)\bigr)\\
    &= \klmorphismus r\bigl(\klmorphismus r(\klmorphismus r\klmorphismus x) * (\klmorphismus r\klmorphismus y * \klmorphismus r\klmorphismus z)\bigr)\\
    &= \klmorphismus x\bullet(\klmorphismus y \bullet\klmorphismus z).
  \end{align*}
  \qed
\end{proof}
\begin{corollary}
  The mapping $\klmorphismus r$ is a category homomorphism from
  $\kategorie K$ with the concatenation $*$ to
  $\klmorphismus r[\kategorie K]$ with the concatenation $\bullet$.
\end{corollary}

\begin{definition}
  Let $\kategorie K$ be a category and $\gruppe\leq\Aut\kategorie K$ a
  translative automorphism group that acts normal on $\kategorie
  K$. Let further $\orbitfaltigkeit{\kategorie K}{\gruppe}$
  representable by irreducible arrows. Then we call the category
  $\klmorphismus r[\orbitfaltigkeit{\kategorie K}{\gruppe}]$ with the
  concatenation from Lemma~\ref{lemma:16} \defindex{flat orbit
    category} of $\kategorie K$.
\end{definition}

\noindent In a similar way as the orbit categories, also flat orbit categories
can be unfolded. and both unfoldings lead to isomorphic
categories. This isomorphism is presented here, while more information
about the resulting representations is given in the next section.

\begin{lemma}\label{lemma:18}
  Let $\kategorie K$ be a category, $\gruppe\leq \Aut\kategorie K$ a translative automorphism group that acts normal on $\kategorie K$ such that $\orbitfaltigkeit{\kategorie K}{\gruppe}$ is uniquely representable by irreducible arrows, $\tupel{\orbitfaltigkeit{\kategorie K}{\gruppe},A,\gruppe}$ a representation of $\kategorie K$,
  and $\klmorphismus C$ the mapping from Definition~\ref{def:normal-group-action-category}. 
  Under the conditions of Lemma~\ref{lemma:16} let
  \begin{equation}
    \label{eq:24}
    \wabbildung{\abbildung{\hat{\klmorphismus n}}{\Morall{\kategorie K}\times\Morall{\kategorie K}}{\Morall{\kategorie K}}}
    {\paar{\klmorphismus a}{\klmorphismus b}}{\klmorphismus n(\klmorphismus a*\klmorphismus b)}.
  \end{equation}
  Then, the mapping 
  \begin{equation}
    \wabbildung{\abbildung{φ}{\klmorphismus r[\orbitfaltigkeit{\kategorie K}{\gruppe}]\times \orbitfaltigkeit{\erzkategorie{\kategorie K}{\gruppe}}{\gruppe}}{\orbitfaltigkeit{\kategorie K}{\gruppe}}}
    {\paar {\klmorphismus x}{\klmorphismus y}}{\klmorphismus x*\klmorphismus y}\label{eq:29}
  \end{equation}
  from the orbit category $\orbitfaltigkeit {\kategorie K}{\gruppe}$ into the product category
  $\klmorphismus r[\orbitfaltigkeit{\kategorie K}{\gruppe}]\times \orbitfaltigkeit{\erzkategorie{\kategorie K}{\gruppe}}{\gruppe}$ with the concatenation
  \begin{equation}
    \label{eq:18}
    \paar{\klmorphismus r\klmorphismus a}{\klmorphismus x}*\paar{\klmorphismus r\klmorphismus b}{\klmorphismus y}
    = \paar[big]{\klmorphismus r(\klmorphismus r\klmorphismus a * \klmorphismus r\klmorphismus b)}
    {\hat{\klmorphismus n}\paar{\klmorphismus r\klmorphismus a}{\klmorphismus r\klmorphismus b}*\klmorphismus C\paar{\klmorphismus r\klmorphismus b}{\klmorphismus x}*\klmorphismus{y}}
  \end{equation}
  is an isomorphism.
\end{lemma}
\begin{proof}
  Let
  $\paar {\klmorphismus r\klmorphismus a}{\klmorphismus x}, \paar {\klmorphismus
    r\klmorphismus b}{\klmorphismus y}∈\klmorphismus r[\orbitfaltigkeit{\kategorie K}{\gruppe}]\times \orbitfaltigkeit{\erzkategorie{\kategorie K}{\gruppe}}{\gruppe}$. Then,
  \begin{align*}
    φ\paar {\klmorphismus r\klmorphismus a}{\klmorphismus x}*φ \paar {\klmorphismus
    r\klmorphismus b}{\klmorphismus y}
    &= (\klmorphismus r\klmorphismus a*\klmorphismus x) * ({\klmorphismus
      r\klmorphismus b}*{\klmorphismus y})
    = \klmorphismus r\klmorphismus a * \klmorphismus
      r\klmorphismus b * \klmorphismus C\paar{\klmorphismus r\klmorphismus b}{\klmorphismus x} *\klmorphismus y\\
    &= \klmorphismus r(\klmorphismus r\klmorphismus a * \klmorphismus
      r\klmorphismus b) *\klmorphismus n(\klmorphismus r\klmorphismus a * \klmorphismus
      r\klmorphismus b) * \klmorphismus C\paar{\klmorphismus r\klmorphismus b}{\klmorphismus x} *\klmorphismus y\\
    &= φ\paar[big]{\klmorphismus r(\klmorphismus r\klmorphismus a * \klmorphismus
      r\klmorphismus b)}{\klmorphismus n(\klmorphismus r\klmorphismus a * \klmorphismus
      r\klmorphismus b) * \klmorphismus C\paar{\klmorphismus r\klmorphismus b}{\klmorphismus x} *\klmorphismus y}\\
    &= φ\bigl(\paar{\klmorphismus r\klmorphismus a}{\klmorphismus x}
      *\paar{\klmorphismus r\klmorphismus b}{\klmorphismus y}\bigr).
  \end{align*}

  According to
  Lemma~\ref{lemma:13} and Lemma~\ref{lemma:eind.reduzibel-phi-injektiv} a category homomorphism
  $\abbildung {ψ_{\klmorphismus
      x_2}}{\Morall{\erzkategorie{\orbitfaltigkeit{\kategorie
          K}{\gruppe}}{\moranfang\klmorphismus
        x_2}}}{\Morall{\erzkategorie{\orbitfaltigkeit{\kategorie
          K}{\gruppe}}{\morende\klmorphismus x_2}}}$ exists such that
  $\klmorphismus y_1 * \klmorphismus x_2 = \klmorphismus x_2 *
  ψ_{\klmorphismus x_2}\klmorphismus y_1$.
  Lemma~\ref{lemma:15} assures that this arrow is unique as the
  arrow $\klmorphismus x_2 * ψ_{\klmorphismus x_2}\klmorphismus y_1$
  has a unique representation in the form \eqref{eq:17}. For the same reason the representation
  $\klmorphismus x_1 * \klmorphismus x_2 *ψ_{\klmorphismus
    x_2}\klmorphismus y_1 * \klmorphismus y_2$ is unique.
  \qed
\end{proof}
\begin{corollary}\label{cor:eq:31}
  Let $\abbildung{A}{\orbitfaltigkeit{\kategorie K}{\gruppe}}{\gruppe}$ be an annotation. Then
  \begin{equation}
    \label{eq:31}
    A(\klmorphismus r\klmorphismus b)A(\klmorphismus r\klmorphismus a) = A(\klmorphismus r\klmorphismus a * \klmorphismus r\klmorphismus b) = A\bigl(\hat{\klmorphismus n}(\klmorphismus r\klmorphismus a * \klmorphismus r\klmorphismus b)\bigr)
    A\bigl(\klmorphismus r(\klmorphismus r \klmorphismus a *\klmorphismus r\klmorphismus b)\bigr)
  \end{equation}
\end{corollary}
\begin{corollary}\label{cor:eq:41}
  The mapping $\hat{\klmorphismus n}$ from \eqref{eq:24} fulfils the following equation:
  \begin{subequations}
    \begin{align}
      \hat{\klmorphismus n}\paar{\klmorphismus a}{\id_{\morende\klmorphismus a}}
      &= \klmorphismus \id_{\morende\klmorphismus a}\label{eq:41}\\
      \hat{\klmorphismus n}\paar{\id_{\moranfang\klmorphismus a}}{\klmorphismus a }
      &= \klmorphismus \id_{\morende\klmorphismus a}\label{eq:41a}
    \end{align} 
  \end{subequations}
\end{corollary}

\noindent Recall, that we aim at a description of orbit categories of
po-groups. As they form groups with binary relations on them, it is an
interesting question, in which case the orbit category can be
described by a binary relation. So, in which case is the orbit
category or the flat orbit category a binary relation aka simple
category? The following lemma should help.

\begin{lemma}\label{lemma:simele-flat-orbit-category-when-simple}
  Let $\kategorie K$ denote a category and $\gruppe\leq\Aut\kategorie K$ a
  translative automorphism group that acts normal on $\kategorie
  K$. Let further $\orbitfaltigkeit{\kategorie K}{\gruppe}$ be
  representable by irreducible arrows. If $\kategorie K$ is a simple
  category and for some vertex $x∈\Ob\kategorie K$ for any two
  vertices $x',x''∈\Ob\kategorie G_x$ there exists a vertex
  $y∈\morende[\ordnungsfilter[\kategorie G_x]{x'}]\cap
  \morende[\ordnungsfilter[\kategorie G_x]{x''}]$, then the flat orbit
  category $\klmorphismus r[\orbitfaltigkeit{\kategorie K}{\gruppe}]$
  is a simple category.
\end{lemma}

\begin{proof}
  Suppose $\klmorphismus r[\orbitfaltigkeit{\kategorie K}{\gruppe}]$
  is not a simple category. Then there are two vertices
  $x,y∈\Ob\klmorphismus r[\orbitfaltigkeit{\kategorie K}{\gruppe}]$
  and two different arrows
  $\klmorphismus a,\klmorphismus b ∈\Mor{\klmorphismus
    r[\orbitfaltigkeit{\kategorie K}{\gruppe}]}xy$. Additionally, we
  can find vertices $x'∈x$ and $y',y''∈y$ and arrows
  $\klmorphismus a'∈\Mor{\kategorie K}{x'}{y'}\cap\klmorphismus a$ and
  $\klmorphismus b'∈\Mor{\kategorie K}{x'}{y''}\cap\klmorphismus
  b$. As $\kategorie K$ is a simple category from
  $\klmorphismus a\neq\klmorphismus b$ follows $y'\neq y''$. Then we know the following facts:

  \begin{itemize}
  \item There is no direct arrow from $y'$ to $y''$. Otherwise, the identity $\klmorphismus a = \klmorphismus
    b$ holds, as $\orbitfaltigkeit{\kategorie K}{\gruppe}$ is representable by irreducible arrows. And,
  \item If there exists a vertex $z∈y$ with arrows
    \begin{tikzcd}[column sep=small]
      y'\arrow{r}{\klmorphismus z_1}&z&y''\arrow[swap]{l}{\klmorphismus z_2},
    \end{tikzcd}
    then the equation
    $\klmorphismus a'*\klmorphismus z_1 = \klmorphismus
    b'*\klmorphismus z_2$ holds as $\kategorie K$ is a simple
    category. Since $\orbitfaltigkeit{\kategorie K}{\gruppe}$ is
    uniquely representable by irreducible arrows at least one of
    $\klmorphismus a'$ or $\klmorphismus b'$ must be reducible. Thus,
    $\klmorphismus a = \klmorphismus b$.
  \end{itemize}

  As there cannot be two distinct arrows in the same direction between
  the same two vertices, the category
  $\klmorphismus r[\orbitfaltigkeit{\kategorie K}{\gruppe}]$ is a
  simple category.
  \qed
\end{proof}

\begin{corollary}
  If $\kategorie K$ is a directed po-group, the flat orbit
  category is a simple category.
\end{corollary}

\noindent Also circles simplify the structure of flat orbit categories a lot.

\begin{lemma}\label{lemma:flat-categories-no-parallel-arrows-in-circles}
  Let $\kategorie K$ a category that is representable by
  irreducible arrows. Then each subcategory of
  $\klmorphismus r[\kategorie K]$ that is generated by a finite circle, is
  simple.
\end{lemma}
\begin{proof}
  In a category every circle can be reduced to two antiparallel arrows
  between two arbitrary vertices of the circle. If
  $\klmorphismus r[\kategorie K]$ contains a circle with two parallel
  arrows, then we find vertices $x,y∈\Ob\klmorphismus r[\kategorie K]$
  and arrows
  $\klmorphismus a, \klmorphismus b∈\Mor{\klmorphismus r[\kategorie
    K]}{x}y$ and
  $\klmorphismus c∈\Mor{\klmorphismus r[\kategorie K]}yx$.  Then in
  $\kategorie K$ the arrows
  $\klmorphismus a *_{\kategorie K} \klmorphismus c$ and
  $\klmorphismus c *_{\kategorie K} \klmorphismus b$ are loops. This
  leads to the equation
  \begin{align*}
    \klmorphismus a = \klmorphismus r\bigl(
    \klmorphismus a *_{\kategorie K} (\klmorphismus c *_{\kategorie K} \klmorphismus b)\bigr)
    = \klmorphismus r\bigl((\klmorphismus a *_{\kategorie K} \klmorphismus c) *_{\kategorie K}\klmorphismus b\bigr)
    = \klmorphismus r\bigl(\klmorphismus b *_{\kategorie K}
    \klmorphismus C\paar{\klmorphismus b}{\klmorphismus a *_{\kategorie K} \klmorphismus c}\bigr)
    = \klmorphismus b
  \end{align*}
  Consequently, there can be only one arrow in each direction between
  two vertices of the circle in $\klmorphismus r[\kategorie K]$.
  \qed
\end{proof}

\noindent It is easy to see that the flat orbit category of an
$\ell$-group is isomorphic to the factor relation that arises from the
underlying simple category of $\kategorie K$ via factorisation by the
orbit partition of a normal subgroup $\kategorie N$ acting on
$\kategorie K$.

At the first glance it also suffers from the fact that
the direction information might get lost. However, any annotation on
the orbit category is still available as a mapping from the flat orbit
category into the group. It is not necessarily a homomorphism anymore.
However, it can be used together with Lemma~\ref{lemma:18} to recreate the
annotation on the product category. If $A$ is an annotation of
$\orbitfaltigkeit{\kategorie K}{\Ob{\kategorie N}}$ and
– consequently – $\entfaltung{\orbitfaltigkeit{\kategorie K}{\Ob{\kategorie
      N}}}A{\Ob{\kategorie N}}$, and if there exist three arrows
$\klmorphismus a, \klmorphismus b, \klmorphismus c∈\klmorphismus
r[\orbitfaltigkeit{\kategorie K}{\Ob{\kategorie N}}]$ of the flat
orbit category with
$\klmorphismus a \bullet \klmorphismus b = \klmorphismus c$, then
$\klmorphismus c$ is redundant if there exists an arrow
$\klmorphismus d∈\Mor{\orbitfaltigkeit{\kategorie K}{\Ob{\kategorie
      N}}}{\morende \klmorphismus b}{\morende \klmorphismus b}$ such
that
$\klmorphismus c = \klmorphismus a * \klmorphismus b *\klmorphismus
d$. This is the case when four arrows
$\hat{\klmorphismus a}∈\klmorphismus a$,
$\hat{\klmorphismus b}∈\klmorphismus b$,
$\hat{\klmorphismus c}∈\klmorphismus c$ and
$\hat{\klmorphismus d}∈\klmorphismus d$ exist such that
$\hat{\klmorphismus c} = \hat{\klmorphismus a} * \hat{\klmorphismus b}
* \hat{\klmorphismus d}$, which is equivalent to
$\morende \hat{\klmorphismus b}\leq \morende \hat{\klmorphismus c}$.
We will explore these facts in the following sections.

\section{Flat category representations}\label{sec:flat-categ-repr}

As we have seen, we can express the category part of a representation
by a category, a groupal category and two mappings. This should be enough
in order to define representations that use the idea of a flat category.

In analogy to group extensions we can define also category
extensions. The isomorphism \eqref{eq:29} is a good
candidate. Actually, we have to resemble some results from
\cite{Schreier:1925,Schreier:1926} with respect to flat orbit
categories.

\begin{definition}\label{def:singleton-category-extension}
  Let $\kategorie K$ denote a category with concatenation $\bullet$
  and $\kategorie G$ another category with the
  concatenation $*$ and $\Ob\kategorie G = \menge 1$, and let
  \begin{subequations}
    \begin{align}\label{eq:56}
      \wabbildung{\abbildung{\moranfang_{\kategorie L}&}{\Morall{\kategorie K}\times\Morall{\kategorie G}}{\Ob\kategorie K}}
      {\paar{\klmorphismus a}{\klmorphismus g}}{\moranfang\klmorphismus a}\\
      \wabbildung{\abbildung{\morende_{\kategorie L}&}{\Morall{\kategorie K}\times\Morall{\kategorie G}}{\Ob\kategorie K}}
      {\paar{\klmorphismus a}{\klmorphismus g}}{\morende\klmorphismus a}\\
      \abbildung{\klmorphismus C&}{\Morall{\kategorie K}\times\Morall{\kategorie G}}{\Morall{\kategorie G}},\\
      \intertext{mappings, where $\klmorphismus C$ fulfils  \eqref{eq:12a} and \eqref{eq:36}, and}
      \abbildung{\klmorphismus n&}{\Morall{\kategorie K}\times\Morall{\kategorie K}}{\Morall{\kategorie G}},\\
      \wabbildung{\abbildung{{*_{\kategorie L}}&}{(\Morall{\kategorie K}\times\kategorie G)^2}
                       {\Morall{\kategorie K}\times\kategorie G}
      \\
                             &}{\paar[big]{\paar{\klmorphismus a}{\klmorphismus x}}
                               {\paar{\klmorphismus b}{\klmorphismus y}}}
                         {\paar{\klmorphismus a\bullet\klmorphismus b}
        {\klmorphismus n\paar{\klmorphismus a}{\klmorphismus b}*\klmorphismus C\paar{\klmorphismus b}{\klmorphismus x}*\klmorphismus y}},\nonumber
    \end{align}
  two partial mappings such that the equations
  \begin{align}
    \klmorphismus C\paar{\klmorphismus a}{\id_1} &= \id_1\\
    {\klmorphismus n}\paar{\id_x}{\id_x}
    &= {\klmorphismus n}\paar{\klmorphismus a}{\id_{\morende\klmorphismus a}}
      ={\klmorphismus n}\paar{\id_{\moranfang\klmorphismus a}}{\klmorphismus a }
      = \id_{1}\label{eq:46}
  \end{align}
  \end{subequations}
  hold and
  $\kategorie L\definiert\tupel{\Ob\kategorie K,\Morall{\kategorie
      K}\times\Morall{\kategorie G}, \moranfang_{\kategorie
      L},\morende_{\kategorie L},*_{\kategorie L}}$ is a small
  category.
  Then, the category $\kategorie L$ is called a
  \defindex{singleton category extension} of $\kategorie K$ by
  $\kategorie G$. This is denoted by
  $\kategorie E\tupel{\kategorie K, \klmorphismus n,\klmorphismus
    C,\kategorie G}\definiert\kategorie L$.
\end{definition}

\noindent For convenience we identify the category $\kategorie L$ with
the singleton category extension that generates it. This can be used
to define flat category representations as another kind of an
unfoldable factor structure.

\begin{definition}\label{def:flat-representations}
  Let $\kategorie K$ denote a category,
  $\kategorie G$ a right groupal category, such that the singleton category extension
  $\kategorie L\definiert\kategorie E\tupel{\kategorie K, \klmorphismus n, \klmorphismus C,
    \orbitfaltigkeit{\kategorie G}{\Ob\kategorie G}}$ exists,
  and let
  \begin{subequations}
    \begin{align}
      \abbildung{A&}{\Morall{\kategorie K}}{\Ob\kategorie G},\\
      \wabbildung{\abbildung{A'&}{\Morall{\kategorie K}\times
                                 {\Morall{\orbitfaltigkeit{\kategorie G}{\Ob\kategorie G}}}}
                                 {\Ob{\kategorie G}}\nonumber\label{eq:25}\\
                                &}{\paar{\klmorphismus a}{\klmorphismus x^{\Ob\kategorie G}}}
                                  {\morende\klmorphismus x(\moranfang\klmorphismus x)^{-1}A(\klmorphismus a)}
    \end{align}
  \end{subequations}
  two mappings.  Then, the tuple
  $\tupel{\kategorie K,A,\klmorphismus n,\klmorphismus C, \kategorie
    G}$ is called a \defindex{flat category representation} of
  $\entfaltung{\kategorie L}{A'}{\Ob\kategorie G}$ via
  $\tupel{\kategorie L,A',\Ob\kategorie G}$, iff
  $\tupel{\kategorie L,A',\Ob\kategorie G}$ is a representation.
\end{definition}

\begin{corollary}
  In equation~\eqref{eq:25} $A'$ is well-defined i.e.\ for every vertex $g∈\Ob\kategorie G$ the
  following equation holds:
  \begin{equation}
    \label{eq:26}
    A'\paar{\klmorphismus a}{\klmorphismus x^{\Ob\kategorie G}} =
    A'\paar[big]{\klmorphismus a}{(\klmorphismus x^g)^{\Ob\kategorie G}} 
  \end{equation}
\end{corollary}
\begin{proof}
  $\morende\klmorphismus x^g(\moranfang\klmorphismus x^g)^{-1} = \morende\klmorphismus xgg^{-1}(\moranfang\klmorphismus x)^{-1}=\morende\klmorphismus x(\moranfang\klmorphismus x)^{-1}$
  \qed
\end{proof}

\noindent Additionally to the indirect definitions above, we can find also
axiomatic descriptions of singleton category extensions and flat
category representations.

\begin{theorem}
  Let $\kategorie K$ denote a category with concatenation $\bullet$
  and $\kategorie G$ another category with
  $\Ob\kategorie G = \menge 1$ and concatenation $*$, and let
  $\klmorphismus C$, $\moranfang_{\kategorie L}$,
  $\morende_{\kategorie L}$ mappings and $\klmorphismus n$ and
  $*_{\kategorie L}$ partial mappings according to the properties~\eqref{eq:56} to~\eqref{eq:46} from Definition
  \ref{def:singleton-category-extension}.
  Then the structure 
  $\kategorie L\definiert\tupel{\Ob\kategorie K,\Morall{\kategorie
      K}\times\Morall{\kategorie G}, \moranfang_{\kategorie
      L},\morende_{\kategorie L},*_{\kategorie L}}$ is a
  \defindex{singleton category extension} of $\kategorie K$ by
  $\kategorie G$ iff the following equations hold:
  \begin{subequations}
    \begin{align}
      \paar{\id_{\moranfang\klmorphismus a}}{\klmorphismus x} *_{\kategorie L}\paar{\klmorphismus a}{\id_1}
      &=\paar[big]{\klmorphismus a}{\klmorphismus C\paar{\klmorphismus a}{\klmorphismus x}}\label{eq:47}\\
      \klmorphismus n\paar{\klmorphismus a}{\klmorphismus b \bullet \klmorphismus c}
      *\klmorphismus n\paar{\klmorphismus b}{\klmorphismus c}
      &= \klmorphismus n\paar{\klmorphismus a \bullet\klmorphismus b}{\klmorphismus c}
        * \klmorphismus C\paar{\klmorphismus c}{\klmorphismus n\paar{\klmorphismus a}{\klmorphismus b}}\label{eq:43} \\
      \klmorphismus C\paar{\klmorphismus a}{\klmorphismus x * \klmorphismus y}
      &= \klmorphismus C\paar{\klmorphismus a}{\klmorphismus x} * \klmorphismus  C\paar{\klmorphismus a}{\klmorphismus y}\label{eq:44}\\
      \klmorphismus C\paar{\klmorphismus a \bullet \klmorphismus b}{\klmorphismus x}
      *\klmorphismus n\paar{\klmorphismus a}{\klmorphismus b}
      &= \klmorphismus n\paar{\klmorphismus a}{\klmorphismus b}*\klmorphismus C\paar[big]{\klmorphismus b}{\klmorphismus C\paar{\klmorphismus a}{\klmorphismus x}}\label{eq:45}
    \end{align}
  \end{subequations}
\end{theorem}

\begin{proof}
  Let us first assume that $\kategorie L$ is a singleton category extension of $\kategorie K$. Then we can prove
  the Equations~\eqref{eq:47}, \eqref{eq:43}, \eqref{eq:44} and~\eqref{eq:45}.
  Equation~\eqref{eq:47} is a direct consequence from the definition:
  \begin{align*}
      \paar{\id_{\moranfang\klmorphismus a}}{\klmorphismus x} *_{\kategorie L}\paar{\klmorphismus a}{\id_1}
      &=\paar[big]{\id_{\moranfang\klmorphismus a}\bullet\klmorphismus a}{\underbrace{\klmorphismus n\paar{\id_{\moranfang\klmorphismus a}}{\klmorphismus a}}_{{}=\id_1}*\klmorphismus C\paar{\klmorphismus a}{\klmorphismus x}*\id_1}
      =\paar[big]{\klmorphismus a}{\klmorphismus C\paar{\klmorphismus a}{\klmorphismus x}}
  \end{align*}
  Let us consider Equation~\eqref{eq:43}. From the transitivity of the category $\kategorie L$ follows
  \begin{align*}
    \paar[big]{\klmorphismus a \bullet \klmorphismus b \bullet \klmorphismus c}
    {\klmorphismus n\paar{\klmorphismus a}{\klmorphismus b\bullet \klmorphismus c}*
    \klmorphismus n \paar {\klmorphismus b}{\klmorphismus c}}
    &= \paar{\klmorphismus a}{\id_1}*_{\kategorie L}
      \paar[big]{\klmorphismus b \bullet \klmorphismus c}
    {\klmorphismus n \paar {\klmorphismus b}{\klmorphismus c}}\\
    &= \paar{\klmorphismus a}{\id_1}*_{\kategorie L}
      \paar[big]{\klmorphismus b}{\id_1}*_{\kategorie L}
      \paar[big]{\klmorphismus c}{\id_1}\\
    &=\paar[big]{\klmorphismus a\bullet \klmorphismus b}
      {\klmorphismus n\paar{\klmorphismus a}{\klmorphismus b}}*_{\kategorie L}
      \paar{\klmorphismus c}{\id_1}\\
    &=\paar[Big]{\klmorphismus a\bullet \klmorphismus b\bullet \klmorphismus c}
      {\klmorphismus n\paar{\klmorphismus a\bullet \klmorphismus b}{\klmorphismus c}*
      \klmorphismus C\paar[big]{\klmorphismus c}{\klmorphismus n\paar{\klmorphismus a}{\klmorphismus b}}}
  \end{align*}
  Equation~\eqref{eq:44} follows from
  \begin{align*}
    \paar{\klmorphismus a}{\klmorphismus C\paar{\klmorphismus a}{\klmorphismus x * \klmorphismus y}}
    &=
      \paar{\id_{\moranfang\klmorphismus a}}{\klmorphismus x*\klmorphismus y}*_{\kategorie L}
      \paar{\klmorphismus a}{\id_1}
      =
      \paar{\id_{\moranfang\klmorphismus a}}{\klmorphismus x}*_{\kategorie L}
      \paar{\id_{\moranfang\klmorphismus a}}{\klmorphismus y}*_{\kategorie L}
      \paar{\klmorphismus a}{\id_1}\\
    &=
      \paar{\id_{\moranfang\klmorphismus a}}{\klmorphismus x}*_{\kategorie L}
      \paar[big]{\klmorphismus a}{\klmorphismus  C\paar{\klmorphismus a}{\klmorphismus y}}
      = \paar[big]{\klmorphismus a}{\klmorphismus C\paar{\klmorphismus a}{\klmorphismus x}
      * \klmorphismus  C\paar{\klmorphismus a}{\klmorphismus y}}
  \end{align*}
  The same scheme can be used to prove Equation~\eqref{eq:45}:
  \begin{align*}
    \paar[big]{\klmorphismus a \bullet \klmorphismus b}{\klmorphismus C\paar{\klmorphismus a \bullet \klmorphismus b}{\klmorphismus x}
    *\klmorphismus n\paar{\klmorphismus a}{\klmorphismus b}}
    &=\paar{\id_{\moranfang\klmorphismus a}}{\klmorphismus x}*_{\kategorie L}\paar[big]{\klmorphismus a \bullet \klmorphismus b}{\klmorphismus n\paar{\klmorphismus a}{\klmorphismus b}}\\
    &=\paar{\id_{\moranfang\klmorphismus a}}{\klmorphismus x}*_{\kategorie L}
      \paar{\klmorphismus a}{\id_1}*_{\kategorie L}\paar{\klmorphismus b}{\id_1}\\
    &=\paar[big]{\klmorphismus a}{\klmorphismus C\paar{\klmorphismus a}{\klmorphismus x}}*_{\kategorie L}
      \paar{\klmorphismus b}{\id_1}\\
    &= \paar[Big]{\klmorphismus a \bullet\klmorphismus b}{\klmorphismus n\paar{\klmorphismus a}{\klmorphismus b}*\klmorphismus C\paar[big]{\klmorphismus b}{\klmorphismus C\paar{\klmorphismus a}{\klmorphismus x}}}
  \end{align*}

  Now, suppose that the Equations~\eqref{eq:47}, \eqref{eq:43},
  \eqref{eq:44} and~\eqref{eq:45} hold. We prove that $\kategorie L$
  is a category. Start and end of the arrows and the identities
  inherit their category structure from $\kategorie K$
  and~\eqref{eq:46}. It remains to show that the partial mapping
  $*_{\kategorie L}$ is a category concatenation, which means that we
  have to prove that it is associative:
  \begin{align*}
    \bigl(\paar{\klmorphismus a}{\klmorphismus x}*_{\kategorie L}
    \paar{\klmorphismus b}{\klmorphismus y}\bigr)*_{\kategorie L}
    \paar{\klmorphismus c}{\klmorphismus z}
    &= \paar[big]{\klmorphismus a\bullet\klmorphismus b}{
      \klmorphismus n\paar{\kategorie a}{\kategorie b}*
      \klmorphismus C\paar{\klmorphismus b}{\klmorphismus x}*\klmorphismus y}
      *_{\kategorie L}\paar{\klmorphismus c}{\klmorphismus z}\\
    &= \paar[Big]{\klmorphismus a\bullet\klmorphismus b\bullet\klmorphismus c}{
      \klmorphismus n\paar{\klmorphismus a\bullet\klmorphismus b}{\klmorphismus c}*
      \klmorphismus C\paar[big]{\klmorphismus c}{
      \klmorphismus n\paar{\kategorie a}{\kategorie b}*
      \klmorphismus C\paar{\klmorphismus b}{\klmorphismus x}*\klmorphismus y}*\klmorphismus z}\\
    \intertext{with \eqref{eq:44} we can rewrite this term as }
    &= \paar[Big]{\klmorphismus a\bullet\klmorphismus b\bullet\klmorphismus c}{
      \klmorphismus n\paar{\klmorphismus a\bullet\klmorphismus b}{\klmorphismus c}*
      \klmorphismus C\paar[big]{\klmorphismus c}{
      \klmorphismus n\paar{\kategorie a}{\kategorie b}}*
      \klmorphismus C\paar[big]{\klmorphismus c}{
      \klmorphismus C\paar{\klmorphismus b}{\klmorphismus x}}*
      \klmorphismus C\paar[big]{\klmorphismus c}{\klmorphismus y}*\klmorphismus z}\\
    \intertext{\eqref{eq:43} leads to }
    &= \paar[Big]{\klmorphismus a\bullet\klmorphismus b\bullet\klmorphismus c}{
      \klmorphismus n\paar{\kategorie a}{\kategorie b\bullet\klmorphismus c}*
      \klmorphismus n\paar{\klmorphismus b}{\klmorphismus c}*
      \klmorphismus C\paar[big]{\klmorphismus c}{
      \klmorphismus C\paar{\klmorphismus b}{\klmorphismus x}}*
      \klmorphismus C\paar[big]{\klmorphismus c}{\klmorphismus y}*\klmorphismus z}\\
  \intertext{with \eqref{eq:45} this is equivalent to}
    &=
    \paar[big]{\klmorphismus a\bullet\klmorphismus b\bullet\klmorphismus c}
      {\klmorphismus n\paar{\klmorphismus a}{\klmorphismus b\bullet\klmorphismus c}*
      \klmorphismus C\paar{\klmorphismus b\bullet\klmorphismus c}{\klmorphismus x}*
      \klmorphismus n\paar{\klmorphismus b}{\klmorphismus c}*
    \klmorphismus C\paar{\klmorphismus c}{\klmorphismus y}*\klmorphismus z}\bigr)\\
    &=
    \paar{\klmorphismus a}{\klmorphismus x}*_{\kategorie L}
    \paar[big]{\klmorphismus b\bullet\klmorphismus c}
    {\klmorphismus n\paar{\klmorphismus b}{\klmorphismus c}*
    \klmorphismus C\paar{\klmorphismus c}{\klmorphismus y}*\klmorphismus z}\bigr)\\
    &=
    \paar{\klmorphismus a}{\klmorphismus x}*_{\kategorie L}
    \bigl(\paar{\klmorphismus b}{\klmorphismus y}*_{\kategorie L}
    \paar{\klmorphismus c}{\klmorphismus z}\bigr)
  \end{align*}
  So the partial mapping $*_{\kategorie L}$ is a category
  concatenation. Now, it is obvious that $\kategorie L$ is a category.
  \qed
\end{proof}

\noindent For the flat category representation we need some axioms
that properly define the annotation.

\begin{theorem}\label{thm:flat-annotation-is-annotation}
  Let $\kategorie K$  be a category with concatenation $\bullet$,
  $\kategorie G$ a right groupal category, such that the singleton category extension
  $\kategorie L\definiert\kategorie E\tupel{\kategorie K, \klmorphismus n, \klmorphismus C,
    \orbitfaltigkeit{\kategorie G}{\Ob\kategorie G}}$ exists,
  and let
  \begin{subequations}
    \begin{align}
      \abbildung{A&}{\Morall{\kategorie K}}{\Ob\kategorie G},\\
      \wabbildung{\abbildung{A'&}{\Morall{\kategorie K}\times
                                 {\Morall{\orbitfaltigkeit{\kategorie G}{\Ob\kategorie G}}}}
                                 {\Ob{\kategorie G}}\nonumber\\
                  &}{\paar{\klmorphismus a}{\klmorphismus x^{\Ob\kategorie G}}}
                    {\morende\klmorphismus x(\moranfang\klmorphismus x)^{-1}A(\klmorphismus a)}\label{eq:52}
    \end{align}
  \end{subequations}
  two mappings.  Then, the tuple
  $\tupel{\kategorie K,A,\klmorphismus n,\klmorphismus C, \kategorie
    G}$ is a {flat category representation} of
  $\entfaltung{\kategorie L}{A'}{\Ob\kategorie G}$ via
  $\tupel{\kategorie L,A',\Ob\kategorie G}$, iff the following equations hold:
  \begin{subequations}
    \begin{align}
      A'\paar{\klmorphismus a}{\klmorphismus x}
      &= A'\paar{\id_{\morende \klmorphismus a}}{\klmorphismus x} A'\paar{\klmorphismus a}{\id_1}\label{eq:42}\\
      \morende\klmorphismus n\paar{\klmorphismus a}{\klmorphismus b}\bigl(\moranfang\klmorphismus n\paar{\klmorphismus a}{\klmorphismus b}\bigr)^{-1}
      &= A(\klmorphismus b)A(\klmorphismus a)\bigl(A(\klmorphismus a\bullet\klmorphismus b)\bigr)^{-1}\label{eq:34}\\
      A(\klmorphismus b)\morende{\klmorphismus x}(\moranfang{\klmorphismus x})^{-1}
      &=
      \morende\klmorphismus C\paar{\klmorphismus b}{\klmorphismus x}
        \bigl(\moranfang\klmorphismus C\paar{\klmorphismus b}{\klmorphismus x}\bigr)^{-1}
        A(\klmorphismus b)\label{eq:48}
    \end{align}
  \end{subequations}
\end{theorem}
\begin{proof}
  If $\tupel{\kategorie K,A,\klmorphismus n,\klmorphismus C, \kategorie
    G}$ is a flat representation of $\entfaltung{\kategorie L}{A'}{\Ob\kategorie G}$, then
  Equation~\eqref{eq:42} follows from Corollary~\ref{cor:eq:41}, and
  Equation~\eqref{eq:34} is a direct consequence of Corollary~\ref{cor:eq:31}.  Equation~\eqref{eq:48} can also be easily proved:
  \begin{align*}
    A(\klmorphismus b) \morende\klmorphismus x(\moranfang\klmorphismus x)^{-1}
    &=1A(\klmorphismus b) \morende\klmorphismus x(\moranfang\klmorphismus x)^{-1}A(\id_{\moranfang\klmorphismus b})\\
      &=A'\paar{\klmorphismus b}{\id_1}A'\paar{\id_{\moranfang\klmorphismus b}}{\klmorphismus x}\\
    &=A'\bigl(\paar{\id_{\moranfang\klmorphismus b}}{\klmorphismus x}*_{\kategorie L}\paar{\klmorphismus b}{\id_1})\\
    &=A'\paar[big]{\klmorphismus b}{\klmorphismus C\paar{\klmorphismus b}{\klmorphismus x}}\\
    &=\morende\klmorphismus C\paar{\klmorphismus b}{\klmorphismus x}
        \bigl(\moranfang\klmorphismus C\paar{\klmorphismus b}{\klmorphismus x}\bigr)^{-1}
        A(\klmorphismus b)  
  \end{align*}

  Now, Suppose Equations~\eqref{eq:42}, \eqref{eq:34} and~\eqref{eq:48} hold. We have to prove that $A'$ is an annotation.
  \begin{align*}
    A'\bigl(\paar{\klmorphismus a}{\klmorphismus x}*_{\kategorie L}\paar{\klmorphismus b}{\klmorphismus x}\big)
    &= A'\paar{\klmorphismus a\bullet\klmorphismus b}{
      \klmorphismus n\paar{\klmorphismus a}{\klmorphismus b}*
      \klmorphismus C\paar{\klmorphismus b}{\klmorphismus x}*
      \klmorphismus y}\\
    &= \morende\klmorphismus y(\moranfang\klmorphismus y)^{-1}
      \morende\klmorphismus C\paar{\klmorphismus b}{\klmorphismus x}
      \bigl(\moranfang\klmorphismus C\paar{\klmorphismus b}{\klmorphismus x}\bigr)^{-1}
      \morende\klmorphismus n\paar{\klmorphismus a}{\klmorphismus b}
      \bigl(\moranfang\klmorphismus n\paar{\klmorphismus a}{\klmorphismus b}\bigr)^{-1}
      A(\klmorphismus a\bullet\klmorphismus b)\\
    &= \morende\klmorphismus y(\moranfang\klmorphismus y)^{-1}
      \morende\klmorphismus C\paar{\klmorphismus b}{\klmorphismus x}
      \bigl(\moranfang\klmorphismus C\paar{\klmorphismus b}{\klmorphismus x}\bigr)^{-1}
      A(\klmorphismus b)A(\klmorphismus a)\\
    &= \morende\klmorphismus y(\moranfang\klmorphismus y)^{-1}A(\klmorphismus b)
      \morende{\klmorphismus x}(\moranfang{\klmorphismus x})^{-1}A(\klmorphismus a)\\
    &= A'\paar{\klmorphismus b}{\klmorphismus y}A'\paar{\klmorphismus a}{\klmorphismus x}.
  \end{align*}
  So $A'$ is an annotation from $\kategorie L$ into the group $\Ob\kategorie G$ and the tuple
  $\tupel{\kategorie L,A',\Ob\kategorie G}$ is a representation, which proves the theorem.
  \qed
\end{proof}

\noindent Now, we can prove that small categories that are uniquely
representable by irreducible arrows give rise to flat category
representations.

\begin{lemma}\label{lemma:flat-representations}
  Let $\kategorie K$ be a category, $\gruppe\leq \Aut\kategorie K$
  a translative automorphism group that acts normal on $\kategorie K$
  such that $\orbitfaltigkeit{\kategorie K}{\gruppe}$ is uniquely
  representable by irreducible arrows and
  $\tupel{\orbitfaltigkeit{\kategorie K}{\gruppe},A,\gruppe}$ a
  representation of $\kategorie K$.
  Under the conditions of Lemma~\ref{lemma:16} and
  Lemma~\ref{lemma:18} for every $x∈\Ob\kategorie K$, the tuple
  $\tupel{\klmorphismus r[\orbitfaltigkeit{\kategorie
      K}{\gruppe}],A\eingeschrmenge{\Morall{\klmorphismus
        r[\orbitfaltigkeit{\kategorie
          K}{\gruppe}]}},\hat{\klmorphismus n},\klmorphismus C,
    \kategorie G_x}$ is a flat category
  representation via $\tupel{\orbitfaltigkeit{\kategorie K}{\gruppe},A,\gruppe}$.
\end{lemma}

\begin{proof}
  As $A$ is determined by a natural annotation of
  $\orbitfaltigkeit{\bigl(\entfaltung{(\orbitfaltigkeit{\kategorie
        K}{\gruppe})}{A}{\gruppe}\bigr)}{\gruppe}$, it is sufficient
  to prove the lemma for the case where $A$ is a natural annotation.

  Let us first assume that $A$ is a natural annotation.
  Lemma~\ref{lemma:18}, tells us that the category $\kategorie L$ with
  $\Ob\kategorie L\definiert\Ob(\klmorphismus
  r[\orbitfaltigkeit{\kategorie K}{\gruppe}])$,
  $\Mor{\kategorie L}xy\definiert\Mor{\klmorphismus
    r[\orbitfaltigkeit{\kategorie
\      K}{\gruppe}]}xy\times\Morall{\orbitfaltigkeit{\kategorie
      G_x}{\gruppe}}$ and the concatenation from Lemma~\ref{lemma:18}
  is a category. So we have to prove that $A'$ from
  equation~\eqref{eq:25} is an annotation. With the definition
  \begin{equation*}
    \wabbildung{\abbildung{A''}{\Morall{\kategorie K}\times{\Morall{\kategorie G_x}}}{\Ob{\kategorie G_x}}}
    {\paar{\klmorphismus r\klmorphismus a}{\klmorphismus x^{\Ob\kategorie G_x}}}{A'\paar{\klmorphismus r\klmorphismus a}{\klmorphismus x}}\tag{*}
  \end{equation*}
  we know that
  \begin{align*}
    A''\paar{\klmorphismus r\klmorphismus b^{\gruppe}}{\klmorphismus y^{\gruppe}}
    A''\paar{\klmorphismus r\klmorphismus a^{\gruppe}}{\klmorphismus x^{\gruppe}}
    &= \morende\klmorphismus y(\moranfang\klmorphismus y)^{-1}A(\klmorphismus r\klmorphismus b^{\gruppe})
       \morende\klmorphismus x(\moranfang\klmorphismus x)^{-1}A(\klmorphismus r\klmorphismus a^{\gruppe})\\
    \intertext{As $A$ is a natural annotation of $\kategorie K$, $A\eingeschrmenge{\kategorie G_x}$
    is also a natural annotation which coincides on $\Morall{\kategorie G_x}$ with the natural anotation $A_{\menge x}$ that is defined on $\kategorie G_x$.  So we can rewrite this:}
    &= A(\klmorphismus y^{\gruppe})A(\klmorphismus r\klmorphismus b^{\gruppe})
       A(\klmorphismus x^{\gruppe})A(\klmorphismus r\klmorphismus a^{\gruppe})\\
    A''\bigl(\paar{\klmorphismus r\klmorphismus a^{\gruppe}}{\klmorphismus x^{\gruppe}}
      *\paar{\klmorphismus r\klmorphismus b^{\gruppe}}{\klmorphismus y^{\gruppe}}\bigr)
    &= A''\paar{\klmorphismus r\klmorphismus a^{\gruppe}*\klmorphismus r\klmorphismus b^{\gruppe}}{\hat{\klmorphismus n}\paar{\klmorphismus r\klmorphismus a^{\gruppe}}{\klmorphismus r\klmorphismus b^{\gruppe}}*\klmorphismus C\paar{\klmorphismus r\klmorphismus b^{\gruppe}}{\klmorphismus x^{\gruppe}}*\klmorphismus y^{\gruppe}}\\
    &= A''\paar[big]{\klmorphismus r(\klmorphismus r\klmorphismus a^{\gruppe}*\klmorphismus r\klmorphismus b^{\gruppe})}{\hat{\klmorphismus n}\paar{\klmorphismus r\klmorphismus a^{\gruppe}}{\klmorphismus r\klmorphismus b^{\gruppe}}*\klmorphismus C\paar{\klmorphismus r\klmorphismus b^{\gruppe}}{\klmorphismus x^{\gruppe}}*\klmorphismus y^{\gruppe}}\\
    \intertext{Then there exist arrows $\klmorphismus n', \klmorphismus c', \klmorphismus y'∈\Morall{\kategorie K}$ with $\klmorphismus {n'} ^{\gruppe} = \hat{\klmorphismus n}\paar{\klmorphismus r\klmorphismus a^{\gruppe}}{\klmorphismus r\klmorphismus b^{\gruppe}}$, ${\klmorphismus c'}^{\gruppe} =  \klmorphismus C\paar{\klmorphismus r\klmorphismus b^{\gruppe}}{\klmorphismus x^{\gruppe}}$ and $\klmorphismus {y'}^{\gruppe} = \klmorphismus y$ such that}
    &=\morende{\klmorphismus y'}(\moranfang\klmorphismus y')^{-1}\morende\klmorphismus c'(\moranfang\klmorphismus c')^{-1}\underbrace{\morende\klmorphismus n'(\moranfang\klmorphismus n')^{-1}A\bigl(\klmorphismus r(\klmorphismus r\klmorphismus a^{\gruppe}*\klmorphismus r\klmorphismus b^{\gruppe})\bigr)}_{{}=A(\klmorphismus r\klmorphismus a^{\gruppe} * \klmorphismus r\klmorphismus b^{\gruppe})}\\
    \intertext{As $A$ is a natural annotation, we can rewrite this:}
    &=A(\klmorphismus y^{\gruppe})A\bigl(\klmorphismus c'^{\gruppe}\bigr)A(\klmorphismus r\klmorphismus a^{\gruppe}*\klmorphismus r\klmorphismus b^{\gruppe})\\
    &=A(\klmorphismus y^{\gruppe})A\bigl(\klmorphismus C\paar{\klmorphismus r\klmorphismus b^{\gruppe}}{\klmorphismus x^{\gruppe}}\bigr)A(\klmorphismus r\klmorphismus b^{\gruppe})A(\klmorphismus r\klmorphismus a^{\gruppe})\\
    \intertext{And with $\klmorphismus r\klmorphismus b^{\gruppe} * \klmorphismus C\paar{\klmorphismus r\klmorphismus b^{\gruppe}}{\klmorphismus x^{\gruppe}} = \klmorphismus x^{\gruppe}* \klmorphismus r\klmorphismus b^{\gruppe}$ we get}
    &=A(\klmorphismus y^{\gruppe})A(\klmorphismus r\klmorphismus b^{\gruppe})A(\klmorphismus x^{\gruppe})A(\klmorphismus r\klmorphismus a^{\gruppe})
  \end{align*}
  So $A''$ is an annotation. Furthermore $A''=A$.
  \qed
\end{proof}

\begin{corollary}
  Let $\tupel{\kategorie K,A,\gruppe}$ denote a representation where
  $\kategorie K$ is uniquely representable by irreducible arrows. Then there exists a flat representation $\tupel{\klmorphismus r[\orbitfaltigkeit{\kategorie
      K}{\gruppe}],A\eingeschrmenge{\Morall{\klmorphismus
        r[\orbitfaltigkeit{\kategorie
          K}{\gruppe}]}},\hat{\klmorphismus n},\klmorphismus C,
    \kategorie G_x}$ of $\entfaltung{\kategorie K}{A}{\gruppe}$.
\end{corollary}

\noindent Now, that we know that we can unfold a flat representation via a
representation, it would be handy to have a direct description of the
unfolding of a flat representation. First we define it, afterwards we
prove that the definition is correct.

\begin{definition}
  Let
  $\tupel{\kategorie K,A,\klmorphismus n,\klmorphismus C,\kategorie G}$ be a flat category
  representation. Then, the structure $\entfaltung{\kategorie K}{A,\klmorphismus n,\klmorphismus C}{\kategorie G}$ with
  \begin{subequations}
    \begin{align}
      \Ob{\entfaltung{\kategorie K}{A,\klmorphismus n,\klmorphismus C}{\kategorie G}}
      &\definiert \Ob(\kategorie K \times \kategorie G),\\
      \Mor{\entfaltung{\kategorie K}{A,\klmorphismus n,\klmorphismus C}{\kategorie G}}[big]{\paar xg}{\paar yh}
      &\definiert \Menge{\paar{\klmorphismus a}{\klmorphismus g}∈\Morall{\kategorie K\times \kategorie G}}{\morende \klmorphismus g = h, \moranfang\klmorphismus g = A(\klmorphismus a)g}\text{ and}\\
      \paar{\klmorphismus a}{\klmorphismus g}*_{\entfaltung{\kategorie K}{A,\klmorphismus n,\klmorphismus C}{\kategorie G}} \paar{\klmorphismus b}{\klmorphismus h}
      &\definiert \paar{\klmorphismus a * \klmorphismus b}{\klmorphismus n\paar{\klmorphismus a}{\klmorphismus b}*\klmorphismus C\paar{\klmorphismus b}{\klmorphismus g}*\klmorphismus h}
    \end{align}
  \end{subequations}
  whenever the concatenations are defined in $\kategorie K$ and $\kategorie G$, is called \defindex{unfolding} of $\tupel{\kategorie K,A,\klmorphismus n,\klmorphismus C,\kategorie G}$.
\end{definition}

\begin{lemma}
  Let
  $\tupel{\kategorie K,A,\klmorphismus n,\klmorphismus C,\kategorie
    G}$ be a flat category representation together with a mapping
  \begin{subequations}
    \begin{equation}
      \label{eq:28}
      \wabbildung{\abbildung{A'}{\Morall{\kategorie K}\times
          \Morall{\orbitfaltigkeit{\kategorie G}{\Ob\kategorie G}}}
        {\Ob\kategorie G}}
      {\paar{\klmorphismus a}{\klmorphismus x^{\Ob\kategorie G}}}
      { \morende\klmorphismus x(\moranfang \klmorphismus x)^{-1}A(\klmorphismus a)}.
    \end{equation}
    Then, $A'$ is a category annotation of $\kategorie E(\kategorie K,\klmorphismus n,\klmorphismus
    C,\orbitfaltigkeit{\kategorie G}{\Ob\kategorie
      G})$, and the pair of mappings
    \begin{align}
      \wabbildung{\abbildung{Φ&}{\Ob\kategorie K\times\Ob\kategorie G}
                                {\Ob\kategorie K\times\Ob\kategorie G}}
                                {\paar{a}{g}}
                                {\paar{a}{g}}\\
      \wabbildung{\abbildung{Ψ&}{\Morall{\kategorie K}\times\Morall{\kategorie G}}
                                {\Morall{\kategorie K}\times
                                \Morall{\orbitfaltigkeit{\kategorie G}{\Ob\kategorie G}}\times
                                \Morall{\kategorie G}\nonumber\\&}}
      {\paar{\klmorphismus a}{\klmorphismus g}}
      {\paar[Big]{\paar{\klmorphismus a}{\klmorphismus g^{\Ob\kategorie G}}}
      {\bigl(A(\klmorphismus a)\bigr)^{-1}\moranfang\klmorphismus g}}
    \end{align}
  \end{subequations}
  form an isomorphism between the unfoldings 
  $\entfaltung{\kategorie K}{A,\klmorphismus n,\klmorphismus
    C}{\kategorie G}$ and
  $\entfaltung{\kategorie E\tupel{\kategorie K,\klmorphismus n,\klmorphismus
    C,\orbitfaltigkeit{\kategorie G}{\Ob\kategorie
      G}}}{A'}{\Ob\kategorie G}$.
\end{lemma}
\begin{proof}
  For convenience we write $\kategorie L\definiert\entfaltung{\kategorie K}{A,\klmorphismus n,\klmorphismus
    C}{\kategorie G}$ and $\kategorie M\definiert\entfaltung{\kategorie E(\kategorie K,\klmorphismus n,\klmorphismus
    C,\orbitfaltigkeit{\kategorie G}{\Ob\kategorie
      G})}{A'}{\Ob\kategorie G}$.

  With Theorem~\ref{thm:flat-annotation-is-annotation} the mapping $A'$ is an annotation of $\kategorie M$.
  
  Furthermore,
  \begin{align*}
    Φ\bigl(\moranfang_{\kategorie L}\paar{\klmorphismus a}{\klmorphismus g}\bigr) 
    &= Φ\paar[Big]{\moranfang_{\kategorie K}\klmorphismus a}{\bigl(A(\klmorphismus a)\bigr)^{-1}\moranfang_{\kategorie G}\klmorphismus g}\\
    &= \paar[Big]{\moranfang_{\kategorie K}\klmorphismus a}{\bigl(A(\klmorphismus a)\bigr)^{-1}\moranfang_{\kategorie G}\klmorphismus g}\\
    &= \moranfang_{\kategorie M}\paar[Big]{\paar{\klmorphismus a}{\klmorphismus g^{\Ob\kategorie G}}}
      {\bigl(A(\klmorphismus a)\bigr)^{-1}\moranfang_{\kategorie G}\klmorphismus g}\\
    &= \moranfang_{\kategorie M}Ψ\paar{\klmorphismus a}{\klmorphismus g}\\
    Φ\bigl(\morende_{\kategorie L}\paar{\klmorphismus a}{\klmorphismus g}\bigr) 
    &= Φ\paar{\morende_{\kategorie K}\klmorphismus a}
      {\morende_{\kategorie G}\klmorphismus g}\\
    &= \paar[Big]{\morende_{\kategorie K}\klmorphismus a}
      {\underbrace{\morende_{\kategorie G}\klmorphismus g(\moranfang_{\kategorie G}\klmorphismus g)^{-1}A(\klmorphismus a)
      }_{=A'\paar{\klmorphismus a}{\klmorphismus g^{\Ob\kategorie G}}}
      \bigl(A(\klmorphismus a)\bigr)^{-1}\moranfang_{\kategorie G}\klmorphismus g}\\
    &= \morende_{\kategorie M}\paar[Big]{\paar{\klmorphismus a}{\klmorphismus g^{\Ob\kategorie G}}}
      {\bigl(A(\klmorphismus a)\bigr)^{-1}\moranfang_{\kategorie G}\klmorphismus g}\\
    &= \morende_{\kategorie M}Ψ\paar{\klmorphismus a}{\klmorphismus g}\\
  \end{align*}
  \begin{align*}
  Ψ\bigl(\paar{\klmorphismus a}{\klmorphismus g}*_{\kategorie L}\paar{\klmorphismus b}{\klmorphismus h}\bigr)
    &=Ψ\paar{\klmorphismus a*_{\kategorie K}\klmorphismus b}
      {\klmorphismus n\paar{\klmorphismus a}{\klmorphismus b}*_{\kategorie G}\klmorphismus C\paar{\klmorphismus b}{\klmorphismus g}*_{\kategorie G}\klmorphismus h}\\
    &=\paar[bigg]{\paar[Big]{\klmorphismus a*_{\kategorie K}\klmorphismus b}
      {\bigl(\klmorphismus n\paar{\klmorphismus a}{\klmorphismus b}*_{\kategorie G}\klmorphismus C\paar{\klmorphismus b}{\klmorphismus g}*_{\kategorie G}\klmorphismus h\bigr)^{\Ob\kategorie G}}}
      {\\&\qquad\bigl(A(\klmorphismus a*_{\kategorie K}\klmorphismus b)\bigr)^{-1}\moranfang(\klmorphismus n\paar{\klmorphismus a}{\klmorphismus b}*_{\kategorie G}\klmorphismus C\paar{\klmorphismus b}{\klmorphismus g}*_{\kategorie G}\klmorphismus h)}\\
    &=\paar[bigg]{\paar[Big]{\klmorphismus a*_{\kategorie K}\klmorphismus b}
      {\bigl(\klmorphismus n\paar{\klmorphismus a}{\klmorphismus b}*_{\kategorie G}\klmorphismus C\paar{\klmorphismus b}{\klmorphismus g}*_{\kategorie G}\klmorphismus h\bigr)^{\Ob\kategorie G}}}
      {\\&\qquad\underbrace{
    \bigl(A(\klmorphismus a*_{\kategorie K}\klmorphismus b)\bigr)^{-1}
      \moranfang(\klmorphismus n\paar{\klmorphismus a}{\klmorphismus b})
    \bigl(\morende(\klmorphismus n\paar{\klmorphismus a}{\klmorphismus b})\bigr)^{-1}
    }_{\bigl(A(\klmorphismus b)A(\klmorphismus a)\bigr)^{-1}\text{\rlap{\qquad (c.f.\ Eq.~\eqref{eq:31})}}}
    \moranfang\bigl(\klmorphismus C\paar{\klmorphismus b}{\klmorphismus g}*_{\kategorie G}\klmorphismus h\bigr)}\\
    \intertext{Since $\Ob\kategorie G$ acts translatively on $\kategorie G$, this can be rewritten into}
  &=\paar[Big]{
    \paar[big]{
      \klmorphismus a*_{\kategorie K}
      \klmorphismus b
    }{
      \klmorphismus n
      \paar{
        \klmorphismus a
      }{
        \klmorphismus b
      }^{\Ob\kategorie G}
      *_{\orbitfaltigkeit{\kategorie G}{\Ob\kategorie G}}
      \klmorphismus C\paar{\klmorphismus b}{\klmorphismus g}^{\Ob\kategorie G}
      *_{\orbitfaltigkeit{\kategorie G}{\Ob\kategorie G}}
      \klmorphismus h^{\Ob\kategorie G}
    }
  }{\\&\qquad
    \bigl(A(\klmorphismus a)\bigr)^{-1}
    \underbrace{
    \bigl(A(\klmorphismus b)\bigr)^{-1}
    \moranfang\klmorphismus C\paar{\klmorphismus b}{\klmorphismus g}
    \bigl(\morende\klmorphismus C\paar{\klmorphismus b}{\klmorphismus g}\bigr)^{-1}
    }_{=\Bigl(A'\paar[big]{\klmorphismus b}{\klmorphismus C\paar{\klmorphismus b}{\klmorphismus g}}\Bigr)^{-1}
    =\bigl(A(\klmorphismus b)\morende\klmorphismus g(\moranfang\klmorphismus g)^{-1}\bigr)^{-1}
    }
    \moranfang(\klmorphismus h)}\\
    \intertext{We can apply Eq.~\eqref{eq:12b} in order to prove the equivalence under the brace. Then, we get}
    &=\paar[Big]{
      \paar{\klmorphismus a}{\klmorphismus g^{\Ob\kategorie G}}
      *_{\kategorie E\tupel{\kategorie K,\klmorphismus n,\klmorphismus
    C,\orbitfaltigkeit{\kategorie G}{\Ob\kategorie
      G}}}
      \paar{\klmorphismus b}{\klmorphismus h^{\Ob\kategorie G}}
      }{\\&\qquad\bigl(A(\klmorphismus a)\bigr)^{-1}
    \moranfang{\klmorphismus g}
    \underbrace{
    (\morende\klmorphismus g)^{-1}
    \bigl(A(\klmorphismus b)\bigr)^{-1}\moranfang\klmorphismus h)
    }_{=1}}\\
    \intertext{The brace can be omitted as
    $\morende_{\kategorie L}\paar{\klmorphismus a}{\klmorphismus g} =
    \moranfang_{\kategorie L}\paar{\klmorphismus b}{\klmorphismus h}$,
    and thus $\morende\klmorphismus g = \bigl(A(\klmorphismus b)\bigr)^{-1}\moranfang{\klmorphismus h}$.}
    &=\paar[Big]{
      \paar{\klmorphismus a}{\klmorphismus g^{\Ob\kategorie G}}
      *_{\kategorie E\tupel{\kategorie K,\klmorphismus n,\klmorphismus
    C,\orbitfaltigkeit{\kategorie G}{\Ob\kategorie
      G}}}
      \paar{\klmorphismus b}{\klmorphismus h^{\Ob\kategorie G}}
      }{\bigl(A(\klmorphismus a)\bigr)^{-1}\moranfang\klmorphismus g}\\
    &= \paar[Big]{\paar{\klmorphismus a}{\klmorphismus g}}{\bigl(A(\klmorphismus a)\bigr)^{-1}\moranfang\klmorphismus g}
      *_{\kategorie M}
    \paar[Big]{\paar{\klmorphismus b}{\klmorphismus h}}{\bigl(A(\klmorphismus b)\bigr)^{-1}\moranfang\klmorphismus h}\\
    &= Ψ\paar{\klmorphismus a}{\klmorphismus g}*_{\kategorie M}Ψ\paar{\klmorphismus b}{\klmorphismus h}
  \end{align*}

  So it remains to show the bijectivity of $Ψ$ from which follows the
  bijectivity of $Φ$. The injectivity follows directly from
  Eq.~$(*)$ below. Consider the mapping
  \begin{align*}
      \wabbildung{\abbildung{Ψ'&} {\Morall{\kategorie K}\times
                                \Morall{\orbitfaltigkeit{\kategorie G}{\Ob\kategorie G}}\times
                                 \Morall{\kategorie G}}
    {\Morall{\kategorie K}\times\Morall{\kategorie G}}
                                 \nonumber\\&}
      {\paar[Big]{\paar{\klmorphismus a}{\klmorphismus g^{\Ob\kategorie G}}}{h}}\tag{*}
    {\paar{\klmorphismus a}{\klmorphismus g'}\text{ where }
    \klmorphismus g'∈\klmorphismus g^{\Ob\kategorie G}\text{ with }
    \moranfang\klmorphismus g' = \bigl(A(\klmorphismus a)\bigr)h}.
  \end{align*}
  This mapping is well-defined since $\Ob\kategorie G$ acts
  translatively on $\kategorie G$. Obviously, the equation
  $Ψ\Bigl(Ψ'\paar[big]{\paar{\klmorphismus a}{\klmorphismus
      g^{\Ob\kategorie G}}}{h}\Bigr)= \paar[big]{\paar{\klmorphismus
      a}{\klmorphismus g^{\Ob\kategorie G}}}{h}$ holds for all arrows
  of $\kategorie M$. So $Ψ'$ is the inverse of $Ψ$ which proves the 
  bijectivity.
  \qed
\end{proof}

\noindent Now, we can consider the base structure of our representation as a
reflexive and transitive binary relation with some additional
decorations. If we know that the unfolded structure is a transitive relation,
it may be more helpful to preserve antisymmetry than transitivity.

Before we discuss this in detail, let us have some philosophical
remarks: In \cite{Schreier:1926,Schreier:1925}, Two mappings are
considered: $A^B$ which corresponds to $\klmorphismus C\paar BA$ and
$A_{B,B'}$ which corresponds to $\klmorphismus n\paar B{B'}$, here.
Both group extensions as well as unfoldings are reconstructions of
factored structures. And both constructions consider the factor
structure as transversal of the unfolded structure: They describe, how
the elements of the transversal are related to each other and how the
equivalence classes of the kernel of the canonical homomorphism are
mapped to each other.

In case of group extensions, the kernel of the canonical homomorphism
is defined by the partition of right cosets of a normal
subgroup. These are mapped to each other by internal homomorphisms
($A^B$) where the neutral elements of the group operations of the
cosets are the elements of a transversal.  The relationship between
the transversal elements is encoded by the elements $A_{B,B'}$. These
elements also encode how larger cycles are factored into smaller
cycles.

When we build representations with orbit categories, the category
structure is mainly encoded in the orbit category. The relationship
between the transversal elements and the location of particular arrows
are encoded in the annotation. So, the offset correction behaves
differently in unfoldings than $A_{B,B'}$ for Schreier's group extensions.
%There is no need for an offset
%correction like $A_{B,B'}$.
The relationship of arrows between the
elements of the same vertex congruence class is encoded in the
corresponding vertex monoid and the annotation.

During the transition to a flat representation the category structure
of the vertex monoids is transferred to the unfolding group, resulting
in a right-groupal category. So the connection between the individual
loops and the arrows between different vertices gets destroyed. This
must be reconstructed using the two operators $\klmorphismus C$ and
$\klmorphismus n$. So $\klmorphismus C$ plays the role of an
homomorphism and $\klmorphismus n$ encodes how the annotations match
each other.

As we will see later, the natural annotations of the images of
$\klmorphismus n$ change by inner automorphisms of the annotation
group when the underlying transversal changes. In contrast the
corresponding mapping for Schreier's group extensions may result in
arbitrary group elements when the underlying transversal changes.

In order to describe this, we introduce isomorphisms between flat
category representations.

\begin{definition}\label{def:flat-representations-isomorphism}
  Under the conditions of Lemma~\ref{lemma:flat-representations} two
  flat category representations
  $\tupel{\kategorie K,A,\hat{\klmorphismus
      n},\klmorphismus C, \kategorie G}$ and
  $\tupel{\kategorie L,B,\hat{\klmorphismus
      m},\klmorphismus D, \kategorie H}$ are called
  \defindex{isomorphic} if there exist a category isomorphism
  $\abbildung{φ}{\kategorie L}{\kategorie K}$, a groupal category isomorphism
  $\abbildung{ψ}{\kategorie H}{\kategorie G}$, and a mapping
  $\abbildung{h}{\Ob\kategorie L}{\Aut_{CAT}{\kategorie G}}$ from the
  vertices of category $\kategorie L$ into the set of category
  automorphisms on $\kategorie G$ such that for all
  $\klmorphismus a,\klmorphismus b∈\kategorie L$ and
  $\klmorphismus x∈\kategorie H$ and all vertices $x∈\Ob\kategorie H$ and $f,g∈\Ob\kategorie G$ the following equations hold:
  \begin{subequations}
    \begin{align}
      h(x)(\klmorphismus a)^{φ(g)}
      &= h(x)(\klmorphismus a^g)\label{eq:32}\\
      B(\klmorphismus a)
      &= h(\morende\klmorphismus a)^{-1}(1)A\bigl(φ(\klmorphismus a)\bigr)h(\moranfang\klmorphismus a)(1)\label{eq:49}\\
      \klmorphismus m\paar{\klmorphismus a}{\klmorphismus b}
      &= h(\morende\klmorphismus b)^{-1}(\klmorphismus n\paar{φ\klmorphismus a}{φ\klmorphismus b})\label{eq:50}\\
      \klmorphismus D\paar{\klmorphismus a}{\klmorphismus x}
      &= h(\morende\klmorphismus a)^{-1}\Bigl(\klmorphismus C\paar[big]{φ\klmorphismus a}{h(\moranfang\klmorphismus a)(\klmorphismus x)}\Bigr)\label{eq:51}
    \end{align}
  \end{subequations}
  In this case the triplet $\tupel{φ,ψ,h}$ is called \defindex{isomorphism}.
\end{definition}

\noindent The flat cateogry representations form a category together with the isomorphisms. This is shown next.

\begin{lemma}\label{lemma:flat-categ-concat-isomorphism}
  The concatenation of two isomorphisms of flat categories is also an isomorphism.
\end{lemma}

\begin{proof}
  Let 
  $\tupel{\kategorie K,A,\hat{\klmorphismus
      n},\klmorphismus C, \kategorie G}$,
  $\tupel{\kategorie L,B,\hat{\klmorphismus
      m},\klmorphismus D, \kategorie H}$ 
  and
  $\tupel{\kategorie M,C,\hat{\klmorphismus
      l},\klmorphismus E, \kategorie I}$ flat category representations.

  Let further the triplets
  $\abbildung{φ}{\kategorie L}{\kategorie K}$,
  $\abbildung{ψ}{\kategorie H}{\kategorie G}$,
  $\abbildung{h}{\Ob\kategorie L}{\Aut_{CAT}{\kategorie G}}$ and  
  $\abbildung{φ'}{\kategorie M}{\kategorie L}$, 
  $\abbildung{ψ'}{\kategorie I}{\kategorie H}$,
  $\abbildung{h'}{\Ob\kategorie M}{\Aut_{CAT}{\kategorie H}}$ isomorphisms.

  Then for the mappings
  \begin{align*}
    \wabbildung{
    \abbildung{\hat{φ}&}{\kategorie M}{\kategorie K}}
                                   {\klmorphismus a}
                                   {φ\bigl(φ'(\klmorphismus a)\bigr)}\\
    \wabbildung{\abbildung{\hat{ψ}&}{\kategorie I}{\kategorie G}}
                                   {\klmorphismus a}
                                   {ψ\bigl(ψ'(\klmorphismus a)\bigr)}\\
    \wabbildung{\abbildung{\hat h&}{\Ob\kategorie M}{\Aut_{CAT}{\kategorie G}}}
                                   {x}
                                   {h'(x)\circ h\bigl(φ(x)\bigr)}
  \end{align*}
  The following equations hold:
  \begin{align*}
    \hat h(x)(\klmorphismus b^g)
    &= h\bigl(φ(x)\bigr)\bigl(h'(x)(\klmorphismus b^g)\bigr)
    = h\bigl(φ(x)\bigr)\bigl(h'(x)(\klmorphismus b)^{φ'(g)}\bigr)\\
    &= h\bigl(φ(x)\bigr)\bigl(h'(x)(\klmorphismus b)\bigr)^{φ\bigl(φ'(g)\bigr)}
      = h\bigl(φ(x)\bigr)\bigl(h'(x)(\klmorphismus b)\bigr)^{\hat {φ}(g)}\\
    &= \hat h(x)(\klmorphismus b^g)^{\hat {φ}(g)}\\
    C(\klmorphismus a)
    &=h'(\morende\klmorphismus a)^{-1}(1)B\bigl(φ'(\klmorphismus a)\bigr)h'(\moranfang\klmorphismus a)(1)\\
    &=h'(\morende\klmorphismus a)^{-1}(1)h(\morende φ'\klmorphismus a)^{-1}(1)A\bigl(φ(φ'\klmorphismus a)\bigr)h(\moranfang φ'\klmorphismus a)(1)h'(\moranfang\klmorphismus a)(1)\\
    &=h'(\morende\klmorphismus a)^{-1}\bigl(h(\morende φ'\klmorphismus a)^{-1}(1)\bigr)A\bigl(φ(φ'\klmorphismus a)\bigr)h(\moranfang φ'\klmorphismus a)\bigl(h'(\moranfang\klmorphismus a)(1)\bigr)\\
    &=\Bigl(h'(\morende\klmorphismus a)\circ h\bigl(φ'(\morende \klmorphismus a)\bigr)\Bigr)^{-1}(1)
      A\bigl(φ(φ'\klmorphismus a)\bigr)
      \Bigl(h'(\moranfang\klmorphismus a)\circ h\bigl( φ'(\moranfang\klmorphismus a)\bigr)\Bigr)(1)\\
    &=\bigl(\hat h(\morende\klmorphismus a)\bigr)^{-1}(1)
      A\bigl(\hat {φ}(\klmorphismus a)\bigr)
      \hat h(\moranfang\klmorphismus a)(1)\\
    \hat{\klmorphismus l}\paar{\klmorphismus a}{\klmorphismus b}
      &= h'(\morende\klmorphismus b)^{-1} \bigl(\hat{\klmorphismus m}\paar{φ'\klmorphismus a}{φ'\klmorphismus b}\bigr)\\
      &= h'(\morende\klmorphismus b)^{-1}\biggl(h\bigl(\morende(φ'\klmorphismus b)\bigr)^{-1}\Bigl(\hat{\klmorphismus n}\paar[big]{φ(φ'\klmorphismus a)}{φ(φ'\klmorphismus b)}\Bigr)\biggr)\\
    &= \hat h(\morende\klmorphismus b)^{-1}\bigl(\hat{\klmorphismus n}\paar{\hat{φ}\klmorphismus a}{\hat {φ}\klmorphismus b}\bigr)\\
      \klmorphismus E\paar{\klmorphismus a}{\klmorphismus x}
      &= h'(\morende\klmorphismus a)^{-1}\Bigl(\klmorphismus D\paar[big]{φ'\klmorphismus a}{h'(\moranfang\klmorphismus a)(\klmorphismus x)}\Bigr)\\
      \klmorphismus E\paar{\klmorphismus a}{\klmorphismus x}
      &= h'(\morende\klmorphismus a)^{-1}\Biggl(h\bigl(\morende(φ'\klmorphismus a)\bigr)^{-1}\biggl(\klmorphismus C\paar[Big]{φ(φ'\klmorphismus a)}{h\bigl(\moranfang(φ'\klmorphismus a)\bigr)\bigl(h'(\moranfang\klmorphismus a)(\klmorphismus x)\bigr)}\biggr)\Biggr)\\
      &= \Bigl(h'(\morende\klmorphismus a)\circ h\bigl(\morende(φ'\klmorphismus a)\bigr)\Bigr)^{-1}\Biggl(\klmorphismus C\paar[bigg]{\hat {φ}(\klmorphismus a)}{\Bigl(h'(\moranfang\klmorphismus a)\circ h\bigl(\moranfang(φ'\klmorphismus a)\bigr)\Bigr)(\klmorphismus x)}\Biggr)\\
      &= \bigl(\hat h(\morende\klmorphismus a)\bigr)^{-1}\Bigl(\klmorphismus C\paar[big]{\hat {φ}(\klmorphismus a)}{\hat h(\moranfang\klmorphismus a)(\klmorphismus x)}\Bigr)\\
  \end{align*}
  Obviously $\hat{ψ}$ and $\hat {φ}$ are isomorphisms, which fulfil
  Equations~(\ref{eq:32}) to (\Ref{eq:51}). So
  $\tupel{\kategorie K,A,\hat{\klmorphismus n},\klmorphismus C,
    \kategorie G}$, is isomorphic to
  $\tupel{\kategorie M,C,\hat{\klmorphismus l},\klmorphismus E,
    \kategorie I}$.
  \qed
\end{proof}

\noindent In fact, the isomorphisms are invertible. We leave the proof to the
interested reader. We use the term “isomorphism”, here, because two
isomorphic flat category representations unfold into isomorphic
categories. So they are isomorphic in the sense that they represent
essentially the same thing.

\begin{lemma}
  The unfoldings of two isomorphic flat category representations are isomorphic categories.
\end{lemma}

\begin{proof}
  We use the same notations as in Lemma~\ref{def:flat-representations-isomorphism}.
  At first we consider the case where $φ$ and $ψ$ are trivial. Then Equations~\eqref{eq:32} to~\eqref{eq:51} are reduced to
  \begin{align*}
    h(x)(g)^f
    &= h(x)(g^f)= h(x)(gf)=h(x)(1)gf\\
    B(\klmorphismus a)
    &= h(\morende\klmorphismus a)^{-1}(1)A(\klmorphismus a)h(\moranfang\klmorphismus a)(1)\\
    \hat{\klmorphismus m}\paar{\klmorphismus a}{\klmorphismus b}
    &= h(\morende\klmorphismus b)^{-1}(\hat{\klmorphismus n}\paar{\klmorphismus a}{\klmorphismus b})\\
    \klmorphismus D\paar{\klmorphismus a}{\klmorphismus x}
    &= h(\morende\klmorphismus a)^{-1}\Bigl(\klmorphismus C\paar[big]{\klmorphismus a}{h(\moranfang\klmorphismus a)(\klmorphismus x)}\Bigr)
  \end{align*}
  Then, the unfoldiding $\entfaltung{\kategorie K}{B,\hat{\klmorphismus m},\klmorphismus D}{\kategorie G}$ is defined by
  \begin{align*}
    \Ob{\entfaltung{\kategorie K}{B,\hat{\klmorphismus m},\klmorphismus D}{\kategorie G}}
    &\definiert \Ob(\kategorie K \times \kategorie G),\\
    \Mor{\entfaltung{\kategorie K}{B,\hat{\klmorphismus m},\klmorphismus D}{\kategorie G}}[big]{\paar xg}{\paar yh}
    &\definiert \Menge{\paar{\klmorphismus a}{\klmorphismus g}∈\Morall{\kategorie K\times \kategorie G}}{\morende \klmorphismus g = h, \moranfang\klmorphismus g = B(\klmorphismus a)g}\text{ and}\\
    \paar{\klmorphismus a}{\klmorphismus g}*_{\entfaltung{\kategorie K}{B,\hat{\klmorphismus m},\klmorphismus D}{\kategorie G}} \paar{\klmorphismus b}{\klmorphismus h}
    &\definiert \paar{\klmorphismus a * \klmorphismus b}{\hat{\klmorphismus m}\paar{\klmorphismus a}{\klmorphismus b}*\klmorphismus D\paar{\klmorphismus b}{\klmorphismus g}*\klmorphismus h}
  \end{align*}
  Now, we substitute $B$, $\hat{\klmorphismus m}$ and $\klmorphismus \klmorphismus D$ according to the Equations~\eqref{eq:32} to~(\ref{eq:51}):
  \begin{align*}
    \Mor{{\entfaltung{\kategorie K}{
    B,\hat{\klmorphismus m},\klmorphismus D}{\kategorie G}}}[big]{
    &\paar xg}{\paar yh}\\
    &= \Menge{\paar{\klmorphismus a}{\klmorphismus g}∈\Morall{\kategorie K\times \kategorie G}}
      {\morende \klmorphismus g = h,
      \moranfang\klmorphismus g = h(\morende\klmorphismus a)^{-1}(1)A(\klmorphismus a)h(\moranfang\klmorphismus a)(1)g}\\
    &= \Menge{\paar{\klmorphismus a}{\klmorphismus g}∈\Morall{\kategorie K\times \kategorie G}}{h(\morende\klmorphismus a)(\morende \klmorphismus g) = h(\morende\klmorphismus a)(h),
      h(\morende\klmorphismus a)(\moranfang\klmorphismus g) = A(\klmorphismus a)h(\moranfang\klmorphismus a)(g)}\\
    &= \Menge[big]{\paar[big]{\klmorphismus a}{h(\morende\klmorphismus a)^{-1}(\klmorphismus g)}}
      {\paar{\klmorphismus a}{\klmorphismus g}∈\Morall{\kategorie K\times \kategorie G},
      \morende \klmorphismus g = h,
      \moranfang\klmorphismus g = A(\klmorphismus a)g}\\
    \paar{\klmorphismus a}{\klmorphismus g}
    &*_{(\entfaltung{\kategorie K}{B,\hat{\klmorphismus m},\klmorphismus D}{\kategorie G})} \paar{\klmorphismus b}{\klmorphismus h}\\
    &= \paar[bigg]{\klmorphismus a * \klmorphismus b}{h(\morende\klmorphismus b)^{-1}(\hat{\klmorphismus n}\paar{\klmorphismus a}{\klmorphismus b})
      *h(\morende\klmorphismus b)^{-1}\Bigl(\klmorphismus C\paar[big]{\klmorphismus b}{h(\moranfang\klmorphismus b)(\klmorphismus g)}\Bigr)
      *\klmorphismus h}\\
    &= \paar[bigg]{\klmorphismus a * \klmorphismus b}{h(\morende\klmorphismus b)^{-1}\Bigl(\hat{\klmorphismus n}\paar{\klmorphismus a}{\klmorphismus b}
      *\klmorphismus C\paar[big]{\klmorphismus b}{h(\moranfang\klmorphismus b)(\klmorphismus g)}
      *h(\morende\klmorphismus b)(\klmorphismus h)\Bigr)}
  \end{align*}
  So, obviously, the mapping
  \begin{align*}
    \wabbildung{\abbildung{Φ}{\Morall{{\entfaltung{\kategorie K}{
    B,\hat{\klmorphismus m},\klmorphismus D}{\kategorie G}}}}{\Morall{{\entfaltung{\kategorie K}{
    A,\hat{\klmorphismus n},\klmorphismus C}{\kategorie G}}}}}
    {\paar{\klmorphismus a}{\klmorphismus g}}
    {\paar[big]{\klmorphismus a}{h(\morende\klmorphismus a)(\klmorphismus g)}}
  \end{align*}
  is a bijection and compatible with the category concatenation. Thus,
  it raises a category isomorphism between
  $\entfaltung{\kategorie K}{ B,\hat{\klmorphismus m},\klmorphismus
    D}{\kategorie G}$ and
  $\entfaltung{\kategorie K}{ A,\hat{\klmorphismus n},\klmorphismus
    C}{\kategorie G}$.

  As the isomorphisms $φ$ and $ψ$ are compatible with the category
  structure and $ψ$ is also an isomorphism with respect to the group
  operation, we can substitute them in the corresponding equations
  without changing their syntactic applicability. Thus
  \begin{align*}
    \wabbildung{\abbildung{Ψ}{\Morall{{\entfaltung{\kategorie L}{
    B,\hat{\klmorphismus m},\klmorphismus D}{\kategorie H}}}}{\Morall{{\entfaltung{\kategorie K}{
    A,\hat{\klmorphismus n},\klmorphismus C}{\kategorie G}}}}}
    {\paar{\klmorphismus a}{\klmorphismus g}}
    {\paar[big]{φ\klmorphismus a}{h(\morende\klmorphismus a)(ψ\klmorphismus g)}}
  \end{align*}
  raises a category isomorphism between
  $\entfaltung{\kategorie L}{ B,\hat{\klmorphismus m},\klmorphismus
    D}{\kategorie H}$ and
  $\entfaltung{\kategorie K}{ A,\hat{\klmorphismus n},\klmorphismus
    C}{\kategorie G}$.
  \qed
\end{proof}

\noindent The isomorphism $Φ$ in the previous proof corresponds to a change of
the transversal for a natural annotation. Equation~\eqref{eq:49} tells
us that the natural annotations of the unfoldings of any image of the
mapping $\hat{\klmorphismus n}$ are related by inner automorphisms of the
group $\Ob\kategorie G$: The equations 

\begin{gather*}
  B(\klmorphismus b)B(\klmorphismus a)
  =\morende\klmorphismus m\paar{\klmorphismus a}{\klmorphismus b}
    \bigl(\moranfang\klmorphismus m\paar{\klmorphismus a}{\klmorphismus b}\bigr)^{-1}
    B(\klmorphismus a\bullet\klmorphismus b)\\
  A(φ\klmorphismus b)A(φ\klmorphismus a)
  =\morende\klmorphismus n\paar{φ\klmorphismus a}{φ\klmorphismus b}
    \bigl(φ\moranfang\klmorphismus a\paar{φ\klmorphismus a}{φ\klmorphismus b}\bigr)^{-1}
    A(φ\klmorphismus a\bullet\klmorphismus b)
    \intertext{lead to the equations}
  \begin{multlined}
    ψ(\morende\klmorphismus m\paar{\klmorphismus a}{\klmorphismus b})
    ψ\bigl(\moranfang\klmorphismus m\paar{\klmorphismus a}{\klmorphismus b}\bigr)^{-1}
    h\bigl(\morende\klmorphismus (a\bullet\klmorphismus b)\bigr)^{-1}(1)A
    \bigl(φ(\klmorphismus a\bullet\klmorphismus b)\bigr)h\bigl(\moranfang(\klmorphismus a\bullet\klmorphismus b)\bigr)(1)=
    \\
    = h(\morende\klmorphismus b)^{-1}(1)A\bigl(φ(\klmorphismus b)\bigr)h(\moranfang\klmorphismus b)(1)
    h(\morende\klmorphismus a)^{-1}(1)A\bigl(φ(\klmorphismus
    a)\bigr)h(\moranfang\klmorphismus a)(1),\text{ and}
  \end{multlined}\\
  \morende\klmorphismus n\paar{φ\klmorphismus a}{φ\klmorphismus b}
  \bigl(\moranfang\klmorphismus n\paar{φ\klmorphismus a}{φ\klmorphismus b}\bigr)^{-1}
  A
  \bigl(φ(\klmorphismus a\bullet\klmorphismus b)\bigr)
  = A\bigl(φ(\klmorphismus b)\bigr)
  A\bigl(φ(\klmorphismus a)\bigr)
%\end{gather*}
\intertext{These can be combined into a single equation connecting $\klmorphismus n$ and $\klmorphismus m$.}
%\begin{gather*}
  ψ(\morende\klmorphismus m\paar{\klmorphismus a}{\klmorphismus b})
  ψ\bigl(\moranfang\klmorphismus m\paar{\klmorphismus a}{\klmorphismus b}\bigr)^{-1}
%  \\
  = h(\morende\klmorphismus b)^{-1}(1)
  \morende\klmorphismus n\paar{φ\klmorphismus a}{φ\klmorphismus b}
  \bigl(\moranfang\klmorphismus n\paar{φ\klmorphismus a}{φ\klmorphismus b}\bigr)^{-1}
  h(\morende\klmorphismus b)(1)
\end{gather*}

This is especially important for those arrows
$\klmorphismus a,\klmorphismus b$ where
$\klmorphismus n\paar{\klmorphismus a}{\klmorphismus b}$ is a loop. In
this case being a loop does not depend on the choice of the
transversal elements of a natural annotation. So, we can consider the
pair of arrows $\paar{\klmorphismus a}{\klmorphismus b}$ to be
consistent. On the other hand, we would call a pair of arrows
$\klmorphismus a,\klmorphismus b$ a long pair if
$\klmorphismus n\paar{\klmorphismus a}{\klmorphismus b}\neq 1$. In
case of a po-group a long pair is essentially a pair of arrows such
that for all convex transversals at least one of the two or their
concatenation comes from or points to a vertex outside of the convex
transversal.

\section{Flat representations}\label{sec:flat-representations}

Being a long pair does not necessarily mean that one of the arrows is
redundant. Examples of po-groups with uniquely representable orbit
categories can easily be constructed. On the other hand the orbit
category
$\orbitfaltigkeit{\paar{\rationalezahlen\times\rationalezahlen}{\leq})}{(12\ganzzahlen\times12\ganzzahlen)}$
has a system of pairwise consistent arrows that – in combination with a vertex category – represent the whole category itself.

Now, we will construct a structure that represents the antisymmetry of
the flat orbit category in the flat category representation. This
means, that we must restrict ourselves to partial categories, where
the concatenation is not always defined. In our case, each partial
category is constructed with a category in mind.

\begin{definition}
  Let $\kategorie K$ denote a category with concatenation $*$. Any subgraph
  $\kategorie P$ of $\kategorie K$ together with a partial binary
  opertator $\cdot$ called \defindex{concatenation} is called
  \defindex{partial subcategory} of $\kategorie K$, if the following
  conditions are met:
  \begin{enumerate}    
  \item Every vertex has an identity loop.
  \item For any two arrows
    $\klmorphismus a,\klmorphismus b∈\Morall{\kategorie P}$ the
    concatenation
    $\klmorphismus a\cdot \klmorphismus b=\klmorphismus
    a*\klmorphismus b$ exists iff the concatenation
    $\klmorphismus a*\klmorphismus b∈ \Morall{\kategorie P}$ exists in
    $\kategorie P$.
  \end{enumerate}

  A partial subcategory $\kategorie P$ of $\kategorie K$ is called
  \defindex{full} if there exists a homomorphism from the path
  category $\pfadkategorie{\kategorie P}$ onto $\kategorie K$ that is
  compatible with the concatenation in $\kategorie P$.

  The partial subcategory $\kategorie P$ is called \defindex{fully defining} if for the
  congruence relation $\equiv$ defined on the path category
  $\pfadkategorie{\kategorie P}$ and generated by formula
  \begin{equation}
    \klmorphismus a *\klmorphismus b \equiv \klmorphismus a\cdot \klmorphismus b
  \end{equation}
  the factor category $\faktorisiert {\pfadkategorie {\kategorie P}} {{\equiv}}$ is isomorphic to $\kategorie K$.

  It is called \defindex{flat defining} if for the 
  congruence relation $\bowtie$ generated by the formula
  \begin{equation}
    \label{eq:30}
    \forall x∈\Ob\kategorie P, \klmorphismus x∈\Mor{\klmorphismus P}{x}{x}: \klmorphismus x\bowtie \id_x
  \end{equation}
  the category $\faktorisiert {\pfadkategorie {\kategorie P}} {({\equiv}\vee {\bowtie})}$ is isomorphic to $\kategorie K$.
\end{definition}

\noindent With
Lemma~\ref{lemma:simele-flat-orbit-category-when-simple} and
\ref{lemma:flat-categories-no-parallel-arrows-in-circles} we can tell,
when uniquely representable orbit categories of po-groups give rise to
simple partial categories.

Our goal is now, that we want to express flat category representations
in the sense of partial categories. This means we somehow need a way
to reconstruct the flat category representation. If a flat category
representation
$\tupel{\kategorie K,A,\klmorphismus n,\klmorphismus C,\kategorie G}$
has a simple vertex category $\kategorie G$, there is at most one
arrow between any two vertices in $\kategorie G$ and thus, the mapping
$\klmorphismus n$ can be reconstructed according to
Equation~\ref{eq:34} from the annotation $A$ and the category
concatenation. Then,
$\klmorphismus n\paar{\klmorphismus a}{\klmorphismus b}$ is the arrow
between $\morende\klmorphismus b$ and
$\morende\klmorphismus b^{A(\klmorphismus b)A(\klmorphismus
  a)A\bigl(\klmorphismus r(\klmorphismus a * \klmorphismus
  b)\bigr)^{-1}}$. Applying this rule recursively, we can extend this
construction to the concatenation of any paths in $\kategorie P$. This
implies that the partial subcategory must include all concatenated arrows
$\klmorphismus a$ for which the annotation
$A(\klmorphismus b * \klmorphismus c)$ differs from the product
$A(\klmorphismus c)A(\klmorphismus b)$ of all possible combinations
$\klmorphismus a = \klmorphismus b * \klmorphismus c$.

In the same way we can reconstruct the mapping $\klmorphismus C$ from
$\klmorphismus C\eingeschrmenge{\Morall{\kategorie
    P}\times\Morall{\kategorie G}}$, by iteratively applying
$\klmorphismus C\eingeschrmenge{\Morall{\kategorie
    P}\times\Morall{\kategorie G}}$ on all arrows of the paths
according to Equation~\ref{eq:45}. This is well-defined if
$\kategorie G$ is a simple category.

This means that under certain conditions (e.g. the above mentioned
ones) we can define a flat representation based on a partial subcategory.

Given a category $\kategorie K$ and an annotation
$\abbildung A{\kategorie K}{\gruppe}$ into a group then for each
partial subcategory $\kategorie P\leq \kategorie K$ we call
$A\eingeschrmenge{\kategorie P}$ \defindex{annotation} of
$\kategorie P$.

\begin{definition}
  Let
  $\tupel{\kategorie K,A,\klmorphismus n, \klmorphismus C, \kategorie
    G}$ be a flat category representation and $\kategorie P$ be a partial
  category of $\kategorie K$. Then we call the tuple
  $\tupel{\kategorie P,A,\klmorphismus n, \klmorphismus C, \kategorie
    G}$ \defindex{flat representation}. It is called
  \defindex{simple}, iff $\kategorie P$ is a simple graph, and it is
  called \defindex{faithful} iff $A$ is faithful.

  If $\kategorie Q \leq \kategorie P$ is a partial subcategory of
  $\kategorie K$ and a subgraph of $\kategorie P$, the tuple
  $\tupel{\kategorie Q, A\eingeschrmenge{\kategorie Q}, \klmorphismus
    n, \klmorphismus C, \kategorie G}$ is called flat
  subrepresentation of
  $\tupel{\kategorie P,A,\klmorphismus n, \klmorphismus C, \kategorie
    G}$.
\end{definition}

We say a flat category representation
$\tupel{\kategorie K, \hat A,\hat{\klmorphismus n}, \hat{\klmorphismus
    C}, \kategorie G}$ is called a \defindex{completion} of a flat
representation
$\tupel{\kategorie P,A,\klmorphismus n, \klmorphismus C,\kategorie
  G}$, if there exists an embedding
$\abbildung{φ}{\kategorie P}{\kategorie K}$ such that $\kategorie P$
is a fully defining partial subcategory of $\kategorie K$ and
$A = φ\circ \hat A\eingeschrmenge{φ[\Morall{\kategorie P}]}$,
$\klmorphismus n = (φ\times φ)\circ\hat{\klmorphismus
  n}\eingeschrmenge{φ[P]\times φ[P]}$, and
$\klmorphismus C = (φ\times \id)\circ\hat{\klmorphismus
  C}\eingeschrmenge{φ[P]\times \kategorie G}$. So that for the
partial mappings $\klmorphismus n$, $\hat{\klmorphismus n}$,
$\klmorphismus C$, $\hat{\klmorphismus C}$ and the annotations $A$ and
$\hat A$ the following diagrams commute.

\noindent{%
  \hfill
  \begin{tikzcd}
    {\Morall{\kategorie P}\times\Morall{\kategorie P}}\arrow{r}{φ\times φ}\arrow{rd}{\klmorphismus n}&
    {\Morall{\kategorie{K}}\times\Morall{\kategorie K}}\arrow{d}{\hat {\klmorphismus n}} \\
    &\Morall{\kategorie G}
  \end{tikzcd}
  \hfill
  \begin{tikzcd}
    {\Morall{\kategorie P}\times\Morall{\kategorie G}}\arrow{r}{φ\times\id}\arrow{rd}{\klmorphismus C}&
    {\Morall{\kategorie{K}}\times\Morall{\kategorie G}}\arrow{d}{\hat {\klmorphismus C}} \\
    &\Morall{\kategorie G}
  \end{tikzcd}
  \hfill\\\strut
  \hfill
  \begin{tikzcd}
    {\kategorie P}\arrow{r}{φ}\arrow{rd}{A}&
    \kategorie{K} \arrow{d}{\hat A} \\
    &\Ob\kategorie G
  \end{tikzcd}
  \hfill
  \strut
}
  
Obviously a flat category representation of a simple
category is faithful. Consequently its flat representations are also
faithful.

The largest category that can be generated by a graph is its
path category. This implies that there is a homomorphism from the path
category of $\kategorie P$ into $\kategorie K$. As the annotation of
the concatenation of two arrows is the product of the annotations of
both of them, it is uniquely defined. Thus, a faithful completion of a
simple flat representation is the representation with the smallest
possible category with respect to vertex-injective category
homomorphisms that preserve the kernel of the annotation. Every
category that is a non-injective image of the faithful completion must
identify two arrows with different annotations.

\noindent{%
  \hfill
  \begin{tikzcd}
    {\kategorie P}\arrow{r}{id}\arrow{rd}{A}&
    \pfadkategorie {\kategorie P} \arrow{r}{φ} &
    \kategorie{K} \arrow{ld}{A} \\
    & 
    \gruppe&
  \end{tikzcd}
  \hfill
  \strut
}

There are other properties of representations, flat category
representations and flat representations that might be useful for
us. If $\tupel{\kategorie K,A,\gruppe}$ is a representation or $\tupel{\kategorie K, A,\klmorphismus n, \klmorphismus C, \kategorie G}$ is a flat representation, it is called
\begin{description}[{antisymmetrically $\gruppe[S]$-complete}]
\item[faithful] if the annotation is faithful,  
\item[simple] if the partial subcategory is simple
\item[ordered] if the category $\kategorie K$ is an ordered set
\item[{$\gruppe[S]$-symmetric}] if $\gruppe[S]\leq \Aut\kategorie K$ is
  a transitive automorphism group of the partial subcategory,
\item[{translatively $\gruppe[S]$-symmetric}] if $\gruppe[S]$ is a
  translative automorphism group of the partial subcategory and the
  representation is $\gruppe[S]$-symmetric.
\item[antisymmetric] if $A(\klmorphismus a)A(\klmorphismus b) = 1$
  implies $\morende \klmorphismus a \neq \moranfang \klmorphismus b$
  or $\moranfang\klmorphismus a\neq\morende \klmorphismus b$ in $\kategorie K$.
\item[complete wrt. $\kategorie K$] if there is a completion and an
  automorphism from the completion into the category $\kategorie K$.
\item[{antisymmetrically $\gruppe[S]$-complete}] If it is antisymmetric,
  $\gruppe[S]$-symmetric and complete, and it is not a proper flat subrepresentation of any  antisymmetric,
  $\gruppe[S]$-symmetric and complete flat representation.
\end{description}

A po-group is a relational structure. Thus, we specialise on
relational fundamental systems. This implies that we consider such
flat representations whose category is a simple graph: simple flat
representations. In order to keep the direction information we should also focus
on antisymmetric flat representations. As in a po-group every group
element acts as an automorphism on the order relation we want this
behaviour also for the factor group and the corresponding relation.

The $\gruppe[S]$-symmetry is of special interest for us, as this is
the property which ensures that the factorisation is
structure-preserving for convex normal subgroups of po-groups and
$\ell$-groups. The factorisation of a po-group is symmetric with
respect to its factor group. So there is still some hope to preserve
this property during “simplification” of the orbit category to a
simple partial subcategory.

\begin{lemma}
  Let $\kategorie K = \Halbord{\gruppe}{\leq}$ denote a po-group,
  $\gruppe[N]\normalteilervon \gruppe$ a normal subgroup of $\gruppe$.
  Then, every simple flat category representation
  $\tupel{\klmorphismus r[\orbitfaltigkeit {\kategorie
      K}{\gruppe[N]}],A,\klmorphismus n,\klmorphismus C,\gruppe[N]}$
  of $\kategorie K$ is translatively
  $\faktorisiert{\gruppe}{\gruppe[N]}$-symmetric.
\end{lemma}
\begin{proof}
  Recall, $\gruppe$ is translative, and so is $\gruppe[N]$. Let
  $g∈\gruppe$ and $n,\hat n∈\gruppe[N]$ denote three elements of the group
  and its normal subgroup. All three act on $\gruppe$ as order
  automorphisms.  For any arrows
  $\klmorphismus a∈\Morall{\kategorie K}$ there exists a group element
  $\bar n ∈\gruppe[N]$ such that
  $ (\moranfang\klmorphismus a^n)^{(g\hat n)} =
  (\moranfang\klmorphismus a)^{(g\hat n)\bar n}$ holds. The same is true
  if we exchange left and right. As either $n$ or $\bar n$, and –
  independently from them – $\hat n$ can be chosen freely, this leads to
  the equation
  $ (\moranfang\klmorphismus a)^{\gruppe[N](g\gruppe[N])} =
  (\moranfang\klmorphismus a)^{(g\gruppe[N])\gruppe[N]} =
  \moranfang(\klmorphismus a^{g\gruppe[N]})$ from the corresponding set
  inclusions. In the same way we get the equations

  \begin{align*}
    (\morende\klmorphismus a)^{\gruppe[N](g\gruppe[N])}
    &= \morende(\klmorphismus a^{g\gruppe[N]})
    &&\text{and}\\
    \klmorphismus a^{\gruppe[N](g\gruppe[N])} * \klmorphismus b^{\gruppe[N](g\gruppe[N])}
    &= (\klmorphismus a * \klmorphismus b)^{g\gruppe[N]}
    &&\text{iff $\klmorphismus a * \klmorphismus b$ exists.}
  \end{align*}
  Consequently the induced action of
  $\faktorisiert{\gruppe}{\gruppe[N]}$ is an automorphism action on
  $\orbitfaltigkeit{\kategorie K}{\gruppe[N]}$.

  As $\klmorphismus r[\orbitfaltigkeit{\kategorie K}{\gruppe[N]}]$ is
  a simple category, applying $\klmorphismus r$ commutes with the
  action of $\faktorisiert{\gruppe}{\gruppe[N]}$, which shows that
  $\faktorisiert{\gruppe}{\gruppe[N]}$ acts regular on
  $\klmorphismus r[\orbitfaltigkeit{\kategorie K}{\gruppe[N]}]$. So
  this is a translative group action.
  \qed
\end{proof}

This allows us to introduce the well-known factorisation results from
classical theory of po-groups into this framework. First we want
to mention the fact, that for every po-group
$\kategorie K = \Halbord{\gruppe}{\leq}$ and every of its convex
normal subgroups $\gruppe[N]\normalteilervon\gruppe$ the orbit
category $\orbitfaltigkeit{\kategorie K}{\gruppe[N]}$ is isomorphic to
the factor group $\faktorisiert{\kategorie K}{\gruppe[N]}$, when the
latter is considered as a category. This isomorphism is fully
defined by its restriction to the objects.

\begin{theorem}
  The factorisation of a po-group
  $\kategorie K=\Halbord{\gruppe}{\leq}$ by a convex normal subgroup
  $\gruppe[N]$ together with the natural annotation give rise to a
  faithful translative antisymmetrically
  $\faktorisiert{\gruppe}{\gruppe[N]}$-complete flat representation of
  $\kategorie K$.
\end{theorem}
\begin{proof}
  
  The factorisation and its natural annotation define the
  representation
  $\tupel{\orbitfaltigkeit{\kategorie
      K}{\gruppe[N]},A_{\gruppe[N]},\gruppe[N]}$.  It is complete,
  shown by the unfolding, it is antisymmetric, the underlying simple category of the partial subcategory is
  an ordered set, which we know from the theory of po-groups. Consequently the representation is anti-symmetrical
  and faithful. Finally it is translatively
  $\faktorisiert{\gruppe}{\gruppe[N]}$-symmetric as shown in the
  previous lemma.
  \qed
\end{proof}

Given a category $\kategorie K$, its partial subcategories
together with the canonical embeddings form an ordered set
$\Halbord P\leq$, actually a complete meet-semilattice. The partial
categories which contain all vertices form a proper subset
$\Halbord{P_V}{\leq\eingeschrmenge {P_V}}$ of $\Halbord P\leq$. As the
discrete category is always a partial subcategory, there are simple
partial categories among the elements of $P_V$.

When we fix the annotation $A$ and the two partial mappings
$\klmorphismus n$ and $\klmorphismus C$, we can order the set of flat
representations of a given representation according to the order
relation $\leq\eingeschrmenge {P_V}$ of the partial subcategory of the
representations. We will call this order \defindex{induced order} in
this context. In order to find the best candidates of flat
representations, we can analyse this ordered set if we find some
maximal or minimal flat representations.

\begin{description}
\item[faithfulness] If a mapping is faithful, its restriction to a
  subset is also faithful. Thus the faithful flat representations form
  an order ideal. If a non-faithful flat representation has a faithful
  flat subrepresentation, this can happen only by removing conflicting
  arrows. Consequently the union of a chain of faithful
  representations is also faithful, which implies that there are
  maximal faithful flat representations.
\item[simplicity] Obviously the simple flat representations form an
  order ideal, too. The union of the partial categories of a chain of simple
  flat representations is simple. So also the ideal of simple flat
  representations has maximal elements.
\item[ordering] The intersection of a set of ordered sets is an
  ordered set. Thus, the ordered flat representations together with
  the maximal flat representation form a closure system.
\item[{$\gruppe[S]$-symmetric}] The discrete category is also
  symmetric with respect to the symmetric group of its vertices. Given a
  fixed permutation group $\gruppe[S]$, $\gruppe[S]$-symmetric partial
  categories below a given $\gruppe[S]$-symmetric (partial subcategory)
  form a closure system: The intersection of a set of
  $\gruppe[S]$-symmetric partial subcategories is always also
  $\gruppe[S]$-symmetric. Consequently the simple
  $\gruppe[S]$-symmetric partial categories together with the maximal
  flat representation form a closure system.
\item[{translatively $\gruppe[S]$-symmetric}] As the action on the
  vertices is fixed the intersection of a set of translatively
  $\gruppe[S]$-symmetric partial categories is still translatively
  $\gruppe[S]$-symmetric. Thus the translatively
  $\gruppe[S]$-symmetric flat representations form a closure system.
\item[antisymmetric] The antisymmetric partial subcategories form an
  order ideal. When we consider a chain in this ideal, its union is
  still antisymmetric.
\item[complete wrt. $\kategorie K$] If a flat representation is not
  complete with respect to a category $\kategorie K$, then all its
  flat subrepresentations are also not complete. So the complete
  subrepresentations form an order filter.
\item[{antisymmetrically $\gruppe[S]$-complete}] If there exist
  antisymmetrically $\gruppe[S]$-complete flat representations they
  have maximal elements.
\end{description}

So the set of (simple) faithful antisymmetric translatively
$\faktorisiert {\gruppe}{\gruppe[N]}$-complete flat representations of
$\Halbord {\gruppe}{\leq}$ is either empty or has maximal elements.

We call a maximal faithful, antisymmetric, translatively
$\faktorisiert {\gruppe}{\gruppe[N]}$-complete flat representation
\defindex{cyclically fundamental}.

This catalogue of features has the advantage, that different
structures can be chosen depending on the intended use. For
completeness, we should also consider approaches as the abridged
annotation introduced in \cite{Borchmann:2009}, which considers the
neighbourhood relation (also: Hasse relation) of a preordered
set. However, this does not necessarily exist for every po-group.

%\todo{Sinnvolle Kombinationen eingrenzen}
%\todo{Faktorisierungen nach konvexen Normalteilern darstellen}

\section{Further results and applications to music theory}\label{sec:furth-results-appl}

This article provides a fundamental construction for a certain type of
factorisations of po-groups and similar small categories into
orbits of certain automorphism groups. This is a starting point for
further investigations. For example for the mathemusical applications
it is helpful to describe group extensions by means of multiple
annotations of (flat) representation. This is available as unpublished
result by the author.

There are examples of factorisations of po-groups that do not
permit simple flat representations. There exist ideas to characterise
the existence of such representations by means of similarity or
generalised neighbourhood relations.

The structures introduced in this article are inspired by ideas
published in \cite{Zickwolff1991}. For simple small categories there
exists a direct mapping between our structures and the ones developed
by Monika Zickwolff.

The class of representations can be equipped with morphisms in a
similar way as the isomorphisms are defined in
\cite{BorchmannDipl,Borchmann:2009} and
Definition~\ref{def:flat-representations-isomorphism}, leading to the
category of representations. Unfolding isomorphic representations
leads to isomorphic groups.

Furthermore there are different applications of representations when
it comes to fundamental systems of tone structures.  The following
examples are based on the notion of a tone system: A set of tones,
equipped with a group of differences, called intervals. For details we
refer to \cite{Schlemmer:2015,schlemmer:2013,Tobias:2010}.

\subsection{Chroma Systems}
Chroma intervals can be considered a factorised interval group. They can be used to describe the cyclic notion of the intervals
of a tone system in music theory. Examples include:

\begin{itemize}
\item The physical frequency or physical frequency space. The pitch of
  a tone can be roughly described as a frequency or the period time of
  a periodic vibration. Both lead to the po-group
  $\OrdGruppe{\reellezahlen}{\cdot}{\multinverses}1{\leq}$, though they
  are dually ordered. For many music theoretical considerations this
  po-groups can be factored by the octave relation, which leads to a
  representation that is isomorphic to
  $\tupel[big]{\orbitfaltigkeit{\Halbord{\poskegel\reellezahlen}{\leq}}{2^{\ganzzahlen}},o,2^{\ganzzahlen}}$
  with $o(x) \definiert 2^{\lfloor log_2x\rfloor}$, which is
  isomorphic via the logarithm to
  $\tupel[big]{\orbitfaltigkeit{\Halbord{\reellezahlen}{\leq}}{\ganzzahlen},o',\ganzzahlen}$
  with $o'(x) \definiert \lfloor x\rfloor$.
\item The Shepard tones\cite{shepard:2346} with their intervals and
  the relation that expresses that one tone $t_1$ is more likely
  perceived lower than another tone $t_2$ form an application of a
  flat representation
  $\tupel{\kategorie S,o',\klmorphismus n',\klmorphismus
    C',\Halbord{\ganzzahlen}{\leq}}$ of
  $\tupel{\orbitfaltigkeit{\reellezahlen}{\ganzzahlen},o',\ganzzahlen}$
  with $\Ob{\kategorie S} = \Ob{\orbitfaltigkeit{\Halbord{\reellezahlen}{\leq}}{\ganzzahlen}}$ and
  $\Mor{\kategorie S}{t_1}{t_2} = \Menge{\paar
    xy∈\Mor{\orbitfaltigkeit{\Halbord{\reellezahlen}{\leq}}{\ganzzahlen}}xy}{0\leq
    t_2-t_1< 1/2}$.
\item The $n$-tone equal temperament ($n$-TET) is often represented as a flat
  representation of $\Halbord{\ganzzahlen}{\leq}$ via the orbit
  category
  $\orbitfaltigkeit{\Halbord{\ganzzahlen}{\leq}}{n\ganzzahlen}$. In
  contrast to other temperaments it can often be modelled as a
  $\ganzzahlen_n$-symmetric flat representation.

  Musical scales can be modelled as a (non-necessary symmetric)
  $m$-TET and a mapping into a $\ganzzahlen_n$-symmetric $n$-TET. This
  mapping is usually not a homomorphism as the group structure is not
  preserved. However the category structure from the scale can be
  reconstructed from the category structure of the enclosing $n$-TET
  tone system.

  For example the diatonic scale is often be considered as a
  $\ganzzahlen_7$-symmetric $7$-TET that is embedded into a
  $\ganzzahlen_{12}$-symmetric $12$-TET.
\item The torus of chromas is a $2$-dimensional chroma system that
  describes the Tonnetz of major thirds and fifths. A natural model
  for it would be an antisymmetrically
  $\ganzzahlen_1\times\ganzzahlen_{12}\times\ganzzahlen_3$-complete
  flat representation
  \[\tupel[big]{\kategorie P,A,\klmorphismus n,\klmorphismus C,
      \Halbord{\ganzzahlen}{\leq}\times
      12\Halbord{\ganzzahlen}{\leq}\times
      2\Halbord{\ganzzahlen}{\leq}}\] where $\kategorie P$ is a
  partial subcategory
  of
  \[\orbitfaltigkeit{\bigl(\Halbord{\ganzzahlen}{\leq}\times\Halbord{\ganzzahlen}{\leq}\times\Halbord{\ganzzahlen}{\leq}\bigr)}{(\ganzzahlen
      \times 12\ganzzahlen \times 3\ganzzahlen)}.\]

\item Musical events have finite durations. When we describe them, we
  have the properties onset, duration and end. This allows to
  represent them as points in a 2-dimensional subspace of a
  3-dimensional space. One possibility is to represent each event by
  onset and duration. The typical operations are translation of the
  onset and stretching a passage by a certain time factor. Both
  together form the group $AGL_1(\reellezahlen)$ of affine operations
  on the real line. As each event could be generated from $(1,1)$ by
  one of the affine operations, we can model the space of musical
  events as $\Halbord[big]{AGL_1(\reellezahlen)}{\leq}$ where $\leq$
  is a lexicographic product order, in which the duration is
  infinitesimal with respect to the onset.

  So the rhythmic space is an example where non-commutative po-groups
  play an important role in music theory. Typical rhythmic patters can
  come from the rhythm itself, or from an external structure like
  measures. The latter can be divided into parts or grouped into
  larger ensembles (typical groups consist of 3, 4, 6 or 8
  measures). These generate subgroups of the translational subgroup of
  $AGL_1(\reellezahlen)$. Rhythmic patterns can be described as subsets in the orbit
  category of $AGL_1(\reellezahlen)$ by one of these subgroups.
\end{itemize}

Chroma systems describe the aspect of being a factor structure of a
tone system. In this sense they cannot distuingish between different
diatonic modes. 

\subsection{Fundamental Systems in the Narrow Sense}\label{sec:Fundamentalsysteme}

Given a partition on a set. A transversal is a set that contains
exactly one element of each class of the partition.  Given a tone
system and a normal subgroup of its interval group that acts interval
preserving on the tone system. In that case we can consider a
transversal of the orbits of the tones together with a sufficient set
of positive intervals as a scale. Such a fundamental system can be
considered as a tile that can be repeated in all directions in order
to generate the tone system.

Examples:
\begin{itemize}
\item Special cases of the diatonic scale as C-major or
  d-minor scales are often considered using partial categories which
  are not symmetric, and where all arrows start in one tone (the
  reference tone). This gives rise to the modern view on the so called
  „church modes“.
\item Leitton in chroma systems
\item Dominant sevenths chords in chroma systems
\item In Jazz music the notion “scale” is linked with the chroma
  supply that may be used in improvisation in a certain tonality. Jazz
  scales are constructed using two complementing orders: the order in
  the physical pitch space, which lead to a chroma system and harmonic
  orderings based on the tonality.
\end  {itemize}

Fundamental systems tend to have a larger variety of intervals than
chroma systems. For example, The diatonic C major mode in $12$-tet has
a minor $7$th as interval in its fundamental system. This interval is
reduced to a negative major second in the corresponding chroma
system.

\subsection{Extended Scales}

Fundamental systems as discussed above contain information about
intervals. In a tone system two tones are identical
if they have the same interval to common third tone. This will not
happen if we reconstruct a tone system from fundamental system as in
section~\ref{sec:Fundamentalsysteme}. But, why two tiles must not
overlap? In the real world overlapping tiles can form a strong
connection when glued together. Actually what we call gluing in
mathematics can be better compared to welding in real
live. Nevertheless, it is easy to define a binary relation $σ$
according to the rules
\begin{align}
  t_1\mathrel{σ} t_2 ⇔ ∃t_3: δ(t_1,t_3) = δ(t_2,t_3).
\end{align}
It is easy to prove that $σ$ is an equivalence relation. Factorisation
by $σ$ leads to the desired result.

In musical discussions sometimes scale steps beyond $7$ occur. While
the octave can still be considered as a representation of the prime or
first scale step, the numbers $9$ and above tones that are different
from their equivalent scale steps below $8$. However, these
representations live in very small parts of musical compositions
(often only one chord). Some short time later the same notes may be
described by different numbers. This suggests that music theorists
consider the tone system to be generated by substructures that cannot
be described by transversals of the orbits. As they still generate the
whole tone system the “fundamental” system must generate overlapping
tiles which are glued together during composition or performance.

\section{Further research topics}\label{sec:furth-rese-topics}
Some questions have already been raised in the previous sections. They
are not repeated, here.

So far flat representations depend on categories that are uniquely
representable by minimal arrows. In case of po-groups this implies
that the order must be Archimedian. However, if the unfolding also
reconstructs infinitesimal elements, this condition can be dropped. As
the order relation of the positive cone is dually isomorphic to the
one of the negative cone via the group inversion antiautomorphism, the
relevant information is already encoded in the vertex category. It is
an open question how this information can be made accessible for
sufficiently simple algorithms.

In a po-group each principal ideal of
an element is dually isomorphic to the principal filter of the same
element and with respect to incoming and outgoing arrows.
The order filters of a po-group form a monoid where the positive cone
is the neutral element. The elementwise product of two order filters
is again an order filter. Together with the subset relation they form
a monoidal category..

Actually the vertex annotation of a representation of an
$\ell$-group can be understood as the defining elements of principal
order filters in the vertex categories. This gives an order relation
on the arrows between two vertices in the orbit category. It can be
used to define flat representations based on valued
categories on the category of order filters in the vertex
category.
%
%This allows us to consider flat representations of arbitrary
%po-groups.
So, flat representations of arbitrary po-groups can be considered.
As long as we have no parallel arrows with the same
annotation in the representation this setting can be easily extended
to the case where the vertex categories are right groupal
categories. If the restriction of parallel arrows with the same
annotation can be lifted this type of annotation would be an analogy
to the annotation as described in \cite{Zickwolff1991}.

The order filters of a cancelable po-monoid together with the set
inclusion morphisms and the block operation form a simple monoidal
category $OF$. Thus, flat representations could be modelled based on
$OF$-valued categories.

\printbibliography

%% % flyspell-default-dictionary: "british"

% Local Variables:
% ispell-local-dictionary: "en_GB"
% TeX-engine: default
% TeX-master: t
% End:

\end{document}